\documentclass[a4paper, british]{amsart}

\usepackage{libertine}
\usepackage[libertine]{newtxmath}

\usepackage{babel}
\usepackage{enumitem}
\usepackage{hyperref}
\usepackage[utf8]{inputenc}
\usepackage{newunicodechar}
\usepackage{mathtools}
\usepackage{varioref}
\usepackage[arrow,curve,matrix]{xy}

\usepackage{colortbl}
\usepackage{graphicx}
\usepackage{tikz}

\usepackage{mathrsfs}

\definecolor{linkred}{rgb}{0.7,0.2,0.2}
\definecolor{linkblue}{rgb}{0,0.2,0.6}

\numberwithin{figure}{section}

\usepackage[hyperpageref]{backref}

\setdescription{labelindent=\parindent, leftmargin=2\parindent}
\setitemize[1]{labelindent=\parindent, leftmargin=2\parindent}
\setenumerate[1]{labelindent=0cm, leftmargin=*, widest=iiii}

\newunicodechar{א}{\ensuremath{\aleph}}
\newunicodechar{α}{\ensuremath{\alpha}}
\newunicodechar{β}{\ensuremath{\beta}}
\newunicodechar{χ}{\ensuremath{\chi}}
\newunicodechar{δ}{\ensuremath{\delta}}
\newunicodechar{ε}{\ensuremath{\varepsilon}}
\newunicodechar{Δ}{\ensuremath{\Delta}}
\newunicodechar{η}{\ensuremath{\eta}}
\newunicodechar{γ}{\ensuremath{\gamma}}
\newunicodechar{Γ}{\ensuremath{\Gamma}}
\newunicodechar{ι}{\ensuremath{\iota}}
\newunicodechar{κ}{\ensuremath{\kappa}}
\newunicodechar{λ}{\ensuremath{\lambda}}
\newunicodechar{Λ}{\ensuremath{\Lambda}}
\newunicodechar{ν}{\ensuremath{\nu}}
\newunicodechar{μ}{\ensuremath{\mu}}
\newunicodechar{ω}{\ensuremath{\omega}}
\newunicodechar{Ω}{\ensuremath{\Omega}}
\newunicodechar{π}{\ensuremath{\pi}}
\newunicodechar{Π}{\ensuremath{\Pi}}
\newunicodechar{φ}{\ensuremath{\phi}}
\newunicodechar{Φ}{\ensuremath{\Phi}}
\newunicodechar{ψ}{\ensuremath{\psi}}
\newunicodechar{Ψ}{\ensuremath{\Psi}}
\newunicodechar{ρ}{\ensuremath{\rho}}
\newunicodechar{σ}{\ensuremath{\sigma}}
\newunicodechar{Σ}{\ensuremath{\Sigma}}
\newunicodechar{τ}{\ensuremath{\tau}}
\newunicodechar{θ}{\ensuremath{\theta}}
\newunicodechar{Θ}{\ensuremath{\Theta}}
\newunicodechar{ξ}{\ensuremath{\xi}}
\newunicodechar{Ξ}{\ensuremath{\Xi}}
\newunicodechar{ζ}{\ensuremath{\zeta}}

\newunicodechar{ℓ}{\ensuremath{\ell}}
\newunicodechar{ï}{\"{\i}}

\newunicodechar{𝔸}{\ensuremath{\bA}}
\newunicodechar{𝔹}{\ensuremath{\bB}}
\newunicodechar{ℂ}{\ensuremath{\bC}}
\newunicodechar{𝔻}{\ensuremath{\bD}}
\newunicodechar{𝔼}{\ensuremath{\bE}}
\newunicodechar{𝔽}{\ensuremath{\bF}}
\newunicodechar{𝔾}{\ensuremath{\bG}}
\newunicodechar{ℕ}{\ensuremath{\bN}}
\newunicodechar{ℙ}{\ensuremath{\bP}}
\newunicodechar{ℚ}{\ensuremath{\bQ}}
\newunicodechar{ℝ}{\ensuremath{\bR}}
\newunicodechar{𝕏}{\ensuremath{\bX}}
\newunicodechar{ℤ}{\ensuremath{\bZ}}
\newunicodechar{𝒜}{\ensuremath{\sA}}
\newunicodechar{ℬ}{\ensuremath{\sB}}
\newunicodechar{𝒞}{\ensuremath{\sC}}
\newunicodechar{𝒟}{\ensuremath{\sD}}
\newunicodechar{ℰ}{\ensuremath{\sE}}
\newunicodechar{ℱ}{\ensuremath{\sF}}
\newunicodechar{𝒢}{\ensuremath{\sG}}
\newunicodechar{ℋ}{\ensuremath{\sH}}
\newunicodechar{𝒥}{\ensuremath{\sJ}}
\newunicodechar{ℒ}{\ensuremath{\sL}}
\newunicodechar{ℳ}{\ensuremath{\sM}}
\newunicodechar{𝒪}{\ensuremath{\sO}}
\newunicodechar{𝒬}{\ensuremath{\sQ}}
\newunicodechar{𝒮}{\ensuremath{\sS}}
\newunicodechar{𝒯}{\ensuremath{\sT}}
\newunicodechar{𝒲}{\ensuremath{\sW}}

\newunicodechar{∂}{\ensuremath{\partial}}
\newunicodechar{∇}{\ensuremath{\nabla}}

\newunicodechar{↺}{\ensuremath{\circlearrowleft}}
\newunicodechar{∞}{\ensuremath{\infty}}
\newunicodechar{⊕}{\ensuremath{\oplus}}
\newunicodechar{⊗}{\ensuremath{\otimes}}
\newunicodechar{•}{\ensuremath{\bullet}}
\newunicodechar{Λ}{\ensuremath{\wedge}}
\newunicodechar{↪}{\ensuremath{\into}}
\newunicodechar{→}{\ensuremath{\to}}
\newunicodechar{↦}{\ensuremath{\mapsto}}
\newunicodechar{⨯}{\ensuremath{\times}}
\newunicodechar{∪}{\ensuremath{\cup}}
\newunicodechar{∩}{\ensuremath{\cap}}
\newunicodechar{⊋}{\ensuremath{\supsetneq}}
\newunicodechar{⊇}{\ensuremath{\supseteq}}
\newunicodechar{⊃}{\ensuremath{\supset}}
\newunicodechar{⊊}{\ensuremath{\subsetneq}}
\newunicodechar{⊆}{\ensuremath{\subseteq}}
\newunicodechar{⊂}{\ensuremath{\subset}}
\newunicodechar{⊄}{\ensuremath{\not \subset}}
\newunicodechar{≥}{\ensuremath{\geq}}
\newunicodechar{≠}{\ensuremath{\neq}}
\newunicodechar{≫}{\ensuremath{\gg}}
\newunicodechar{≪}{\ensuremath{\ll}}

\newunicodechar{≤}{\ensuremath{\leq}}
\newunicodechar{∈}{\ensuremath{\in}}
\newunicodechar{∉}{\ensuremath{\not \in}}
\newunicodechar{∖}{\ensuremath{\setminus}}
\newunicodechar{◦}{\ensuremath{\circ}}
\newunicodechar{°}{\ensuremath{^\circ}}
\newunicodechar{…}{\ifmmode\mathellipsis\else\textellipsis\fi}
\newunicodechar{·}{\ensuremath{\cdot}}
\newunicodechar{⋯}{\ensuremath{\cdots}}
\newunicodechar{∅}{\ensuremath{\emptyset}}
\newunicodechar{⇒}{\ensuremath{\Rightarrow}}

\newunicodechar{⁰}{\ensuremath{^0}}
\newunicodechar{¹}{\ensuremath{^1}}
\newunicodechar{²}{\ensuremath{^2}}
\newunicodechar{³}{\ensuremath{^3}}
\newunicodechar{⁴}{\ensuremath{^4}}
\newunicodechar{⁵}{\ensuremath{^5}}
\newunicodechar{⁶}{\ensuremath{^6}}
\newunicodechar{⁷}{\ensuremath{^7}}
\newunicodechar{⁸}{\ensuremath{^8}}
\newunicodechar{⁹}{\ensuremath{^9}}
\newunicodechar{ⁱ}{\ensuremath{^i}}
\newunicodechar{⁺}{\ensuremath{^+}}

\newunicodechar{⌈}{\ensuremath{\lceil}}
\newunicodechar{⌉}{\ensuremath{\rceil}}
\newunicodechar{⌊}{\ensuremath{\lfloor}}
\newunicodechar{⌋}{\ensuremath{\rfloor}}

\newunicodechar{≅}{\ensuremath{\cong}}
\newunicodechar{⇔}{\ensuremath{\Leftrightarrow}}
\newunicodechar{∃}{\ensuremath{\exists}}
\newunicodechar{±}{\ensuremath{\pm}}

\DeclareFontFamily{OMS}{rsfs}{\skewchar\font'60}
\DeclareFontShape{OMS}{rsfs}{m}{n}{<-5>rsfs5 <5-7>rsfs7 <7->rsfs10 }{}
\DeclareSymbolFont{rsfs}{OMS}{rsfs}{m}{n}
\DeclareSymbolFontAlphabet{\scr}{rsfs}
\DeclareSymbolFontAlphabet{\scr}{rsfs}

\DeclareFontFamily{U}{mathx}{\hyphenchar\font45}
\DeclareFontShape{U}{mathx}{m}{n}{
      <5> <6> <7> <8> <9> <10>
      <10.95> <12> <14.4> <17.28> <20.74> <24.88>
      mathx10
      }{}
\DeclareSymbolFont{mathx}{U}{mathx}{m}{n}
\DeclareFontSubstitution{U}{mathx}{m}{n}
\DeclareMathAccent{\wcheck}{0}{mathx}{"71}

\DeclareMathOperator{\Aut}{Aut}

\DeclareMathOperator{\const}{const}

\DeclareMathOperator{\Hom}{Hom}
\DeclareMathOperator{\Id}{Id}

\DeclareMathOperator{\img}{img}
\DeclareMathOperator{\Pic}{Pic}

\DeclareMathOperator{\Ramification}{Ramification}
\DeclareMathOperator{\red}{red}
\DeclareMathOperator{\reg}{reg}

\DeclareMathOperator{\sing}{sing}

\DeclareMathOperator{\Sym}{Sym}
\DeclareMathOperator{\supp}{supp}
\DeclareMathOperator{\tor}{tor}

\newcommand{\sA}{\scr{A}}
\newcommand{\sB}{\scr{B}}
\newcommand{\sC}{\scr{C}}
\newcommand{\sD}{\scr{D}}
\newcommand{\sE}{\scr{E}}
\newcommand{\sF}{\scr{F}}
\newcommand{\sG}{\scr{G}}
\newcommand{\sH}{\scr{H}}
\newcommand{\sHom}{\scr{H}\negthinspace om}

\newcommand{\sJ}{\scr{J}}

\newcommand{\sL}{\scr{L}}
\newcommand{\sM}{\scr{M}}

\newcommand{\sO}{\scr{O}}

\newcommand{\sQ}{\scr{Q}}

\newcommand{\sS}{\scr{S}}
\newcommand{\sT}{\scr{T}}

\newcommand{\sW}{\scr{W}}

\newcommand{\cC}{\mathcal C}

\newcommand{\bA}{\mathbb{A}}
\newcommand{\bB}{\mathbb{B}}
\newcommand{\bC}{\mathbb{C}}
\newcommand{\bD}{\mathbb{D}}
\newcommand{\bE}{\mathbb{E}}
\newcommand{\bF}{\mathbb{F}}
\newcommand{\bG}{\mathbb{G}}

\newcommand{\bN}{\mathbb{N}}

\newcommand{\bP}{\mathbb{P}}
\newcommand{\bQ}{\mathbb{Q}}
\newcommand{\bR}{\mathbb{R}}

\newcommand{\bX}{\mathbb{X}}

\newcommand{\bZ}{\mathbb{Z}}

\theoremstyle{plain}
\newtheorem{thm}{Theorem}[section]

\newtheorem{conjecture}[thm]{Conjecture}
\newtheorem{cor}[thm]{Corollary}
\newtheorem{defn}[thm]{Definition}
\newtheorem{fact}[thm]{Fact}
\newtheorem{lem}[thm]{Lemma}

\newtheorem{problem}[thm]{Problem}
\newtheorem{prop}[thm]{Proposition}

\theoremstyle{remark}

\newtheorem{c-n-d}[thm]{Claim and Definition}
\newtheorem{consequence}[thm]{Consequence}
\newtheorem{construction}[thm]{Construction}

\newtheorem{example}[thm]{Example}

\newtheorem{notation}[thm]{Notation}
\newtheorem{obs}[thm]{Observation}
\newtheorem{rem}[thm]{Remark}
\newtheorem{question}[thm]{Question}
\newtheorem*{rem-nonumber}{Remark}
\newtheorem{setting}[thm]{Setting}
\newtheorem{warning}[thm]{Warning}

\numberwithin{equation}{thm}

\setlist[enumerate]{label=(\thethm.\arabic*), before={\setcounter{enumi}{\value{equation}}}, after={\setcounter{equation}{\value{enumi}}}}

\newcommand{\into}{\hookrightarrow}

\newcommand{\wtilde}{\widetilde}
\newcommand{\what}{\widehat}

\newcommand\CounterStep{\addtocounter{thm}{1}\setcounter{equation}{0}}

\newcommand{\factor}[2]{\left. \raise 2pt\hbox{$#1$} \right/\hskip -2pt\raise -2pt\hbox{$#2$}}
\newcommand{\Publication}[1]{}

\newcommand{\subversionInfo}{}
\newcommand{\svnid}[1]{}
\newcommand{\approvals}[2][Approval]{}
\renewcommand{\phi}{\varphi}

\usepackage{imakeidx}
\makeindex[title=Index of Concepts]
\indexsetup{firstpagestyle=headings,headers={STEFAN KEBEKUS AND ERWAN ROUSSEAU}{$\cC$-PAIRS AND THEIR MORPHISMS}}

\usepackage{cancel}
\usepackage{tikz-cd}
\tikzset{commutative diagrams/arrow style=tikz}
\author{Stefan Kebekus}
\address{Stefan Kebekus, Mathematisches Institut, Albert-Ludwigs-Universität Freiburg, Ernst-Zermelo-Straße 1, 79104 Freiburg im Breisgau, Germany}
\email{\href{mailto:stefan.kebekus@math.uni-freiburg.de}{stefan.kebekus@math.uni-freiburg.de}}
\urladdr{\url{https://cplx.vm.uni-freiburg.de}}

\author{Erwan Rousseau}
\address{Erwan Rousseau, Univ Brest, CNRS UMR 6205,	Laboratoire de Mathematiques de Bretagne Atlantique\\ F-29200 Brest, France}
\email{\href{mailto:erwan.rousseau@univ-brest.fr}{erwan.rousseau@univ-brest.fr}}
\urladdr{\href{http://eroussea.perso.math.cnrs.fr/}{http://eroussea.perso.math.cnrs.fr}}

\thanks{This work started during the visit of Erwam Rousseau to the Freiburg
Institute for Advanced Studies, supported by the European Unions Horizon 2020
research and innovation program under the Marie Sklodowska-Curie grant agreement
No 75434.  Rousseau thanks the Institute for providing an excellent working
environment.}



\title{$\cC$-pairs and their morphisms}
\date{\today}

\makeatletter
\hypersetup{
  pdfauthor={\authors},
  pdftitle={\@title},
  pdfsubject={\@subjclass},
  pdfkeywords={\@keywords},
  pdfstartview={Fit},
  pdfpagelayout={TwoColumnRight},
  pdfpagemode={UseOutlines},
  bookmarks,
  colorlinks,
  linkcolor=linkblue,
  citecolor=linkred,
  urlcolor=linkred
}
\makeatother

\DeclareMathOperator{\alb}{alb}
\DeclareMathOperator{\Alb}{Alb}
\DeclareMathOperator{\Branch}{Branch}
\DeclareMathOperator{\diff}{d}
\DeclareMathOperator{\Div}{Div}

\DeclareMathOperator{\Gal}{Galois}
\DeclareMathOperator{\lu}{lu}
\DeclareMathOperator{\mult}{mult}

\DeclareMathOperator{\orb}{orb}

\DeclareMathOperator{\trace}{trace}

\newcommand{\xrightarrowdbl}[2][]{%
  \xrightarrow[#1]{#2}\mathrel{\mkern-14mu}\rightarrow
}
\theoremstyle{plain}

\theoremstyle{remark}

\newtheorem{conj}[thm]{Conjecture}

\newtheorem{reminder}[thm]{Reminder}

\begin{document}

\approvals[Approval for Abstract]{Erwan & yes \\ Stefan & yes}
\begin{abstract}
\selectlanguage{british}

This paper surveys Campana's theory of $\cC$-pairs (or ``geometric orbifolds'')
in the complex-analytic setting, to serve as a reference for future work.
Written with a view towards applications in hyperbolicity, rational points, and
entire curves, it introduces the fundamental definitions of $\cC$-pair-theory
systematically.  In particular, it establishes an appropriate notion of
``morphism'', which agrees with notions from the literature in the smooth case,
but it is better behaved in the singular setting and has functorial properties
that relate it to minimal model theory.

% !TEX root = orbiAlb1
\end{abstract}

\maketitle
\tableofcontents
\clearpage

\phantomsection\addcontentsline{toc}{part}{Introduction}

%
% Do not edit the following line.  The text is automatically updated by
% subversion.
%
\svnid{$Id: 01-intro.tex 851 2024-07-15 08:07:55Z rousseau $}
\selectlanguage{british}

\section{Introduction}
\subversionInfo

\subsection{$\cC$-pairs}
\approvals{Erwan & yes \\ Stefan & yes}

Following earlier work of Miyaoka and others, Campana introduced ``$\cC$-pairs''
or ``geometric orbifolds'' to complex-analytic geometry in his influential paper
\cite{Cam04}.  In simplest terms, $\cC$-pairs $(X, D)$ consist of a normal
variety $X$ and a $ℚ$-divisor with ``standard coefficients'',
\[
  D = \sum_i \frac{m_i-1}{m_i} · D_i, \quad\text{where all } m_i ∈ ℕ^{≥ 2} ∪ \{∞\}.
\]
Conceptually, $\cC$-pairs and their associated differentials interpolate between
two ``extreme'' geometries.
\begin{itemize}
  \item The geometry of $X$, which is governed by the sheaves $Ω^p_X$ of Kähler
  differentials.

  \item The geometry of $X ∖ D$, which is governed by the sheaves $Ω^p_X(\log
  D)$ of logarithmic differentials.
\end{itemize}
$\cC$-pairs appear in higher-dimensional birational geometry, where geometers
simplify line bundles by passing to branched covers and use boundary divisors
$D$ to keep track of ramification orders.  They appear in the study of entire
curves and rational points over function fields, where the boundary divisors $D$
encode tangency conditions that are different from (and in some settings more
natural than) the conditions imposed by root stacks.  In moduli theory,
geometers study fibrations and use boundary divisors $D$ as bookkeeping devices
for multiplicities of fibres.

\subsubsection*{Special pairs and pairs of general type}
\approvals{Erwan & yes \\ Stefan & yes}

Campana has seen that $\cC$-pairs are the natural objects that generalize the
dichotomy between rational/elliptic and higher-genus curves to higher
dimensions.  He attaches to every compact Kähler manifold $X$ a natural ``core
fibration'' that decomposes $X$ into the \emph{fibre space base}, which is a
$\cC$-pair of general type, and the \emph{fibres}, which are ``special'' in the
sense that they do not admit dominant morphisms to $\cC$-pairs of general type.
$\cC$-pairs of general type and special $\cC$-pairs differ in almost every
aspect of topology or arithmetic/metric/analytic geometry.  One aspect is
illustrated by the following conjecture.

\begin{conj}[Campana]\label{conj:1-1}
  ---
  \begin{itemize}
  \item Let $X$ be a projective manifold defined over a number field $k$.  Then,
    $X$ is special if and only if its rational points are potentially
    dense\footnote{potentially dense = there exists a finite field extension $k
    ⊆ k'$ such that $k'$-rational points are Zariski dense in $X$.}.

  \item Let $X$ be a complex projective manifold.  Then, $X$ is special if and
    only if $X$ admits a Zariski dense entire curve.
  \end{itemize}
\end{conj}

Conjecture~\ref{conj:1-1} is a vast generalization of Lang's famous conjectures.
It motivates the present paper and follow-up work on the ``Albanese of a
$\cC$-pair'' that is currently in preparation.

\subsection{Content of the paper}
\approvals{Erwan & yes \\ Stefan & yes}

This paper surveys Campana's theory of $\cC$-pairs (or ``geometric orbifolds'')
in the complex-analytic setting.  Aiming to serve as a reference for future
work, it introduces the fundamental definitions of $\cC$-pairs-theory
systematically.

\subsubsection*{Part I: Adapted tensors and differentials}
\approvals{Erwan & yes \\ Stefan & yes}

Following ideas of \cite{Miy08} and \cite{CP15}, the first part of the present
paper equips $\cC$-pairs with a (co)tangent bundle, with sheaves of
differential forms, and sheaves of higher-order tensors.  These objects have
been used in \cite{Miy08} to establish novel Chern class inequalities, and in
\cite{CP15} to construct natural foliations on base spaces of families of
canonically polarized manifolds.

Conceptually, the ``sheaves $Ω^p_{(X,D)}$ of $\cC$-pair-differentials''
should interpolate between the sheaves of Kähler-- and logarithmic differentials,
\[
  Ω^p_X ⊆ Ω^p_{(X,D)} ⊆ Ω^p_X(\log D),
\]
in the sense that sections of $Ω^p_{(X,D)}$ should be differentials on $X$,
with logarithmic poles of fractional pole order $\frac{m_i-1}{m_i}$ along the
component $D_i$ of $D$.  While ``differentials with fractional pole order'' do
not exist on $X$ in any meaningful way, they can be defined on suitable covers
of the space $X$.  Sections~\ref{sec:3} and \ref{sec:4} make these vague
concepts precise.  In order to obtain a theory with good universal properties,
these sections define ``sheaves of adapted reflexive differentials'' on
arbitrary covers of arbitrary $\cC$-pairs, including very singular ones.

Section~\ref{sec:5} shows that adapted reflexive differentials over $\cC$-pairs
with mild singularities have optimal pull-back properties, similar to the
pull-back properties of reflexive differentials on spaces with rational
singularities, as established in \cite{GKKP11, MR3084424, KS18}.  These results
relate $\cC$-pairs to minimal model theory and will be instrumental when we
define morphisms of $\cC$-pairs in the second part of this paper.

With all preparations at hand, Section~\ref{sec:6} discusses $\cC$-pair
analogues of several classic invariants, including the ``irregularity of
$\cC$-pair'' and a notion of ``$\cC$-Kodaira-Iitaka dimension for rank-one
sheaves of adapted tensors''.  The section extends the classical vanishing
theorem of Bogomolov-Sommese to $\cC$-pairs and defines the notion of ``special
$\cC$-pairs''.

\subsubsection*{Part II: Morphisms of $\cC$-pairs}
\approvals{Erwan & yes \\ Stefan & yes}

\CounterStep{}The second part of this paper defines and discusses ``morphisms of
$\cC$-pairs''.  The basic idea is simple: If $(X,D_X)$ and $(Y,D_Y)$ are
$\cC$-pairs and if $φ : X → Y$ is a morphism of analytic varieties, call $φ$ a
morphism of $\cC$-pairs if adapted reflexive differentials on $Y$ pull back to
adapted reflexive differentials on $X$.  Sections~\ref{sec:7} and \ref{sec:8}
make this idea precise.  Section~\ref{sec:9} establishes criteria to guarantee
that a given morphism of varieties is a morphism of $\cC$-pairs.  To illustrate
these concepts and highlight some features of our definition,
Section~\ref{sec:10} discusses several (non-)examples.  Sections~\ref{sec:11},
\ref{sec:12} and \ref{sec:13} establish functoriality properties, the existence
of categorical quotients, and relate morphisms of $\cC$-pairs to basic notions
of minimal model theory.

Section~\ref{sec:14} compares our notion ``morphism'' with other notions that
have appeared in the literature.  Among those, Campana's definition of an
\emph{orbifold morphisms} is perhaps the most prominent: if $(X, D_X)$ and $(Y,
D_Y)$ are $\cC$-pairs and if $φ : X → Y$ is holomorphic, then $φ$ is called
orbifold morphism if a purely numerical criterion holds, relating the
coefficients of boundary divisors $D_X$ and $D_Y$ with multiplicities of
pull-back divisors coming from $Y$.  To ensure that all quantities are
well-defined, Campana defines orbifold morphisms only in settings where
\begin{enumerate}
  \item\label{il:1-2-1} the space $Y$ is $ℚ$-factorial, and

  \item\label{il:1-2-2} the morphism $φ$ does not take its image inside the
  support of $D_Y$.
\end{enumerate}
For morphism between smooth pairs that satisfy \ref{il:1-2-2}, Campana's
definition coincides with ours, and is generally easier to check.  For singular
pairs, Campana's definition and ours do not coincide in general, even in cases
where all pairs are uniformizable.  The difference will be of great importance
for the applications to Conjecture~\ref{conj:1-1}, where singularities naturally
appear through minimal model theory and purely numerical data might not always
suffice to capture the conceptually right geometric picture.

Section~\ref{sec:15} gathers several open questions and mentions problems for
future research.

\subsection{Outlook}
\approvals{Erwan & yes \\ Stefan & yes}

This publication is the first in a series; two follow-up papers will likely
appear later this year.  Building on notions and fundamental results obtained
here, a second paper, \cite{orbialb2}, will introduce ``$\cC$-semitoric
varieties'' as analogues of Abelian varieties and (semi)tori used in the classic
complex geometry.  It follows Serre by defining the Albanese of a $\cC$-pair as
the universal map to a $\cC$-semitoric variety and shows that the Albanese
exists in relevant cases.

Aiming to establish hyperbolicity properties of $\cC$-pairs with large
irregularity, a third paper, \cite{orbialb3}, will extend the fundamental
theorem of Bloch-Ochiai to the context of $\cC$-pairs.  Building on works of
Kawamata, Ueno, and Noguchi, it recasts parabolic Nevanlinna theory as a
``Nevanlinna theory for $\cC$-pairs''.  The authors hope that this approach
might be of independent interest

\subsection{Acknowledgements}
\approvals{Erwan & yes \\ Stefan & yes}

The authors would like to thank Oliver Bräunling, Lukas Braun, Michel Brion,
Fabrizio Catanese, Johan Commelin, Andreas Demleitner, Zsolt Patakfalvi, and
Wolfgang Soergel for long discussions.  Pedro Núñez pointed us to several
mistakes in early versions of the paper.  Jörg Winkelmann patiently answered our
questions throughout the work on this project.

The work on this paper was carried out in part while Stefan Kebekus visited
Zsolt Patakfalvi at the EPFL.  He would like to thank Patakfalvi and his group
for hospitality and for many discussions.

% !TEX root = orbiAlb1
%
% Do not edit the following line.  The text is automatically updated by
% subversion.
%
\svnid{$Id: 02-notation.tex 852 2024-07-15 08:20:56Z kebekus $}
\selectlanguage{british}

\section{Notation and standard facts}
\subversionInfo
\approvals{Erwan & yes \\ Stefan & yes}

This paper works in the category of complex analytic spaces and is written with
a view towards a future \emph{$\cC$-Nevanlinna theory}.  All the material in
this paper will however work in the complex-algebraic setting, typically with
simpler definitions and proofs.  We expect that large parts of this paper will
work for algebraic varieties over perfect fields of arbitrary characteristic and
refer the reader to \cite{KPS22}, which discusses $\cC$-curves over global
function fields.

\subsection{Global conventions}
\approvals{Erwan & yes \\ Stefan & yes}

With very few exceptions, we follow the notation of the standard reference texts
\cite{CAS, DemaillyBook, MR3156076}.  An \emph{analytic variety}\index{analytic
variety} is a reduced, irreducible complex space.  For clarity, we refer to
holomorphic maps between analytic varieties as \emph{morphisms} and reserve the
word \emph{map} for meromorphic mappings.

\begin{defn}[Big and small sets]
  Let $X$ be an analytic variety.  An analytic subset $A ⊊ X$ is called
  \emph{small}\index{small set} if it has codimension two or more.  An open set
  $U ⊆ X$ is called \emph{big}\index{big set} if $X∖U$ is analytic and small.
\end{defn}

\subsection{Sheaves}
\approvals{Erwan & yes \\ Stefan & yes}
\label{sec:2-2}

The following notation will be used to speak about ``meromorphic differentials''
on normal spaces.  We refer the reader to
\cite[\href{https://stacks.math.columbia.edu/tag/01X1}{Tag
01X1}]{stacks-project} for a further discussion of meromorphic sections of
coherent sheaves.

\begin{notation}[Meromorphic functions]
  If $X$ is any normal analytic variety, write $ℳ_X$ for the sheaf of
  meromorphic functions.  If $ℰ$ is any reflexive, coherent sheaf on $X$, we
  call $ℰ⊗ℳ_X$ the \emph{sheaf of meromorphic sections in $ℰ$}\index{sheaf of
  meromorphic sections}\index{meromorphic section!of a reflexive sheaf}.
\end{notation}

\begin{notation}[Meromorphic section with prescribed location of poles]
  Let $X$ be a normal analytic variety and let $D ∈ \Div(X)$ be a Weil divisor.
  If $ℰ$ is any coherent, locally free sheaf on $X$, write $ℰ(*D) ⊂ ℰ⊗ℳ_X$ for
  the sheaf of meromorphic sections in $ℰ$ that are allowed to have poles along
  the support of $D$\index{meromorphic section!with prescribed location of
  poles}.
\end{notation}

\begin{notation}[Differentials with logarithmic poles]
  Let $X$ be a normal analytic variety and let $D ∈ ℚ\Div(X)$ be an effective
  Weil $ℚ$-divisor on $X$, with nc support.  For brevity, we will often write
  $Ω^p_X(\log D)$ to denote the sheaves of Kähler differentials with logarithmic
  poles along $\supp D$.  We will use the correct, but longer forms $Ω^p_X(\log
  \supp D)$ or $Ω^p_X(\log D_{\red})$ only if confusion is likely to arise.
\end{notation}

This paper frequently works with reflexive sheaves on normal analytic varieties.
The following notation will be used.

\begin{notation}[Reflexive hull and reflexive pull-back]\label{not:2-5}%
  If $ℰ$ is any coherent sheaf on a normal analytic variety $X$, we will
  frequently write $ℰ^{\vee\vee}$ for its double dual, which is a coherent,
  reflexive sheaf.  We refer to $ℰ^{\vee\vee}$ as the \emph{reflexive
  hull}\index{reflexive!hull} of $ℰ$.  If $n ∈ ℕ$ is any number, define the
  \emph{reflexive tensor power}\index{reflexive!tensor power} and
  \emph{reflexive symmetric power}\index{reflexive!symmetric power} of $ℰ$ as
  \[
    ℰ^{[⊗ n]} := \bigl( ℰ^{⊗ n} \bigr)^{\vee\vee} %
    \quad\text{and}\quad %
    \Sym^{[n]} ℰ := \bigl( \Sym^n ℰ \bigr)^{\vee\vee}.
  \]
  Given any Weil divisor $D ∈ \Div(X)$, define the \emph{reflexive twist of $ℰ$
  by $D$}\index{reflexive!twist} as
  \[
    ℰ(D) := \bigl( ℰ⊗𝒪_X(D) \bigr)^{\vee\vee}.
  \]
  If $φ : Y → X$ is any morphism from a normal analytic variety $Y$, we will
  often write $φ^{[*]} ℰ := (φ^* ℰ)^{\vee\vee}$ and refer to this sheaf as the
  \emph{reflexive pull-back}\index{reflexive!pull-back} of $ℰ$.
\end{notation}

\begin{notation}[Reflexive differentials]\label{not:2-6} %
  Let $X$ be a normal analytic variety and write $Ω^{[p]}_X := \bigl(Ω^p_X
  \bigr)^{\vee\vee}$.  Given a Weil $ℚ$-divisor $D ∈ ℚ\Div(X)$, write
  $Ω^{[p]}_X(\log D) := \bigl(Ω^p_X(\log D) \bigr)^{\vee\vee}$.  By minor abuse
  of language, we refer to sections in these sheaves as \emph{reflexive
  (logarithmic)
  differentials}\index{reflexive!differential}\index{reflexive!logarithmic
  differential}.
\end{notation}

\begin{rem}[Reflexive differentials and prolongations]
  In the setting of Notation~\ref{not:2-6}, if $X⁺ ⊂ X$ is the maximal open set
  where $(X, \supp D)$ is nc and if $ι: X⁺ → X$ denotes the inclusion map, then
  there exists a canonical isomorphism $Ω^{[p]}_X(\log D) ≅ ι_* Ω^p_{X⁺}(\log
  D)$.  We refer the reader to \cite[Thm.~1 and Prop.~7]{MR0212214} for details.
\end{rem}

\begin{reminder}[Reflexive sheaves of rank one, base loci and meromorphic maps]\label{remi:2-8}%
  If $ℱ$ is any rank-one, coherent reflexive sheaf on a normal analytic variety
  $X$, then $ℱ|_{X_{\reg}}$ is locally free.  If $V ⊆ H⁰\bigl( X,\, ℱ \bigr)$ is
  non-trivial and finite-dimensional, then the associated holomorphic
  \[
    X_{\reg} ∖ (\text{Base locus }V) → ℙ(V^\vee)
  \]
  extends to a meromorphic mapping $X \dasharrow ℙ(V^\vee)$.  For a proof,
  recall a fundamental result of Rossi, \cite[Thm.~3.5]{Rossi68} or
  \cite[Thm.~1.1]{GR70}: there exists a proper modification $π : \wtilde{X}
  \twoheadrightarrow X$ such that $π^* ℱ/\tor$ is locally free, hence
  invertible.
\end{reminder}

\subsection{Weil $ℚ$-divisors and pairs}
\label{sec:2-3}
\approvals{Erwan & yes \\ Stefan & yes}

Much of our discussion is centred around Weil $ℚ$-divisors on normal analytic
varieties.  The following standard language will be used.

\begin{notation}[$ℚ$-Cartier and locally $ℚ$-Cartier divisors]\label{not:2-9}%
  Let $X$ be a normal analytic variety and let $D ∈ ℚ\Div(X)$ be a Weil
  $ℚ$-divisor on $X$.
  \begin{enumerate}
    \item The divisor $D$ is called \emph{$ℚ$-Cartier}\index{Q-Cartier
    divisor@$ℚ$-Cartier divisor} if there exists a positive number $m ∈ ℕ⁺$ such
    that $m·D$ is integral and Cartier.

    \item The divisor $D$ is called \emph{locally $ℚ$-Cartier}\index{locally
    $ℚ$-Cartier divisor} if every point $x ∈ X$ has an open neighbourhood $U =
    U(x) ⊆ X$ such that $D|_U$ is $ℚ$-Cartier on $U$.
  \end{enumerate}
\end{notation}

\begin{notation}[$ℚ$-factorial and locally $ℚ$-Cartier varieties]\label{not:2-10}%
  Let $X$ be a normal analytic variety.
  \begin{enumerate}
    \item The variety $X$ is called \emph{$ℚ$-factorial}\index{Q-factorial
    variety@$ℚ$-factorial analytic variety} if every Weil $ℚ$-divisor on $X$ is
    $ℚ$-Cartier.

    \item The variety $X$ is called \emph{locally $ℚ$-factorial}\index{locally
    $ℚ$-factorial analytic variety} if there exists a basis of topology,
    $(U_α)_{α ∈ A}$, where all $U_•$ are $ℚ$-factorial analytic varieties.
  \end{enumerate}
\end{notation}

\begin{rem}[Local and global properties in analytic geometry]
  In contrast to algebraic geometry, the local properties in
  Notation~\ref{not:2-9} and \ref{not:2-10} do not imply the global ones.  It is
  not hard to construct a non-compact, normal analytic surface $S$ and a prime
  Weil divisor $D ∈ \Div(S)$ with the following properties.
  \begin{enumerate}
    \item The singular locus $S_{\sing} = \{ s_1, s_2, …\}$ is countable
    infinite, discrete, and contained in the support of $D$.  In particular, $S$
    has only isolated singularities.

    \item If $n ∈ ℕ⁺$ is any number, then $n·D$ is Cartier on the set $S_{\reg}
    ∪ \{s_n\}$, which is an open neighbourhood of the singular point $s_n$.  If
    $m ∈ ℕ⁺$ is larger than $n$, then $n·D$ is \emph{not} Cartier in any
    neighbourhood of the singular point $s_m$.
  \end{enumerate}
  In these examples, the divisor $D$ is locally $ℚ$-Cartier but not $ℚ$-Cartier.
  Similar examples exist for factoriality.
\end{rem}

\begin{notation}[Operations on Weil $ℚ$-divisors]
  Let $X$ be a normal analytic variety and let $D ∈ ℚ\Div(X)$ be a Weil
  $ℚ$-divisor on $X$.  Following standard notation, we denote the
  \emph{round-down}\index{round-down of divisor}, \emph{round-up}\index{round-up
  of divisor} and the \emph{fractional part}\index{fractional part of divisor}
  of $D$ by $⌊D⌋$, $⌈D⌉$, and $\{D\} := D - ⌊D⌋$, respectively.  We call $D$
  \emph{reduced}\index{reduced divisor} if all its coefficients are one, and
  write $D_{\red}$ for the reduced Weil divisor obtained by setting all
  non-vanishing coefficients of $D$ to one.
\end{notation}

\begin{defn}[Pair, log pair, nc pair]\label{def:2-13}%
  A \emph{pair}\index{pair} is a tuple $(X,D)$ consisting of a normal analytic
  variety $X$ and a Weil $ℚ$-Divisor $D$ on $X$ with coefficients in the
  interval $(0 … 1] ∩ ℚ$.
  \begin{itemize}
    \item Call $(X,D)$ a \emph{log pair}\index{logarithmic pair} if $D$ is
    reduced.

    \item Call $(X,D)$ a \emph{nc pair}\index{nc pair} if $X$ is smooth and $D$
    has normal crossing support.
  \end{itemize}
\end{defn}

\begin{notation}[Open part of a pair]\label{not:2-14}%
  Let $(X,D)$ be a pair.  The \emph{open part}\index{open part of a pair} is the
  pair $(X°,D°)$, where $X° := X ∖ \supp ⌊D⌋$ and $D° := D ∩ X°$.
\end{notation}

\begin{obs}[Open sets where pair is nc]\label{obs:2-15}%
  Let $(X,D)$ be a pair.  Since the property ``nc pair'' is local and open,
  there exists a maximal open subset $X⁺ ⊆ X$ where the pair is nc.  Observe
  that this subset is big.
\end{obs}

\begin{defn}[Gorenstein conditions]\label{def:2-16}
  ---
  \begin{itemize}
    \item Call a pair $(X,D)$ \emph{Gorenstein}\index{Gorenstein pair} if the
    $ℚ$-divisor $D$ is integral and the sheaf
    \[
      \bigl(ω_X ⊗ 𝒪_X(D) \bigr)^{\vee\vee}
    \]
    is locally free.

    \item Call a pair $(X,D)$ \emph{$ℚ$-Gorenstein}\index{Q-Gorenstein
    pair@$ℚ$-Gorenstein pair} if there exists a number $m ∈ ℕ⁺$ such that the
    divisor $m·D$ is integral and the sheaf
    \[
      \bigl( ω^{⊗m}_X ⊗ 𝒪_X(m·D) \bigr)^{\vee\vee}
    \]
    is locally free.

    \item Call a pair $(X,D)$ \emph{locally $ℚ$-Gorenstein}\index{locally
    $ℚ$-Gorenstein pair} if there exists an open covering $X = ∪_i X_i$ such
    that the pairs $(X_i, D∩ X_i)$ are $ℚ$-Gorenstein.
  \end{itemize}
\end{defn}

\begin{rem}[Gorenstein conditions and the canonical divisor]
  For pairs $(X,D)$ where a canonical divisor $K_X$ exists\footnote{A canonical
  divisor exists if the sheaf $ω_X$ has a meromorphic section.  This is the case
  if the normal analytic variety $X$ is Stein or quasi-projective.  A compact
  complex manifold of algebraic dimension zero need not have a canonical
  divisor.}, the Gorenstein conditions of Definition~\ref{def:2-16} can be
  reformulated as asking that $K_X+D$ is Cartier, $ℚ$-Cartier or locally
  $ℚ$-Cartier, respectively.
\end{rem}

\subsection{Covers and $q$-morphisms}
\approvals{Erwan & yes \\ Stefan & yes}

Large parts of this paper are concerned with quasi-finite morphism between
normal varieties of equal dimension.  For brevity, the following notation will
be used.

\begin{defn}[$q$-morphisms, relative automorphisms]
  Quasi-finite morphisms be\-tween normal analytic varieties of equal dimension
  are called \emph{$q$-morphisms}\index{q-morphism@$q$-morphism}.  If $γ :
  \what{X} → X$ is a $q$-morphism, consider the \emph{relative automorphism
  group}\index{relative automorphism group}
  \[
    \Aut_{𝒪}\bigl(\what{X}/X\bigr) := \bigl\{ g ∈ \Aut_{𝒪}(X) \::\: γ◦ g = γ \bigr\}.
  \]
\end{defn}

\begin{reminder}[Openness]\label{remi:2-19}%
  Recall that $q$-morphisms are open, \cite[Sect.~3.2]{CAS}.  The relative
  automorphism group of a $q$-morphism is finite and acts holomorphically.
\end{reminder}

\begin{reminder}[Pull-back of Weil $ℚ$-divisors]\label{remi:2-20}%
  If $γ : \what{X} → X$ is a $q$-morphism, there exists a well-defined pull-back
  morphism for Weil divisors,
  \[
    γ^* : \Div(X) → \Div(\what{X}),
  \]
  that agrees over $X_{\reg}$ with the standard pull-back of Cartier divisors,
  respects linear equivalence and therefore induces a morphism between Weil
  divisor class groups.
\end{reminder}

\begin{defn}[Covers, Galois covers]\label{def:2-21}%
  Finite morphisms between normal analytic varieties of equal dimension are
  called \emph{covers}\index{cover}.  A cover $γ : X → Y$ is called
  \emph{Galois}\index{Galois cover} if it is isomorphic to the quotient morphism
  \[
    q : X → \factor{X}{\Aut_{𝒪}\bigl(X/Y\bigr)}.
  \]
\end{defn}

Covers are necessarily surjective.  We do \emph{not} require that Galois covers
are locally biholomorphic.  We refer the reader to \cite[Thm.~4]{MR0084174} for
quotients of analytic varieties.

\begin{notation}[Branch and ramification divisor]\label{not:2-22}%
  \index{branch divisor}\index{ramification divisor}If $γ : X \twoheadrightarrow
  Y$ is a cover, write $\Branch(γ) ∈ \Div(Y)$ for the reduced Weil divisor on
  $Y$ whose support equals the codimension-one part of the branch locus of $γ$.
  Analogously, write $\Ramification(γ) ∈ \Div(X)$ for the reduced Weil divisor
  on $X$ whose support equals the codimension-one part of the ramification locus
  of $γ$.
\end{notation}

The following lemma allows us to compare logarithmic differentials on the domain
and target of a cover.  The proof follows most easily from a local computation.
We refer the reader to \cite[Sect.~2.C]{GKK08} for details and for related
results.

\begin{lem}[Criterion for log poles]\label{lem:2-23}%
  Let $γ : X \twoheadrightarrow Y$ be a cover, let $D$ be a reduced Weil divisor
  on $Y$, and let $σ$ be a meromorphic differential form with poles along $D$.
  Assume that the pairs $(Y, D)$ and $(X, γ^*D)$ are both nc.  Then, the
  following statements are equivalent.
  \begin{enumerate}
    \item The form $σ ∈ H⁰\bigl( Y,\, Ω^p_Y(*D) \bigr)$ has logarithmic poles.

    \item The pull-back form $(\diff γ)σ ∈ H⁰\bigl( X,\, Ω^p_X(*γ^*D) \bigr)$
    has logarithmic poles.  \qed
  \end{enumerate}
\end{lem}

\subsection{\texorpdfstring{$\cC$}{C}-pairs and adapted morphisms}
\label{sec:2-5}
\approvals{Erwan & yes \\ Stefan & yes}

The key notion of the present paper is the $\cC$-pair, also called
\emph{\foreignlanguage{french}{orbifolde géométrique}}\index{orbifolde
géométrique} by Campana \cite[Def.~2.1]{MR2831280} or \emph{Campana
constellation}\index{Campana constellation} by Abramovich
\cite[Lecture~2]{Abramovich}.  We recall the definition for the reader's
convenience.

\begin{defn}[\protect{$\cC$-pairs, see \cite[Sect.~2.1]{MR2831280}}]\label{def:2-24}%
  A $\cC$-pair\index{C-pair@$\cC$-pair} is a pair $(X, D)$ where the Weil
  $ℚ$-divisor $D$ is of the form
  \[
    D = \sum_i \frac{m_i-1}{m_i}·D_i,
  \]
  with $m_i ∈ ℕ^{≥ 2} ∪ \{∞\}$ and $\frac{∞-1}{∞} = 1$.
\end{defn}

\begin{notation}[$\cC$-multiplicity]
  In the setting of Definition~\ref{def:2-24}, if $H ⊂ X$ is any prime divisor,
  define the \emph{$\cC$-multiplicity of $D$ along
  $H$}\index{C-multiplicity@$\cC$-multiplicity} as
  \begin{align*}
    \mult_{\cC, H} D & := \left\{
    \begin{matrix}
      m_i & \text{if there exists an index $i$ with $H = D_i$} \\
      1 & \text{otherwise.}
    \end{matrix}
    \right.  \\
    \intertext{It will sometimes be convenient to consider the following Weil $ℚ$-divisor}
    D_{\orb} & := \sum_{i \:|\: m_i < ∞} \frac{1}{m_i}·D_i ∈ ℚ\Div(X).
  \end{align*}
\end{notation}

$\cC$-pairs come with a class of distinguished morphisms, called \emph{adapted
morphisms}.

\begin{defn}[\protect{Adapted morphism, compare \cite[Sect.~5.1]{MR3949026}}]\label{def:2-26}%
  Consider a $\cC$-pair $(X, D)$.  A $q$-morphism $γ : \what{X} → X$ is called
  \emph{adapted for $(X, D)$}\index{adapted!q-morphism@$q$-morphism} if $γ^*
  D_{\orb}$ is integral.  The morphism $γ$ is called \emph{strongly adapted for
  $(X, D)$}\index{strongly adapted morphism} if $γ^* D_{\orb}$ is reduced.
\end{defn}

The word ``adapted'' is not used uniformly in the literature.  Morphisms that we
call ``adapted'' are called ``subadapted'' in \cite{MR2860268} and other papers.

\begin{obs}[Composition]
  In the setup of Definition~\ref{def:2-26}, let
  \[
    \begin{tikzcd}
      \what{X}_1 \ar[r, "γ_1"] & \what{X}_2 \ar[r, "γ_2"] & X
    \end{tikzcd}
  \]
  be a sequence of $q$-morphisms.  If $γ_2$ is adapted for $(X,D)$, then so is
  $γ_2◦γ_1$.  \qed
\end{obs}

\subsubsection{Uniformization}
\approvals{Erwan & yes \\ Stefan & yes}

The following Sections~\ref{sec:3} and \ref{sec:4} introduce adapted (reflexive)
differentials, a class of differential forms that exist on the domain of a
$q$-morphism.  Uniformizations are adapted morphisms where these differentials
take a particularly simple form.  Example~\vref{ex:4-6} compares adapted
differentials with Kähler differentials and makes this vague statement precise.

\begin{defn}[Uniformization of a $\cC$-pair]\label{def:2-28}%
  Let $(X,D)$ be a $\cC$-pair.  A \emph{uniformization of
  $(X,D)$}\index{uniformization of a $\cC$-pair} is a strongly adapted cover $u
  : X_u \twoheadrightarrow X$ where $\Branch(u) ⊆ \supp D$ and $(X_u, \supp u^*
  ⌊D⌋)$ has normal crossings.
\end{defn}

\begin{example}[Neil's parabola and three lines through a common point]\label{ex:2-29}%
  The $\cC$-pairs
  \[
    \textstyle
    \bigl(ℂ²,\; \frac{1}{2}·\{x²=y³\} \bigr)
    \quad\text{and}\quad
    \bigl(ℂ²,\; \frac{2}{3}·\{x=0\} + \frac{2}{3}·\{y=0\} +\frac{1}{2}·\{x=y\} \bigr)
  \]
  are uniformizable.  The proof is part of the classification of orbifaces with
  smooth base, \cite{MR2306159}.  In each example, one computes that the
  orbifold fundamental group is finite.  The associated orbifold-universal
  branched covering is a locally simply connected, normal surface and hence
  smooth by Mumford's classic result \cite[Thm.~on p.~229]{MR0153682}.  We refer
  the reader to \cite[Sect.~3 and Thm.~3.2]{MR2306159} for precise statements, a
  full classification with more examples, and details.
\end{example}

\begin{rem}[Branch divisor and branch locus]
  Recall from Notation~\ref{not:2-22} that $\Branch(u)$ denotes the
  codimension-one part of the branch locus for the morphism $u$.  Uniformizations
  may branch over a codimension-two set that is not contained in $\supp D$.
\end{rem}

\begin{rem}[Uniformizations and $q$-morphisms]
  Let $(X,D)$ be a $\cC$-pair and let $γ : \what{X} → X$ be any $q$-morphism.
  Then, there exists a maximal open subset $X⁺ ⊆ X$ over which $γ$ is a
  uniformization.  The set $X⁺ ⊆ X$ is Zariski open, but not necessarily big.
\end{rem}

\begin{defn}[Uniformizable $\cC$-pairs]\label{def:2-32}
  A $\cC$-pair $(X,D)$ is \emph{uniformizable}\index{uniformizable $\cC$-pair}
  if there exists a uniformization.  It is \emph{locally
  uniformizable}\index{locally uniformizable $\cC$-pair} if every point of $X$
  has a uniformizable neighbourhood.
\end{defn}

\begin{rem}
  If a $\cC$-pair is nc, then it is locally uniformizable.  If $(X,D)$ is any
  $\cC$-pair, then there exists a maximal open subset $X^{\lu} ⊆ X$ over which
  $(X,D)$ is locally uniformizable.  The set $X^{\lu} ⊆ X$ is Zariski open and
  big.
\end{rem}

\begin{rem}[$ℚ$-factoriality of (locally) uniformizable pairs]\label{rem:2-34}
  Uniformizable pairs are $ℚ$-factorial.  Locally uniformizable pairs are
  locally $ℚ$-factorial.  Both statements follow \cite[Lem.~5.16]{KM98}, whose
  (short) proof applies without change in the analytic setting.
\end{rem}

\begin{rem}[Singularities of locally uniformizable pairs]
  If a $\cC$-pair $(X,D)$ is locally uniformizable, then the pair $(X,D)$ is
  klt, \cite[Prop.~5.13]{KM98}.  Recalling that klt singularities are rational,
  \cite[Thm.~5.22]{KM98} and \cite[Thm.~3.12]{FujinoMMP}\footnote{See also the
  vanishing theorems proven in \cite{FujinoVanishing}.}, it follows that $X$ has
  rational singularities.
\end{rem}

\subsubsection{Existence of adapted covers}
\approvals{Erwan & yes \\ Stefan & yes}

A standard computation shows that strongly adapted covers always exist
locally.

\begin{lem}[Strongly adapted covers exist locally]\label{lem:2-36}%
  Let $(X,D)$ be a $\cC$-pair as in Definition~\ref{def:2-24}.  Assume that one
  of the following holds.
  \begin{enumerate}
    \item The space $X$ is Stein and the divisor $D$ has only finitely many
    components.

    \item The space $X$ is projective.
  \end{enumerate}
  If $x ∈ X ∖ \supp \{D\}$ is any point, then there exists a strongly adapted
  Galois cover with cyclic group that is locally biholomorphic over $x$.
\end{lem}
\begin{proof}
  We consider only the Stein setting and refer the reader to
  \cite[Sect.~4.1.B]{Laz04-I} for the projective setting.  The assumption that
  $X$ is Stein guarantees that there exist functions $f_i ∈ 𝒪_X(X)$ such that
  the following holds for every index $i$.
  \begin{itemize}
    \item The function $f_i$ vanishes along the Weil divisor $D_i$ to order one.

    \item The function $f_i$ does not vanish along any of the Weil divisors
    $D_j$, for $j ≠ i$.

    \item The function $f_i$ does not vanish at $x$.
  \end{itemize}
  We can then set $n := \operatorname{lcm} \{ m_i \::\: m_i < ∞\}$ and
  \[
    \what{X} := \text{normalisation of } \Bigl\{ (x,y) ∈ X⨯ 𝔸¹ \::\:
    y^n = \prod_{i \:|\: m_i < ∞} f_i^{n/m_i}(x) \Bigr\}.  \qedhere
  \]
\end{proof}

\begin{rem}
  In the proof of Lemma~\ref{lem:2-36}, if the components of $\{D\}$ are Cartier
  and linearly trivial, then we can find functions $f_i$ that vanish only at
  $D_i$, and the construction will yield a cover that is locally biholomorphic
  away from $\supp \{D\}$.
\end{rem}

\begin{rem}
  If $(X,D)$ is a $\cC$-pair where $X$ is compact but not projective, then it is
  not clear that an adapted cover exists.
\end{rem}

\subsubsection{$\cC$-pairs and root stacks}
\approvals{Erwan & yes \\ Stefan & yes}

Looking at the definition of an adapted morphism, the reader might wonder about
the relation of $\cC$-pairs and root stacks.  We argue that the two notions are
conceptually quite different.  For now, observe that Definitions~\ref{def:2-24}
and \ref{def:2-26} do not ask that $D$ or $D_i$ are Cartier.  For the
construction of root stack, the Cartier assumption is however essential.

\subsection{Weil divisorial sheaves}
\approvals{Erwan & yes \\ Stefan & yes}

Integral Weil divisors on normal spaces define \emph{Weil divisorial sheaves},
that is, coherent, reflexive sheaves of rank one.  Since we will later need to
discuss Weil divisorial sheaves on singular spaces, we briefly recall the
relevant definitions, facts and constructions.

\begin{notation}[Weil divisorial sheaves]
  Let $X$ be a normal analytic variety and let $D = \sum m_i·D_i$ be an
  effective Weil divisor on $X$.  Following standard notation, consider the
  associated Weil divisorial sheaf\index{Weil divisorial sheaf}
  \[
    𝒪_X(-D) = \Bigl( \bigotimes_i 𝒥_{D_i}^{⊗ m_i} \Bigr)^{\vee\vee}.
  \]
  By construction, this sheaf is reflexive of rank one, and comes with a natural
  embedding $𝒪_X(-D) ↪ 𝒪_X$.  We denote the quotient by $𝒪_D :=
  𝒪_X/𝒪_X(-D)$.
\end{notation}

\begin{rem}[Inclusions and projections]\label{rem:2-40} %
  Let $X$ be a normal analytic variety and let $D_1 ≤ D_2$ be two effective Weil
  divisors on $X$.  Then, $𝒪_X(-D_2) ⊆ 𝒪_X(-D_1)$, and so there exists a
  natural surjection $𝒪_{D_2} \twoheadrightarrow 𝒪_{D_1}$.
\end{rem}

In the setting of Reminder~\ref{remi:2-20}, where a meaningful pull-back of a
Weil divisor can be defined, the sheaf $𝒪_{\what{X}}(-γ^*D)$ associated to a
pull-back divisor is generally not equal to the pull-back sheaf $γ^*𝒪_X(-D)$;
note that the pull-back sheaf need not be reflexive or torsion free.  Still,
there exists a comparison morphism that becomes isomorphic if $D$ is Cartier.

\begin{rem}[Pull-back of divisors and sheaves]\label{rem:2-41} %
  Let $γ : \what{X} → X$ be a $q$-morphism and let $D ∈ \Div(X)$ be an effective
  Weil divisor on $X$.  By construction, there exists a canonical morphism
  \[
    γ^*𝒪_X(-D) → 𝒪_{\what{X}}\bigl(-γ^*D\bigr)
  \]
  that is isomorphic over the big open subset of $X$ where $D$ is Cartier.  If
  $D$ is Cartier on all of $X$, then $𝒪_X(-D)$ is invertible, the pull-back
  sequence reads
  \[
    0 → 𝒪_{\what{X}}(-γ^*D) → 𝒪_{\what{X}} → γ^*𝒪_D → 0,
  \]
  and is exact.  It follows that $γ^*𝒪_D = 𝒪_{γ^* D}$.
\end{rem}

% !TEX root = orbiAlb1

\phantomsection\addcontentsline{toc}{part}{Adapted tensors and adapted reflexive tensors}

%
% Do not edit the following line.  The text is automatically updated by
% subversion.
%
\svnid{$Id: 03-adaptedTensors.tex 946 2024-11-11 08:37:55Z kebekus $}
\selectlanguage{british}

\section{Adapted tensors}
\label{sec:3}
\subversionInfo
\approvals{Erwan & yes\\ Stefan & yes}

If $X$ is a compact Kähler manifold, the sheaves $Ω^p_X$ of holomorphic
differentials govern the geometry and topology of $X$.  If $D ∈ \Div(X)$ is a
smooth prime divisor, the sheaves $Ω^p_X(\log D)$ of logarithmic differentials
govern the geometry and topology of $X ∖ D$.

The sequence $(X, \frac{n-1}{n}·D)$ of $\cC$-pairs is meant to interpolate
between the compact manifold $X$ and the non-compact manifold $X ∖ D$.  Along
these lines, the \emph{sheaves of adapted differentials} are meant to
interpolate between the sheaves $Ω^p_X$ and the larger sheaf $Ω¹_X(\log D)$.
Conceptually, an adapted differential is a differential on $X$ with a
logarithmic pole of fractional pole order $\frac{n-1}{n}$ along $D$.
Differentials with fractional pole order do not exist on $X$ in any meaningful
way.  To make the vague concept precise, we work on covers, where adapted
differentials can be defined as differentials that ``locally look'' as if they
were the pull-back of differentials on $X$ that had poles of the appropriate
order.

\subsubsection*{Relation to the literature}
\approvals{Erwan & yes\\ Stefan & yes}

Adapted differentials have been discussed in the literature, but usually not in
a very systematic fashion.  We refer the reader to Campana's original papers
\cite{Cam04, MR2831280}, to the paper \cite{MR3949026} of Campana and Păun, and
to Miyaoka's classic \cite{Miy08} for details, applications and further
references.  The paper \cite[Sect.~6]{KPS22} discusses related definitions in
positive characteristic, the paper \cite[Sect.~3]{CKT21} spells out equivalent,
but perhaps more elementary-looking definitions.

Taking a somewhat different point of view, Pedro Núñez considered the algebraic
setting, where $X$ is a scheme, and constructed a presheaf of ``adapted
differentials'' on the category $\operatorname{Sch}/X$ that is a sheaf with
respect to the qfh-topology, \cite{Nun22}.  If everything is algebraic, the
constructions outlined in this section are compatible with those of Núñez.

\subsection{Sample computations}
\approvals{Erwan & yes \\ Stefan & yes}
\label{sec:3-1}

To prepare the reader for the somewhat technical discussion in
Definition~\vref{def:3-2}, we illustrate the main ideas in two simple cases
first.  The reader who is already familiar with ``adapted differentials'' might
want to skip this section.

\subsubsection{Sample computation in dimension one}
\approvals{Erwan & yes\\ Stefan & yes}
\label{sec:3-1-1}

Equip $𝔸¹$ with a coordinate $z$ and consider the simple case where
\[
  \textstyle(X,D) = \Bigl(𝔸¹,\, \frac{m-1}{m}·\{0\}\Bigr), \quad \text{for one number }m ∈ ℕ⁺.
\]
The sheaf $Ω¹_{(X,D)}$ of differential forms with logarithmic poles of order
$\frac{m_i-1}{m_i}$ should then ideally look as shown in Table~\ref{tab:1}.
\begin{table}
  \[
    \begin{matrix}
      Ω¹_X & ⊆ & Ω¹_{(X,D)} & ⊆ & Ω¹_X(\log D) \\[3mm]
      \bigl\langle dz \bigr\rangle_{𝒪_X} & ⊆ & \bigl\langle z^{-\frac{m-1}{m}}·dz \bigr\rangle_{𝒪_X} & ⊆ & \bigl\langle z^{-1}·dz \bigr\rangle_{𝒪_X}.
    \end{matrix}
  \]
  \caption{Hypothetical sheaves and generators on $X = 𝔸¹$}
  \label{tab:1}
\end{table}
The reader will immediately note that $z^{-\frac{m-1}{m}}$ cannot possibly exist
as a single-valued function on $X$, unless $m = 1$ or $m = ∞$.  To resolve the
problem, write $\what{X} := 𝔸¹$, choose a number $α ∈ ℕ⁺$ and consider the
cover
\[
  γ : \what{X} → X, \quad z ↦ z^{α·m}.
\]
Applying the formal rules of derivation, we find that
\[
  dγ\left(z^{-\frac{m-1}{m}}·dz\right) = α m·z^{α-1}·dz.
\]
Accordingly, as shown in Table~\ref{tab:2}, there exists a sheaf $Ω¹_{(X,D,γ)}$
of differentials on $\what{X}$ that looks as if it was the pull-back $γ^*
Ω¹_{(X,D)}$ of the hypothetical sheaf $Ω¹_{(X,D)}$ on $X$.
\begin{table}
  \[
    \begin{matrix}
      γ^* Ω¹_X & ⊆ & Ω¹_{(X,D,γ)} & ⊆ & γ^* Ω¹_X(\log D) \\[3mm]
      \bigl\langle z^{α m-1}·dz \bigr\rangle_{𝒪_{\what{X}}} & ⊆ & \bigl\langle z^{α-1}·dz \bigr\rangle_{𝒪_{\what{X}}} & ⊆ & \bigl\langle z^{-1}·dz \bigr\rangle_{𝒪_{\what{X}}}.
    \end{matrix}
  \]
  \caption{Sheaves and generators on $\what X = 𝔸¹$}
  \label{tab:2}
\end{table}
The inclusion $γ^* Ω¹_X ⊆ Ω¹_{(X,D,γ)}$ allows interpreting sections in
$Ω¹_{(X,D,γ)}$ as meromorphic sections of $γ^* Ω¹_X$ with a pole of order
$α·(m-1)$ along $\{ 0\}$.  Given that $γ^* D = α(m-1)·\{ 0\}$, we can write
$Ω¹_{(X,D,γ)}$ in a coordinate-free way as
\begin{equation}\label{eq:3-0-1}%
  Ω¹_{(X,D,γ)} = 𝒪_{\what{X}}(γ^* D) ⊗ γ^* Ω¹_X.
\end{equation}

\subsubsection{Sample computation in dimension two}
\approvals{Erwan & yes\\ Stefan & yes}
\label{sec:3-1-2}

Next, equip $𝔸²$ with coordinates $x, y$ and consider the pair
\[
  \textstyle (X,D) = \Bigl(𝔸², \frac{m-1}{m}·\{y=0\}\Bigr), \quad \text{for one number } m ∈ ℕ⁺.
\]
The sheaf $Ω¹_{(X,D)}$ of differential forms with logarithmic poles of order
$\frac{m_i-1}{m_i}$ should then ideally look as shown in Table~\ref{tab:3}.
\begin{table}
  \[
    \begin{matrix}
      Ω¹_X & ⊆ & Ω¹_{(X,D)} & ⊆ & Ω¹_X(\log D) \\[3mm]
      \bigl\langle dx,\, dy \bigr\rangle_{𝒪_X} & ⊆ & \bigl\langle dx,\, y^{-\frac{m-1}{m}}·dy \bigr\rangle_{𝒪_X} & ⊆ & \bigl\langle dx,\, y^{-1}·dy \bigr\rangle_{𝒪_X}.
    \end{matrix}
  \]
  \caption{Hypothetical sheaves and generators on $X = 𝔸²$}
  \label{tab:3}
\end{table}
As before, write $\what{X} := 𝔸²$, choose a number $α ∈ ℕ⁺$ and consider the
cover
\[
  γ : \what{X} → X, \quad (x,y) ↦ \bigl(x,y^{α·m}\bigr).
\]
Again, we find a sheaf $Ω¹_{(X,D,γ)}$ of differentials on $\what{X}$ that looks
as if it was the pull-back $γ^* Ω¹_{(X,D)}$ of the hypothetical sheaf
$Ω¹_{(X,D)} = \langle dx, y^{-\frac{m-1}{m}}·dy \rangle_{𝒪_X}$.
Table~\ref{tab:4} spells out the details.
\begin{table}
  \[
    \begin{matrix}
      γ^* Ω¹_X & ⊆ & Ω¹_{(X,D,γ)} & ⊆ & γ^* Ω¹_X(\log D) \\[3mm]
      \bigl\langle dx,\, y^{α m-1}·dy \bigr\rangle_{𝒪_{\what{X}}} & ⊆ & \bigl\langle dx,\, y^{α-1}·dy \bigr\rangle_{𝒪_{\what{X}}} & ⊆ & \bigl\langle dx,\, y^{-1}·dy \bigr\rangle_{𝒪_{\what{X}}}.
    \end{matrix}
  \]
  \caption{Sheaves and generators on $\what X = 𝔸²$}
  \label{tab:4}
\end{table}
In contrast to the one-dimension case, observe that Equation~\eqref{eq:3-0-1} no
longer describes $Ω¹_{(X,D,γ)}$ in the present setting.  In fact, it turns out
that
\[
  𝒪_{\what{X}}(γ^* D) ⊗ γ^* Ω¹_X = \bigl\langle y^{-α·(m-1)}·dx,\, y^{α-1}·dy \bigr\rangle_{𝒪_{\what{X}}}
\]
is a strict supersheaf of $Ω¹_{(X,D,γ)}$.  To this end, observe that sections in
$𝒪_{\what{X}}(γ^* D) ⊗ γ^* Ω¹_X$ are required to have the correct pole order,
but need not have logarithmic poles.  The following coordinate-free description
avoids that problem,
\begin{equation}\label{eq:3-0-2}
  Ω¹_{(X,D,γ)} = \underbrace{𝒪_{\what{X}}(γ^* D) ⊗ γ^* Ω¹_X}_{=: 𝒜_{1,1}} \:∩\: \underbrace{Ω¹_{\what{X}}(\log γ^* D)}_{=: ℬ_{1,1}}.
\end{equation}
Here, the intersection takes place in the sheaf $𝒪_{\what{X}}(*γ^*D)⊗\left(γ^*
Ω¹_X \right)$ that contains both $𝒜_{1,1}$ and $ℬ_{1,1}$.

\subsubsection{Further generalizations}
\approvals{Erwan & yes\\ Stefan & yes}

In principle, one would like to take \eqref{eq:3-0-2} as the definition of
$Ω¹_{(X,D,γ)}$.  There are, however, several further generalizations that we
would like to consider.
\begin{description}
  \item[Arbitrary $q$-morphisms] The morphisms $γ$ that we considered in
  Section~\ref{sec:3-1-1}--\ref{sec:3-1-2} were adapted covers of the pair
  $(X,D)$.  In practise, adapted covers do not always exist\footnote{A compact
  complex manifold of algebraic dimension zero need not have any global covers
  at all!}.  For technical reasons, we will need to define $Ω¹_{(X,D,γ)}$ for
  $q$-morphisms $γ$ that are not necessarily adapted.  There, we use the
  round-down of $γ^* D$ as the best approximation, replacing the sheaf
  $𝒜_{1,1}$ of \eqref{eq:3-0-2} by
  \[
    𝒜_{1,1} := 𝒪_{\what{X}}\bigl(⌊γ^* D⌋\bigr) ⊗ γ^* Ω¹_X.
  \]

  \item[Logarithmic boundary components]
  Sections~\ref{sec:3-1-1}--\ref{sec:3-1-2} considered pairs $(X,D)$ where $⌊D⌋
  = 0$.  In case where $D$ is reduced, the setting simplifies dramatically, as
  the ``sheaf $Ω^p_{(X,D)}$ of differential forms with logarithmic poles of
  order $\frac{∞-1}{∞}$'' is simply $Ω¹_X(\log D)$.  In general, where $D$ is
  allowed to have reduced and non-reduced components, we consider the fractional
  and integral part of $D$ separately, replacing the sheaf $𝒜_{1,1}$ of
  \eqref{eq:3-0-2} by
  \[
    𝒜_{1,1} := 𝒪_{\what{X}}\bigl(⌊γ^* \{D\}⌋\bigr) ⊗ γ^* Ω¹_X\bigl(\log ⌊D⌋\bigr).
  \]
  
  \item[Higher-order tensors] In addition to sections of $Ω¹_•$, we will also
  need to consider $p$-forms and more generally sections in symmetric powers of
  $Ω^p_•$.  Again, this forces us to generalize, replacing $𝒜_{1,1}$ and
  $ℬ_{1,1}$ by the sheaves $𝒜_{n,p}$ and $ℬ_{n,p}$ found in
  Definition~\ref{def:3-2} below.
\end{description}

\subsection{Definition and first examples}
\approvals{Erwan & yes\\ Stefan & yes}

Throughout the present Section~\ref{sec:3}, we will work in the following
setting.

\begin{setting}\label{set:3-1}%
  Let $(X,D)$ be a $\cC$-pair as in Definition~\ref{def:2-24} and let $γ :
  \what{X} → X$ be a $q$-morphism.  Assume that the pairs $(X, D)$ and
  $\bigl(\what{X}, γ^*D\bigr)$ are nc.
\end{setting}

In view of the sample computations in Section~\ref{sec:3-1}, we hope that the
following definition will not come as a surprise.

\begin{defn}[Adapted tensors]\label{def:3-2}%
  Assume Setting~\ref{set:3-1}.  Given numbers $n$, $p ∈ ℕ⁺$, consider the
  sheaves
  \begin{align*}
    𝒜_{n,p} & := 𝒪_{\what{X}}\bigl( ⌊n·γ^* \{D\}⌋ \bigr) ⊗ γ^* \Sym^n Ω^p_X \bigl(\log ⌊D⌋ \bigr) \\
    ℬ_{n,p} & := \Sym^n Ω^p_{\what{X}}(\log γ^* D).
  \end{align*}
  Observe that both are subsheaves of $𝒪_{\what{X}}\bigl(γ^* D \bigr) ⊗
  \left(γ^* \Sym^n Ω^p_X \right)$ and define the \emph{bundle of adapted
  $(n,p)$-tensors}\index{adapted!tensor} as the intersection
  \[
    \Sym^n_{\cC} Ω^p_{(X,D,γ)} := 𝒜_{n,p} \:∩\: ℬ_{n,p}.
  \]
  Collectively, we refer to $\Sym^•_{\cC} Ω^•_{(X,D,γ)}$ as the \emph{bundles of
  adapted tensors}.
\end{defn}

\begin{defn}[\protect{Adapted differentials, cf.~\cite[Sect.~5.2]{MR3949026}}]\label{def:3-3}%
  Assume Setting~\ref{set:3-1}.  Given a number $p ∈ ℕ⁺$, define the
  \emph{bundle of adapted $p$-forms}\index{adapted!p-form@$p$-form} as
  $Ω^p_{(X,D,γ)} := \Sym¹_{\cC} Ω^p_{(X,D,γ)}$.  Collectively, we refer to
  $Ω^•_{(X,D,γ)}$ as the \emph{bundles of adapted
  differentials}\index{adapted!differential}.  The sheaf $Ω¹_{(X,D,γ)}$ is
  called \emph{$\cC$-cotangent bundle}\index{C-cotangent bundle@$\cC$-cotangent
  bundle}.
\end{defn}

The remainder of Section~\ref{sec:3} lists properties of $\Sym^n_{\cC}
Ω^p_{(X,D,γ)}$ that will become relevant later.  The proofs are elementary but
tedious, and will typically involve a computation in local coordinates, or
numerical arguments of the form $⌊γ^* D⌋ ≥ γ^* ⌊D⌋$.  To keep the size of this
(already long) paper within reasonable limits, we refrain from giving full
proofs and formulate a sequence of statements euphemistically called
``observations'', that is, homework left for the reader.  The first observation
justifies the word ``bundle'' in Definitions~\ref{def:3-2} and \ref{def:3-3}.

\begin{obs}[Local freeness]\label{obs:3-4}%
  The sheaves $\Sym^n_{\cC} Ω^p_{(X,D,γ)}$ of Definition~\ref{def:3-2} are
  locally free.  \qed
\end{obs}

The following examples illustrate the construction, expanding
Definition~\ref{def:3-2} in a number of special cases.

\begin{example}[Special cases]\label{ex:3-5}%
  In Setting~\ref{set:3-1}, assume that numbers $n$, $p ∈ ℕ⁺$ are given.
  \begin{enumerate}
    \item If $p = \dim X$, then
    \[
      \Sym^n_{\cC} Ω^p_{(X,D,γ)} = \bigl(γ^* ω_X^{⊗ n}\bigr) ⊗ 𝒪_{\what{X}}\bigl(⌊n·γ^* D⌋\bigr).
    \]

    \item If $D = 0$, then $\Sym^n_{\cC} Ω^p_{(X,0,γ)} = γ^* \Sym^n Ω^p_X$.

    \item\label{il:3-5-3} If $γ = \Id_X$, then $Ω^p_{(X,D,\Id_X)} =
    Ω^p_X\bigl(\log ⌊D⌋\bigr)$.

    \item\label{il:3-5-4} If $γ$ is strongly adapted and $\Branch(γ) ⊆ \supp D$,
    then
    \[
      \Sym^n_{\cC} Ω^p_{(X,D,γ)} = \Sym^n Ω^p_{\what{X}}\bigl(\log γ^*⌊D⌋\bigr).
    \]
  \end{enumerate}
\end{example}

\begin{example}[Functions with adapted differential]\label{ex:3-6}%
  In the setting of Definition~\ref{def:3-2}, assume that $γ$ is a cover.  Given
  a function
  \begin{equation}\label{eq:3-6-1}
    \what{f} ∈ H⁰\Bigl( \what{X},\, 𝒪_{\what{X}}(- \Ramification γ) \Bigr),
  \end{equation}
  that vanishes along the ramification divisor\footnote{Recall from
  Notation~\ref{not:2-22} that the ramification divisor is reduced, so that
  \ref{eq:3-6-1} is a statement about the vanishing locus of $\what{f}$, but not
  about the order of vanishing.} of $γ$, an elementary computation in local
  coordinates shows that the following statements are equivalent.
  \begin{enumerate}
    \item\label{il:3-6-2} The Kähler differential of $\what{f}$ is adapted,
  
    \[
      d\what{f} ∈ H⁰\bigl( \what{X},\, Ω¹_{(X,D,γ)} \bigr) ⊆ H⁰\bigl( \what{X},\,
      Ω¹_{\what{X}} (\log γ^* D)\bigr).
    \]

    \item\label{il:3-6-3} The zero-divisor of the function $\what{f}$ satisfies
    the inequality
    \[
      \operatorname{div} \what{f}
      ≥ \sum_{Δ_{\what{X}} ⊆ \operatorname{div} \what{f}} \frac{\mult_{Δ_{\what{X}}} γ^* Δ_X}{\mult_{\cC, Δ_X} D}·Δ_{\what{X}},
    \]
    where $Δ_X := \bigl( γ_* Δ_{\what{X}} \bigr)_{\red}$ and $(\text{finite})/∞
    = 0$.
  \end{enumerate}
\end{example}

\begin{rem}
  Notice that \ref{il:3-6-2} says nothing about the vanishing order of
  $\what{f}$ along divisors $Δ_{\what{X}}$ that lie over the support of $⌊D⌋$.
  In a similar vein, \ref{il:3-6-3} says nothing about the vanishing order of
  $\what{f}$ along divisors $Δ_{\what{X}}$ outside the ramification locus.
\end{rem}

\subsection{Alternative description}
\approvals{Erwan & yes\\ Stefan & yes}

In his seminal paper \cite{Miy08}, Miyaoka describes the $\cC$-cotangent bundle
in terms of the classic residue sequence for logarithmic differentials,
\[
  0 → Ω¹_X(\log ⌊D⌋) → Ω¹_X(\log D)
  \xrightarrow{\text{residue}} \bigoplus_{i \:|\: m_i < ∞} 𝒪_{D_i}
  → 0.
\]
An elementary computation in local coordinates shows that Miyaoka's construction
agrees with Definition~\ref{def:3-2} above.

\begin{obs}[\protect{Miyaoka's description of the $\cC$-cotangent bundle, \cite[p.~412]{Miy08}}]\label{obs:3-8}%
  In Setting~\ref{set:3-1}, the $\cC$-cotangent bundle $Ω¹_{(X,D,γ)} ⊆ γ^*
  Ω¹_X(\log D)$ equals the kernel of the following composed morphism,
  \[
    γ^* Ω¹_X(\log D) %
    \xrightarrow{γ^*\text{(residue)}} %
    \bigoplus_{i \:|\: m_i < ∞} γ^* 𝒪_{D_i} %
    \overset{\text{Rem.~\ref{rem:2-41}}}{=} %
    \bigoplus_{i \:|\: m_i < ∞} 𝒪_{γ^*D_i} %
    \xrightarrowdbl{\text{Rem.~\ref{rem:2-40}}} %
    \bigoplus_{i \:|\: m_i < ∞} 𝒪_{\left⌈{\textstyle\frac{1}{m_i}}·γ^* D_i\right⌉}.
  \]
  This describes the $\cC$-cotangent bundle by means of the following exact
  sequence,
  \begin{equation}\label{eq:3-8-1}
    0 → Ω¹_{(X,D,γ)}
    → γ^* Ω¹_X (\log D) → \bigoplus_{i} 𝒪_{\left⌈{\textstyle\frac{1}{m_i}}·γ^* D_i\right⌉}
    → 0.
  \end{equation}
  As before, the notation in \eqref{eq:3-8-1} follows the convention that
  $\frac{1}{∞} = 0$.  \qed
\end{obs}

\subsection{Inclusions}
\label{sec:3-4}
\approvals{Erwan & yes\\ Stefan & yes}

The construction equips the bundles of adapted tensors with numerous inclusions
that we use throughout the paper.  The following observation summarizes the most
important ones for later reference.

\begin{obs}[Inclusions]\label{obs:3-9}%
  Assume Setting~\ref{set:3-1}.  Given numbers $n$, $p ∈ ℕ⁺$, there exist
  natural inclusions as follows,
  \[
    \begin{tikzcd}
      \Sym^n Ω^p_{\what{X}} \bigl(\log γ^*⌊D⌋ \bigr) \ar[r, hook] & \Sym^n Ω^p_{\what{X}} \bigl(\log γ^*D \bigr) \\
      & γ^* \Sym^n Ω^p_X(\log D) \ar[u, hook] \\
      \Sym^n_{\cC} Ω^p_{(X,D,γ)} \ar[r, equal] \ar[uu, hook, "ι_{n,p}"] & \Sym^n_{\cC} Ω^p_{(X,D,γ)} \ar[u, hook] \\
      \Sym^n Ω^p_{(X,D,γ)} \ar[r, equal] \ar[u, hook] & \Sym^n Ω^p_{(X,D,γ)} \ar[u, hook] \\
      γ^* \Sym^n Ω^p_X \ar[r, equal] \ar[u, hook] & γ^* \Sym^n Ω^p_X.  \ar[u, hook]
    \end{tikzcd}
  \]
  Note that all sheaves here are subsheaves of the quasi-coherent sheaf of
  meromorphic tensors on $\what{X}$, that is, $ℳ_{\what{X}} ⊗ \Sym^n
  Ω^p_{\what{X}}$.  \qed
\end{obs}

\begin{obs}[Uniformization]\label{obs:3-10}%
  Assume that the morphism $γ : \what{X} → X$ of Setting~\ref{set:3-1} is an
  adapted cover and consider the morphisms
  \[
    ι_{•,•} : \Sym^•_{\cC} Ω^•_{(X,D,γ)} ↪ \Sym^•Ω^•_{\what{X}} \bigl(\log γ^*⌊D⌋ \bigr)
  \]
  of Observation~\ref{obs:3-9}.  Then, the following statements are equivalent.
  \begin{enumerate}
    \item\label{il:3-10-1} The morphism $γ$ is a uniformization.

    \item\label{il:3-10-2} There exist numbers $1 ≤ p ≤ \dim X$ and $1 ≤ n$ such
    that $ι_{n,p}$ is isomorphic.

    \item\label{il:3-10-3} For every pair of numbers $p, n ∈ ℕ⁺$, the inclusion
    $ι_{n,p}$ is isomorphic.  \qed
  \end{enumerate}
\end{obs}

\begin{rem}
  Observation~\ref{obs:3-10} does not hold without the assumption that $γ$ is
  adapted.  For a counterexample, observe that the identity map $γ := \Id_X$
  almost never uniformizes.  Yet, we have seen in Item~\ref{il:3-5-3} of
  Example~\ref{ex:3-5} that $Ω^p_{(X,D,\Id_X)} = Ω^p_X\bigl(\log ⌊D⌋\bigr)$ for
  every number $p$.
\end{rem}

\subsection{Operations}
\approvals{Erwan & yes\\ Stefan & yes}
\label{sec:3-5}

Among all meromorphic differential forms, logarithmic forms are characterized by
the fact that the pole order does not change under exterior derivatives and
wedge products.  The exterior derivatives and wedge products on $Ω^•_{\what{X}}
\bigl(\log γ^*D \bigr)$ therefore induce operations on the bundles of adapted
differentials.

\begin{obs}[Wedge products and exterior derivatives]\label{obs:3-12}%
  In Setting~\ref{set:3-1}, observe that the subsheaves
  \[
    Ω^•_{(X,D,γ)} \overset{ι_{1,•}}{⊆} Ω^•_{\what{X}} \bigl(\log γ^*⌊D⌋ \bigr)
  \]
  are closed under wedge products and exterior derivatives.  Given numbers $p$
  and $q$, we obtain natural operations
  \[
    Λ : Ω^p_{(X,D,γ)} ⨯ Ω^q_{(X,D,γ)} → Ω^{p+q}_{(X,D,γ)}
    \quad\text{and}\quad
    d : Ω^p_{(X,D,γ)} → Ω^{p+1}_{(X,D,γ)}.
    \eqno \qed
  \]
\end{obs}

We turn to symmetric powers.  The trivial observation that
\[
  ⌊n_1·γ^*\{D\}⌋ + ⌊n_2·γ^*\{D\}⌋ ≤ ⌊(n_1+n_2)·γ^*\{D\}⌋, \quad \text{for all } n_1, n_2 ∈ ℕ
\]
allows defining symmetric products on the bundles of adapted tensors that are
compatible with the products in the standard symmetric algebra $\Sym^•
Ω^•_{\what{X}} \bigl(\log γ^*D \bigr)$.

\begin{obs}[Symmetric multiplication]\label{obs:3-13}%
  In Setting~\ref{set:3-1}, observe that the subsheaves
  \[
    \Sym^•_{\cC} Ω^•_{(X,D,γ)} ⊆ \Sym^• Ω^•_{\what{X}} \bigl(\log γ^*⌊D⌋ \bigr)
  \]
  are closed under symmetric multiplication.  Given numbers $p$, $n_1$, $n_2 ∈
  ℕ⁺$, we obtain natural maps
  \[
    \Sym^{n_1}_{\cC} Ω^p_{(X,D,γ)} ⨯ \Sym^{n_2}_{\cC} Ω^p_{(X,D,γ)} → \Sym^{n_1+n_2}_{\cC} Ω^p_{(X,D,γ)}
  \]
  and
  \[
    \Sym^{n_1} \Sym^{n_2}_{\cC} Ω^p_{(X,D,γ)} ↪ \Sym^{n_1·n_2}_{\cC} Ω^p_{(X,D,γ)}.
    \eqno \qed
  \]
\end{obs}

\begin{obs}[Adapted tensors on adapted covers]\label{obs:3-14}%
  If the $q$-morphism $γ$ of Setting~\ref{set:3-1} is adapted, then $γ^* \{D\}$
  is integral and
  \[
    \Sym^n_{\cC} Ω^p_{(X,D,γ)} = \Sym^n Ω^p_{(X,D,γ)}.
  \]
  In particular, we find that the symmetric multiplication maps of
  Observation~\ref{obs:3-13} are surjective.  \qed
\end{obs}

\subsection{Functoriality}
\approvals{Erwan & yes\\ Stefan & yes}
\label{sec:3-6}

Given a $\cC$-pair $(X,D)$ and two covers $\what{Y}_1 \twoheadrightarrow X$ and
$\what{Y}_2 \twoheadrightarrow X$, one would often like to compare adapted
tensors on $\what{Y}_1$ with those on $\what{Y}_2$.  Typically, this amounts to
choosing one cover $\what{X} \twoheadrightarrow X$ that dominates both
$\what{Y}_•$, and then comparing the adapted tensors on $\what{Y}_•$ with those
on $\what{X}$.  The following observation yields the necessary comparison
morphisms.

\begin{obs}[Functoriality in $q$-morphisms]\label{obs:3-15}%
  In the Setting~\ref{set:3-1}, assume that the morphism $γ$ factors into a
  sequence of $q$-morphisms,
  \[
    \begin{tikzcd}
      \what{X} \ar[r, "α"'] \ar[rr, bend left=20, "γ"] & \what{Y} \ar[r, "β"'] & X,
    \end{tikzcd}
  \]
  where $\bigl(\what{Y}, β^*D\bigr)$ is likewise nc.  Given numbers $n$, $p ∈
  ℕ⁺$, there exists a commutative diagram as follows,
  \[
    \begin{tikzcd}[column sep=2cm, baseline=(current bounding box.south east)]
      α^* \Sym^n Ω^p_{\what{Y}}(\log β^* ⌊D⌋) \ar[r, hook, "\diff α"] & \Sym^n Ω^p_{\what{X}}\bigl( \log γ^* ⌊D⌋\bigr) \\
      α^* \Sym^n_{\cC} Ω^p_{(X,D,β)} \ar[r, hook] \ar[u, hook, "\text{inclusion } α^*(ι_{n,p})"] & \Sym^n_{\cC} Ω^p_{(X,D,γ)}.  \ar[u, hook, "\text{inclusion } ι_{n,p}"']
    \end{tikzcd}
    \eqno \qed
  \]
\end{obs}

\begin{obs}[Functoriality in adapted morphisms]\label{obs:3-16}%
  If the morphism $β$ of Observation~\ref{obs:3-15} is adapted for $(X,D)$, then
  $γ$ is likewise adapted.  The natural morphisms $α^* \Sym^•_{\cC}
  Ω^•_{(X,D,β)} ↪ \Sym^•_{\cC} Ω^•_{(X,D,γ)}$ are isomorphic in this case.  \qed
\end{obs}

Observation~\ref{obs:3-15} asserts that the pull-back of an adapted tensor on
$\what{Y}$ is an adapted tensor on $\what{X}$.  For future applications, the
following observation notes that the converse is also true: A meromorphic tensor
on $\what{Y}$ is adapted \emph{if and only if} its pull-back is adapted on
$\what{X}$.  The proof is an exercise in ``pull-back and round up/round down''
and uses the classic fact that a tensor has logarithmic poles if and only if its
pull-back has logarithmic poles.

\begin{obs}[Test for adapted tensors, compare Lemma~\ref{lem:2-23}]\label{obs:3-17}%
  In the setting of Observation~\ref{obs:3-15}, let $n$ and $p ∈ ℕ⁺$ be two
  numbers, let $U ⊆ \img α ⊆ \what{Y}$ be open and let $σ ∈ \Bigl(ℳ_{\what{Y}} ⊗
  \Sym^n Ω^p_{\what{Y}}\Bigr)(U)$ be any meromorphic tensor on $U$.  Then, the
  following are equivalent.
  \begin{enumerate}
    \item\label{il:3-17-1} The section $σ$ is an adapted tensor.  More
    precisely: the meromorphic tensor $σ$ is a section of the subsheaf
    $\Sym^n_{\cC} Ω^p_{(X,D,β)} ⊆ ℳ_{\what{Y}} ⊗ \Sym^n Ω^p_{\what{Y}}$.

    \item\label{il:3-17-2} The pull-back of $σ$ is an adapted tensor.  More
    precisely: the meromorphic tensor $(\diff α)(σ)$ is a section of the
    subsheaf $\Sym^n_{\cC} Ω^p_{(X,D,γ)} ⊆ ℳ_{\what{X}} ⊗ \Sym^n
    Ω^p_{\what{X}}$.  \qed
  \end{enumerate}
\end{obs}

\begin{consequence}[Trace morphism]\label{cons:3-18}%
  In the setting of Observation~\ref{obs:3-15}, the trace map
  \[
    α_* Ω^•_{\what{X}}\bigl(\log γ^* D\bigr)
    \xrightarrow{\trace} Ω^•_{\what{Y}}(\log β^* D)
  \]
  maps adapted differentials to adapted differentials.  More precisely, there
  exist commutative diagrams as follows,
  \[
    \begin{tikzcd}[column sep=2cm, baseline=(current bounding box.south east)]
      α_* Ω^•_{(X,D,γ)} \ar[r, hook, "\text{Obs.~\ref{obs:3-9}}"] \ar[d, "\text{restr.~of }\trace"'] & α_* Ω^•_{\what{X}}\bigl( \log γ^* D\bigr) \ar[d, "\trace"] \\
      Ω^•_{(X,D,β)} \ar[r, hook, "\text{Obs.~\ref{obs:3-9}}"'] & Ω^•_{\what{Y}}( \log β^* D)
    \end{tikzcd}
    \eqno \qed
  \]
\end{consequence}

We remark that Consequence~\ref{cons:3-18} has no analogue for higher-order
tensors.  Already in the simplest case where $γ : 𝔸¹ → 𝔸¹$ is a uniformization
of the pair $(X,D) = \bigl(𝔸¹, \frac{1}{2}·\{0\}\bigr)$ that factorizes as
\[
  \begin{tikzcd}[column sep=1.5cm]
    \underbrace{𝔸¹}_{= \what{X}} \ar[r, "α = γ", "z ↦ z²"'] & \underbrace{𝔸¹}_{= \what{Y}} \ar[r, "β = \Id"] & \underbrace{𝔸¹}_{= X},
  \end{tikzcd}
\]
the Galois-invariant two-tensor
\[
  dz·dz ∈ H⁰\bigl( \what{X},\, \Sym²_{\cC} Ω¹_{(X,D,γ)} \bigr) = H⁰\bigl( 𝔸¹,\, \Sym² Ω¹_{𝔸¹} \bigr)
\]
does not induce any section of $\Sym²_{\cC} Ω¹_{(X,D,β)} = \Sym²
Ω¹_{𝔸¹}$.

\subsection{Galois linearization}
\approvals{Erwan & yes\\ Stefan & yes}
\label{sec:3-7}

All sheaves that we have discussed in Definition~\ref{def:3-2} are linearized
with respect to the action of the relative automorphism group
$\Aut_{𝒪}\bigl(\what{X}/X\bigr)$.  If the morphism $γ$ is Galois, then all
sheaves are Galois-linearized.  We refer the reader to \cite[Appendix~A and
references there]{GKKP11} for more on $G$-sheaves and $G$-invariant
push-forward.

\begin{obs}[Linearisation]\label{obs:3-19}%
  Assume Setting~\ref{set:3-1} and write $G := \Aut_{𝒪}\bigl(\what{X}/X\bigr)$
  for the relative automorphism group.  Then, all sheaves $\Sym^•_{\cC}
  Ω^•_{(X,D,γ)}$ of Definition~\ref{def:3-2} carry natural $G$-linearisations
  that are compatible with the natural $G$-linearisations of
  \begin{gather*}
    \Sym^n γ^* Ω^p_X, \quad γ^* \Sym^n Ω^p_X(\log D), \quad \Sym^n Ω^p_{\what{X}} \bigl(\log γ^*D \bigr), \\
    \Sym^n Ω^p_{\what{X}} \bigl(\log γ^*⌊D⌋ \bigr).
  \end{gather*}
  The inclusions of Observation~\ref{obs:3-9} are morphisms of $G$-sheaves.
  \qed
\end{obs}

Observation~\ref{obs:3-19} applies most prominently in settings where $γ$ is
Galois.  It also applies in the setting of Observation~\ref{obs:3-15},
effectively allowing us to compare adapted tensors on $\what{X}$ with those on
$\what{Y}$.

\begin{lem}[Invariant push-forward]\label{lem:3-20}%
  In the setting of Observation~\ref{obs:3-15}, assume that the $q$-morphisms
  $α$, $β$ and $γ$ are covers, and that $α$ is Galois with group $G$.  Given
  numbers $n$, $p ∈ ℕ⁺$, the natural morphism between $G$-invariant push-forward
  sheaves induced by the bottom row of the commutative diagram in
  Observation~\ref{obs:3-15},
  \[
    \Sym^n_{\cC} Ω^p_{(X,D,β)} = \Bigl( α_* α^* \Sym^n_{\cC} Ω^p_{(X,D,β)} \Bigr)^G ↪ \Bigl( α_* \Sym^n_{\cC} Ω^p_{(X,D,γ)} \Bigr)^G,
  \]
  is isomorphic.
\end{lem}
\begin{proof}
  Given an open set $U ⊆ \what{Y}$ and a $G$-invariant adapted tensor on
  $\what{X}$,
  \[
    σ ∈ \Sym^n_{\cC} Ω^p_{(X,D,γ)} (α^{-1}U),
  \]
  we need to find an adapted tensor $τ ∈ \Sym^n_{\cC} Ω^p_{(X,D,β)} (U)$ whose
  pull-back equals $σ$.  To this end, consider the inclusion
  \begin{equation}\label{eq:3-20-1}
    \Sym^n_{\cC} Ω^p_{(X,D,γ)} ↪ γ^* \Sym^n Ω^p_X(\log D)
  \end{equation}
  discussed in Observation~\ref{obs:3-9}.  We have remarked in
  Observation~\ref{obs:3-19} that \eqref{eq:3-20-1} is an inclusion of
  $G$-linearized sheaves; the $G$-invariant adapted tensor $σ$ will therefore
  define a $G$-invariant section
  \[
    σ' ∈ γ^* \Sym^n Ω^p_X(\log D) (α^{-1}U) = α^* β^* \Sym^n Ω^p_X(\log D) (α^{-1}U).
  \]
  In this $G$-invariant setting, it immediately equips us with an associated
  section
  \[
    τ' ∈ β^* \Sym^n Ω^p_X(\log D) (U),
  \]
  whose pull-back equals $σ'$.  In order to conclude, it remains to show that
  $τ'$ is an adapted tensor.  Equivalently: It remains to show that $τ'$ really
  a section of the subsheaf
  \[
    \Sym^n_{\cC} Ω^p_{(X,D,β)} ↪ β^* \Sym^n Ω^p_X(\log D).
  \]
  This is exactly the implication \ref{il:3-17-2} $⇒$ \ref{il:3-17-1} of
  Observation~\ref{obs:3-17} above.
\end{proof}

\subsection{Chern classes}
\approvals{Erwan & yes\\ Stefan & yes}

Observation~\ref{obs:3-8} describes the $\cC$-cotangent bundle by means of an
exact sequence that allows computing Chern classes.  For simplicity, we restrict
ourselves to the compact setting, where Chern classes for coherent analytic
sheaves can be defined in rational cohomology, as explained \cite{MR0860421} and
briefly recalled in \cite[Sect.~1]{MR2606937}.  In settings where more general
or more refined classes exist, the computations described here will work without
change.

\begin{obs}[\protect{Total Chern class of the $\cC$-cotangent bundle, cf.~\cite[Thm.~4.1.1]{Gomez24}}]
  In Setting~\ref{set:3-1}, assume that the $q$-morphism $γ$ is adapted.  Then,
  \[
    \textstyle \left⌈\frac{1}{m_i}·γ^* D_i\right⌉ = \frac{1}{m_i}·γ^* D_i, \quad \text{for every }i.
  \]
  If $\what{X}$ and $X$ are compact, then an elementary computation applying the
  Whitney formula to Sequence~\eqref{eq:3-8-1} reveals the total Chern class of
  $Ω¹_{(X,D,γ)}$ as
  \begin{align*}
    c\Bigl(Ω¹_{(X,D,γ)}\Bigr) & = c\Bigl(γ^* Ω¹_X(\log D)\Bigr)·\prod_{i \:|\: m_i < ∞} c\Bigl(𝒪_{\frac{1}{m_i}·γ^* D_i}\Bigr)^{-1} \\
    & = γ^* \left( c\bigl(Ω¹_X\bigr)·\prod_i \left({\textstyle \frac{m_i-1}{m_i}·c\bigl(𝒪_{D_i}}\bigr) + {\textstyle \frac{1}{m_i}}\right)\right) ∈ H^*\bigl(\what{X},\, ℚ\bigr).  \\
  \end{align*}
  In particular, we find that
  \[
    c_1\bigl(Ω¹_{(X,D,γ)}\bigr) = γ^* c_1\bigl(K_X + D\bigr) ∈ H^*\bigl(\what{X},\, ℚ\bigr).
  \]
\end{obs}

\begin{defn}[Total $\cC$-Chern class of $\cC$-cotangent bundle]
  If $(X,D)$ is a compact nc $\cC$-pair, write
  \[
    c\Bigl(Ω¹_{(X,D)}\Bigr) := c\bigl(Ω¹_X\bigr)·\prod_i \Bigl({\textstyle \frac{m_i-1}{m_i}}·c\bigl(𝒪_{D_i}\bigr) + {\textstyle \frac{1}{m_i}}\Bigr) ∈ H^*\bigl(X,\, ℚ\bigr)
  \]
  and refer to this quantity as the \emph{total $\cC$-Chern class}\index{total
  $\cC$-Chern class} of the $\cC$-cotangent bundle for $(X,D)$.
\end{defn}

Chern classes of $\cC$-cotangent bundles have been studied at length in the
literature.  While the surface case has already been considered in Miyaoka's
classic paper \cite[Sect.~3]{Miy08}, generalizations to higher dimensions appear
throughout the recent literature, including \cite[Sect.~2.6]{CDR20},
\cite[Sects.~2 and 3]{MR4504444}, and \cite[Chapt.~4]{Gomez24}.

\subsection{Residue sequences for the $\cC$-cotangent bundle}
\approvals{Erwan & yes\\ Stefan & yes}

In view of Miyaoka's description of the $\cC$-cotangent bundle, it is perhaps
not surprising that the classic residue-- and normal bundle sequences for
logarithmic differentials have direct counterparts in the setting of
$\cC$-pairs.

\begin{obs}[\protect{$\cC$-residue sequence, cf.~\cite[Chapt.~3]{Gomez24}}]%
  In Setting~\ref{set:3-1}, assume that the $q$-morphism $γ$ is adapted.  The
  alternative description of the $\cC$-cotangent sheaf in
  Observation~\ref{obs:3-8} expands to the following commutative diagram with
  exact rows and columns,
  \[
    \begin{tikzcd}
      γ^* Ω¹_X\bigl(\log ⌊D⌋\bigr) \ar[d, equal] \ar[r, hook] & Ω¹_{(X,D,γ)} \ar[d, hook] \ar[r, two heads] & \underset{i \:|\: m_i < ∞}{\bigoplus} \factor{𝒥_{\textstyle\frac{1}{m_i}·γ^* D_i}}{𝒥_{γ^* D_i}} \ar[d, hook] \\
      γ^* Ω¹_X\bigl(\log ⌊D⌋\bigr) \ar[d, two heads] \ar[r, hook] & γ^* Ω¹_X (\log D) \ar[d, two heads] \ar[r, two heads] & \underset{i \:|\: m_i < ∞}{\bigoplus} 𝒪_{γ^*D_i} \ar[d, two heads] \\
      0 \ar[r, hook] & \underset{i \:|\: m_i < ∞}{\bigoplus} 𝒪_{\left⌈{\textstyle\frac{1}{m_i}}·γ^* D_i\right⌉} \ar[r, equal] & \underset{i \:|\: m_i < ∞}{\bigoplus} 𝒪_{\textstyle\frac{1}{m_i}·γ^* D_i}.
    \end{tikzcd}
  \]
  Its top row is called the \emph{$\cC$-residue sequence}\index{C-residue
  sequence@$\cC$-residue sequence} of the pair $(X,D)$ and the $q$-morphism $γ$.
\end{obs}

\subsection{Normal bundle sequences for the $\cC$-cotangent bundle}
\approvals{Erwan & yes\\ Stefan & yes}

In contrast to the $\cC$-residue sequence, the $\cC$-normal bundle sequences are
a little more delicate to write down, because we need to choose compatible
submanifolds $Y ⊊ X$ and $\what{Y} ⊊ \what{X}$.  For simplicity, we stick to the
simplest setting where $Y$ and $\what{Y}$ are of pure codimension one, and
consider the cases where $Y ⊄ \supp D$ and $Y ⊂ \supp D$ separately.  The reader
will observe how the $\cC$-normal bundle sequences in
Observations~\ref{obs:3-25} and \ref{obs:3-26} interpolate between the classic
and the logarithmic case.

\begin{setting}[$\cC$-normal bundle sequence]\label{set:3-24}%
  In Setting~\ref{set:3-1}, assume that the $q$-morphism $γ$ is adapted.  Let $Y
  ⊊ X$ and $\what{Y} ⊆ γ^{-1}(Y) ⊊ \what{X}$ be smooth prime divisors, such that
  $Y + D$ and $\what{Y} + γ^* D$ have nc support in $X$ and $\what{X}$,
  respectively.
\end{setting}

\begin{obs}[$\cC$-Normal bundle sequence, I]\label{obs:3-25}%
  In Setting~\ref{set:3-24}, assume that $Y ⊄ \supp D$.  Write
  \[
    D_Y := D|_Y ∈ ℚ\Div(Y)
  \]
  and observe that $(Y, D_Y)$ is again a $\cC$-pair.  The restricted morphism
  $γ|_{\what{Y}} : \what{Y} → Y$ is again an adapted $q$-morphism for the pair
  $(Y, D_Y)$ that satisfies the assumptions of Setting~\ref{set:3-1}, so that a
  well-defined $\cC$-cotangent bundle $Ω¹_{(Y,D_Y,γ|_{\what{Y}})}$ exists.
  Further, there exists a natural sequence
  \begin{equation}\label{eq:3-25-1}
    \textstyle 0 → 𝒪_{\what{X}} \bigl(γ^*Y\bigr)|_{\what{Y}} → Ω¹_{(X,D,γ)}|_{\what{Y}} → Ω¹_{(Y,D_Y,γ|_{\what{Y}})} → 0.
  \end{equation}
\end{obs}

\begin{obs}[$\cC$-Normal bundle sequence, II]\label{obs:3-26}%
  In Setting~\ref{set:3-24}, assume that $Y ⊆ \supp D$, so that $Y = D_{i_0}$
  for one index $i_0$.  Write
  \[
    \textstyle D_Y := \left.\left(D - \frac{m_{i_0}-1}{m_{i_0}} D_{i_0}\right)\right|_Y ∈ ℚ\Div(Y)
  \]
  and observe that $(Y, D_Y)$ is again a $\cC$-pair.  The restricted morphism
  $γ|_{\what{Y}} : \what{Y} → Y$ is again an adapted $q$-morphism for the pair
  $(Y, D_Y)$ that satisfies the assumptions of Setting~\ref{set:3-1}, so that a
  well-defined $\cC$-cotangent bundle $Ω¹_{(Y,D_Y,γ|_{\what{Y}})}$ exists.  With
  the understanding that $\frac{1}{∞} = 0$, there exists a natural sequence
  \begin{equation}\label{eq:3-26-1}
    \textstyle 0 → 𝒪_{\what{X}} \bigl(\frac{1}{m_{i_0}}·γ^*D_{i_0}\bigr)|_{\what{Y}} → Ω¹_{(X,D,γ)}|_{\what{Y}} → Ω¹_{(Y,D_Y,γ|_{\what{Y}})} → 0.
  \end{equation}
\end{obs}

\begin{notation}[$\cC$-normal bundle sequence]
  We refer to Sequences~\eqref{eq:3-25-1} and \eqref{eq:3-26-1} as
  \emph{$\cC$-normal bundle sequences}\index{C-normal bundle
  sequence@$\cC$-normal bundle sequence} of the pair $(X,D)$.
\end{notation}

% !TEX root = orbiAlb1
%
% Do not edit the following line.  The text is automatically updated by
% subversion.
%
\svnid{$Id: 04-adaptedReflexiveTensors.tex 911 2024-09-29 18:09:56Z kebekus $}
\selectlanguage{british}

\section{Adapted reflexive tensors}
\label{sec:4}
\subversionInfo
\approvals{Erwan & yes\\ Stefan & yes}

Given a $\cC$-pair $(X,D)$ and a $q$-morphism $γ : \what{X} → X$,
Section~\ref{sec:3} defined adapted tensors on $\what{X}$, assuming that the
spaces $X$ and $\what{X}$ are smooth and that the divisors $D$ and $γ^*D$ have
normal crossing support.  While we hope that the reader finds the resulting
notions interesting, we have to admit that the strong smoothness assumptions
limit the theory's usefulness in practise.
\begin{itemize}
  \item Pairs $(X,D)$ that appear in classification and birational geometry are
  hardly ever nc.  Practically relevant pairs will typically be klt and might be
  locally uniformizable at best.

  \item Even if $(X,D)$ is nc, most of the covering spaces that one might
   naturally consider will typically be singular.  Observe that smoothness is not
   preserved by fibre-product constructions.
\end{itemize}
This section extends the constructions of Section~\ref{sec:3} to the singular
case, replacing ``adapted tensors'' by the ``adapted reflexive tensors'' that we
define in the next step.  The construction also generalize the ``sheaves of
reflexive differentials'' of Notation~\ref{not:2-6} that have been useful in
the study of singular varieties that appear in Minimal Model Theory,
\cite{GKKP11, KS18}.

\subsection{Definition and first examples}
\approvals{Erwan & yes\\ Stefan & yes}

The present Section~\ref{sec:4} works in the following setting and uses the
following notation.

\begin{setting}\label{set:4-1}%
  Let $(X,D)$ be a $\cC$-pair as in Definition~\ref{def:2-24}, where $D$ is
  written as $\sum_i \frac{m_i-1}{m_i}·D_i$.  Let $γ : \what{X} → X$ be a
  $q$-morphism.
\end{setting}

\begin{notation}\label{not:4-2}%
  In Setting~\ref{set:4-1}, recall from Reminder~\ref{remi:2-19} that $\img γ ⊆
  X$ is open.  Let $X⁺ ⊆ \img γ$ be the maximal open set such that the pairs
  $(X, D)$ and $\bigl(\what{X}, γ^*D\bigr)$ are nc over $X⁺$.  Set
  \[
    D⁺ := D ∩ X⁺ ∈ ℚ\Div(X⁺) \quad\text{and}\quad \what{X}⁺ := γ^{-1}(X⁺).
  \]
  Observe that the subset $\what{X}⁺ ⊆ \what{X}$ is big and consider the
  restriction $γ⁺ : \what{X}⁺ → X⁺$ and the inclusion $ι : \what{X}⁺ →
  \what{X}$.
\end{notation}

Observe that the pair $(X⁺,D⁺)$ and the morphism $γ⁺$ satisfy the
assumptions made in Setting~\ref{set:3-1} above.  Definition~\ref{def:3-2}
therefore equips us with bundles $\Sym^•_{\cC} Ω^•_{(X⁺, D⁺, γ⁺)}$
defined on $\what{X}⁺$.  We extend these bundles from $\what{X}⁺$ to a
quasi-coherent sheaves that are defined on all of $\what{X}$.

\begin{defn}[Adapted reflexive tensors differentials, compare Definition~\ref{def:3-2}]\label{def:4-3}%
  Assume Setting~\ref{set:4-1}.  Given numbers $n, p ∈ ℕ⁺$, define the
  \emph{sheaf of adapted reflexive $(n,p)$-tensors}\index{adapted!reflexive
  tensor} as
  \[
    \Sym^{[n]}_{\cC} Ω^{[p]}_{(X,D,γ)} := ι_* \Sym^n_{\cC} Ω^p_{(X⁺, D⁺, γ⁺)}.
  \]
  Collectively, we refer to $\Sym^{[•]}_{\cC} Ω^{[•]}_{(X,D,γ)}$ as
  the sheaves of adapted reflexive tensors.
\end{defn}

\begin{defn}[Adapted reflexive differentials, compare Definition~\ref{def:3-3}]\label{def:4-4}%
  Assume Setting~\ref{set:4-1}.  Given a number $p ∈ ℕ⁺$, define the \emph{sheaf
  of adapted reflexive $p$-forms}\index{adapted!reflexive $p$-form} as
  \[
    Ω^{[p]}_{(X,D,γ)} := \Sym^{[1]}_{\cC} Ω^{[p]}_{(X,D,γ)}.
  \]
  Collectively, we refer to $Ω^{[•]}_{(X,D,γ)}$ as the \emph{sheaves of adapted
  reflexive differentials}\index{adapted!reflexive differential}.  The sheaf
  $Ω^{[1]}_{(X,D,γ)}$ is called\index{C-cotangent sheaf@$\cC$-cotangent sheaf}
  \emph{$\cC$-cotangent sheaf}.
\end{defn}

If $(X,D)$ and $\bigl(\what{X}, γ^*D\bigr)$ are nc, then the sheaves of adapted
reflexive tensors agree with the bundles constructed in Definition~\ref{def:3-2}
and are thus locally free.  In general, we show that the sheaves of adapted
reflexive tensors are reflexive.

\begin{prop}[Reflexivity, compare Observation~\ref{obs:3-4}]\label{prop:4-5}%
  The sheaves $\Sym^{[n]}_{\cC} Ω^{[p]}_{(X,D,γ)}$ of Definition~\ref{def:4-3}
  are reflexive.
\end{prop}

The reader coming from algebraic geometry might find the following proof of
Proposition~\ref{prop:4-5} surprisingly complicated.  We recall that in the
analytic setting, vector bundles on big open subsets can generally \emph{not} be
extended to coherent sheaves on the whole space and refer the reader to
\cite[p.~372]{MR0212214} for an elementary yet sobering example.

\begin{proof}[Proof of Proposition~\ref{prop:4-5}]
  It follows from Observation~\ref{obs:2-15} that the open subset $\what{X}⁺ ⊆
  \what{X}$ of Notation~\ref{not:4-2} is big.  With that in place, recall
  \cite[Prop.~7]{MR0212214}: To prove that the sheaves $\Sym^{[n]}_{\cC}
  Ω^{[p]}_{(X,D,γ)}$ are reflexive, it suffices to find coherent sheaves
  $ℱ_{n,p}$ on $\what{X}$ whose restrictions to $\what{X}⁺$ agree with the
  bundles of $(n,p)$-tensors,
  \begin{equation}\label{eq:4-5-1}
    ℱ_{n,p}|_{\what{X}⁺} ≅ \Sym^n_{\cC} Ω^p_{(X⁺, D⁺, γ⁺)}.
  \end{equation}
  In order to construct $ℱ_{n,p}$, choose a strong log resolution\footnote{A
  strong log resolution is a proper morphism $π : Y → X$ where $Y$ is smooth,
  where $π$ is isomorphic over the maximal open set where $(X,D)$ is nc, and
  where the $π$-exceptional set $E ⊂ Y$ and $E ∪ π^{-1}(X ∖ U)$ are both of pure
  codimension one with nc support.} $π : Y → X$ of the pair $(X,D)$ and consider
  the strict transform $Δ := π_*^{-1}D$.  Choosing a strong log resolution of
  the normalized fibre product, we obtain a commutative diagram of dominant
  morphisms between normal analytic varieties,
  \[
    \begin{tikzcd}[column sep=3cm]
      \what{Y} \ar[r, "\what{π}\text{, proper birational}"] \ar[d, "δ"'] & \what{X} \ar[d, "γ"] \\
      Y \ar[r, "π\text{, strong log resolution}"'] & X,
    \end{tikzcd}
  \]
  where $π$ and $\what{π}$ are isomorphic over $X⁺$ and $\what{X}⁺$,
  respectively.  In analogy to Definition~\ref{def:3-2}, write
  \begin{align*}
    𝒜'_{n,p} & := 𝒪_{\what{Y}}\bigl( ⌊n·δ^* \{Δ\}⌋ \bigr) ⊗ δ^* \Sym^n Ω^p_Y \bigl(\log ⌊Δ⌋ \bigr) \\
    ℬ'_{n,p} & := \Sym^n Ω^p_{\what{Y}}(\log γ^* Δ).
  \end{align*}
  Observe that both are subsheaves of $𝒪_{\what{Y}}\bigl(δ^* Δ \bigr) ⊗
  \left(δ^* \Sym^n Ω^p_Y \right)$ and define
  \[
    ℱ'_{n,p} := 𝒜'_{n,p} \:∩\: ℬ'_{n,p}.
  \]
  We can then take $ℱ_{n,p} := \what{π}_* ℱ'_{n,p}$, which is coherent because
  $\what{π}$ is proper.  Since $π$ and $\what{π}$ are isomorphic over $X⁺$ and
  $\what{X}⁺$, condition \eqref{eq:4-5-1} holds by construction.
\end{proof}

In analogy with Example~\ref{ex:3-5}, we highlight a few special cases where the
sheaves of adapted reflexive tensors take a particularly simple form.

\begin{example}[Special cases, compare Example~\ref{ex:3-5}]\label{ex:4-6}%
  In Setting~\ref{set:4-1}, assume that numbers $n, p ∈ ℕ⁺$ are given.
  \begin{enumerate}
    \item\label{il:4-6-1} If $p = \dim X$, then
    \[
      \Sym^{[n]}_{\cC} Ω^{[\dim X]}_{(X,D,γ)} = \Bigl(\bigl(γ^{[*]} ω_X^{⊗ n}\bigr) ⊗ 𝒪_{\what{X}}\bigl(⌊n·γ^* D⌋\bigr)\Bigr)^{\vee\vee}.
    \]

    \item If $D = 0$, then $\Sym^{[n]}_{\cC} Ω^{[p]}_{(X,0,γ)} = γ^{[*]}
    \Sym^{[n]} Ω^{[p]}_X$.
  
    \item If $γ = \Id_X$, then $Ω^{[p]}_{(X,D,\Id_X)} = Ω^{[p]}_X\bigl(\log
    ⌊D⌋\bigr)$.
  
    \item If $γ$ is strongly adapted and $\Branch(γ) ⊆ \supp D$, then
    \[
      \Sym^{[n]}_{\cC} Ω^{[p]}_{(X,D,γ)} = \Sym^{[n]} Ω^{[p]}_{\what{X}}\bigl(\log γ^*⌊D⌋\bigr).
    \]

    \item\label{il:4-6-5} If $γ$ uniformizes, then
    \[
      \Sym^{[n]}_{\cC} Ω^{[p]}_{(X,D,γ)} = \Sym^n Ω^p_{\what{X}}\bigl(\log γ^*⌊D⌋\bigr).
    \]
  \end{enumerate}
\end{example}

\begin{rem}[Reflexive hull in \ref{il:4-6-1}]
  The double dual on the right side of \ref{il:4-6-1} is necessary, as the
  tensor product of two reflexive sheaves will generally not be reflexive and
  might even contain torsion.  If a canonical divisor $K_X$ exists on $X$, then
  \ref{il:4-6-1} simplifies to
  \[
    \Sym^{[n]}_{\cC} Ω^{[\dim X]}_{(X,D,γ)} = 𝒪_{\what{X}}\bigl( ⌊n·γ^* (K_X+D) ⌋ \bigr).
  \]
\end{rem}

\subsection{Inclusions}
\approvals{Erwan & yes\\ Stefan & yes}

By construction, the observations in Section~\ref{sec:3} have direct analogues
for the sheaves of adapted reflexive tensors.  For later reference and for the
reader's convenience, we include full statements, even though the text does
become somewhat repetitive and perhaps a little tiring.

\begin{obs}[Inclusions, compare Observation~\ref{obs:3-9}]\label{obs:4-8}%
  Assume Setting~\ref{set:4-1}.  Given numbers $n, p ∈ ℕ⁺$, there exist natural
  inclusions as follows,
  \[
    \begin{tikzcd}
      \Sym^{[n]} Ω^{[p]}_{\what{X}} \bigl(\log γ^*⌊D⌋ \bigr) \ar[r, hook] & \Sym^{[n]} Ω^{[p]}_{\what{X}} \bigl(\log γ^*D \bigr) \\
      & γ^{[*]} \Sym^{[n]} Ω^{[p]}_X(\log D) \ar[u, hook] \\
      \Sym^{[n]}_{\cC} Ω^{[p]}_{(X,D,γ)} \ar[r, equal] \ar[uu, hook, "ι_{n,p}"] & \Sym^{[n]}_{\cC} Ω^{[p]}_{(X,D,γ)} \ar[u, hook] \\
      \Sym^{[n]} Ω^{[p]}_{(X,D,γ)} \ar[r, equal] \ar[u, hook] & \Sym^{[n]} Ω^{[p]}_{(X,D,γ)} \ar[u, hook] \\
      γ^{[*]} \Sym^{[n]} Ω^{[p]}_X \ar[r, equal] \ar[u, hook] & γ^{[*]} \Sym^{[n]} Ω^{[p]}_X.  \ar[u, hook]
    \end{tikzcd}
  \]
  Note that all sheaves here are subsheaves of the quasi-coherent sheaf of
  meromorphic reflexive tensors on $\what{X}$, that is, $ℳ_{\what{X}} ⊗
  \Sym^{[n]} Ω^{[p]}_{\what{X}}$.
\end{obs}
\begin{proof}
  Like nearly every other statement in the remainder of the present
  Section~\ref{sec:4}, the proof follows from the observation that the subset
  $\what{X}⁺ ⊆ \what{X}$ introduced in Setting~\ref{set:4-1} and
  Notation~\ref{not:4-2} is big and that $\what{X}$ is normal.  If $𝒜$ and
  $ℬ$ are reflexive sheaves on $\what{X}$, this implies that the natural
  restriction morphisms
  \[
    H⁰\bigl( \what{X},\, 𝒜 \bigr) → H⁰\bigl( \what{X}⁺,\, 𝒜|_{\what{X}⁺} \bigr)
    \quad\text{and}\quad
    \Hom_{\what{X}}\bigl( 𝒜,\, ℬ \bigr) → \Hom_{\what{X}⁺}\bigl( 𝒜|_{\what{X}⁺} ,\, ℬ|_{\what{X}⁺} \bigr)
  \]
  are isomorphic, \cite[Cor.~3.15]{BS76} and \cite{MR0212214}.
  Observation~\ref{obs:3-9} therefore implies the claim.
\end{proof}

\begin{obs}[Uniformization, compare Observation~\ref{obs:3-10}]\label{obs:4-9}%
  Assume that the morphism $γ : \what{X} → X$ of Setting~\ref{set:4-1} is an
  adapted cover and that the pair $\bigl(\what{X}, (γ^* ⌊D⌋)_{\reg} \bigr)$ is
  nc.  Consider the morphisms
  \begin{equation}\label{eq:4-9-1}
    ι_{•,•} : \Sym^{[•]}_{\cC} Ω^{[•]}_{(X,D,γ)} ↪ \Sym^{[•]}Ω^{[•]}_{\what{X}} \bigl(\log γ^*⌊D⌋ \bigr)
  \end{equation}
  of Observation~\ref{obs:4-8}.  Then, the following statements are equivalent.
  \begin{enumerate}
    \item\label{il:4-9-2} The morphism $γ$ is a uniformization.
  
    \item\label{il:4-9-3} There exist numbers $1 ≤ p ≤ \dim X$ and $1 ≤ n$ such
    that $ι_{n,p}$ is isomorphic.
  
    \item\label{il:4-9-4} For every pair of numbers $p, n ∈ ℕ⁺$, the inclusion
    $ι_{n,p}$ is isomorphic.
  \end{enumerate}
\end{obs}
\begin{proof}
  By construction of the morphism \eqref{eq:4-9-1}, Items~\ref{il:4-9-3} and
  Items~\ref{il:4-9-4} is equivalent the analogous statements for the morphisms
  \[
    ι⁺_{•,•} : \Sym^{•}_{\cC} Ω^{•}_{(X⁺,D⁺,γ⁺)} ↪ \Sym^{•}Ω^{•}_{\what{X}⁺} \bigl(\log (γ⁺)^*⌊D⁺⌋ \bigr)
  \]
  discussed in Observation~\ref{obs:3-9}.  Observation~\ref{obs:3-10} therefore
  implies that Items~\ref{il:4-9-3} and \ref{il:4-9-4} are each equivalent to
  the assertion that $γ⁺$ uniformizes.  On the other hand, the assumption that
  $\bigl(\what{X}, (γ^* ⌊D⌋)_{\reg} \bigr)$ is nc allows
  reformulating~\ref{il:4-9-2} as follows,
  \begin{align*}
    γ \text{ uniformizes } (X,D) & ⇔ \Branch γ ⊆ \supp D \text{ and } γ \text{ strongly adapted} && \text{Def.~\ref{def:2-28}}\\
    & ⇔ \Branch γ⁺ ⊆ \supp D⁺ \text{ and } γ⁺ \text{strongly adapted} && \what{X}⁺ ⊆ \what{X}\text{ big}\\
    & ⇔ γ⁺ \text{ uniformizes } (X⁺,D⁺) && \text{Def.~\ref{def:2-28}}.
  \end{align*}
  For the second equivalence, recall from Notation~\ref{not:2-22} that $\Branch
  γ$ refers to the branch \emph{divisor} and not to the branch \emph{locus},
  which might contain components of high codimension.
\end{proof}

\subsection{Operations}
\approvals{Erwan & yes\\ Stefan & yes}

With the minor difference highlighted in Warning~\ref{warning:4-13} below, the
operations on adapted tensors introduced in Section~\ref{sec:3-5} extend to
identical operations on adapted reflexive tensors.

\begin{obs}[Reflexive wedge products and exterior derivatives, compare Observation~\ref{obs:3-12}]\label{obs:4-10}%
  In Setting~\ref{set:4-1}, observe that the subsheaves
  \[
    Ω^{[•]}_{(X,D,γ)} \overset{ι_{1,•}}{⊆} Ω^{[•]}_{\what{X}} \bigl(\log γ^*⌊D⌋ \bigr)
  \]
  are closed under reflexive wedge products and exterior derivations.  Given
  numbers $p$ and $q$, we obtain natural operations
  \[
    Λ : Ω^{[p]}_{(X,D,γ)} ⨯ Ω^{[q]}_{(X,D,γ)} → Ω^{[p+q]}_{(X,D,γ)}
    \quad\text{and}\quad
    d : Ω^{[p]}_{(X,D,γ)} → Ω^{[p+1]}_{(X,D,γ)}.
    \eqno \qed
  \]
\end{obs}

\begin{obs}[Reflexive symmetric multiplication, compare Observation~\ref{obs:3-13}]\label{obs:4-11}%
  In Setting~\ref{set:4-1}, observe that the subsheaves
  \[
    \Sym^{[•]}_{\cC} Ω^{[•]}_{(X,D,γ)} ⊆ \Sym^{[•]} Ω^{[•]}_{\what{X}} \bigl(\log γ^*⌊D⌋ \bigr)
  \]
  are closed under symmetric multiplication.  Given numbers $p$, $n_1$, $n_2 ∈
  ℕ⁺$, we obtain natural maps
  \[
    \Sym^{[n_1]}_{\cC} Ω^{[p]}_{(X,D,γ)} ⨯ \Sym^{[n_2]}_{\cC} Ω^{[p]}_{(X,D,γ)} → \Sym^{[n_1+n_2]}_{\cC} Ω^{[p]}_{(X,D,γ)}
  \]
  and
  \[
    \Sym^{[n_1]} \Sym^{[n_2]}_{\cC} Ω^{[p]}_{(X,D,γ)} ↪ \Sym^{[n_1·n_2]}_{\cC} Ω^{[p]}_{(X,D,γ)}.
    \eqno \qed
  \]
\end{obs}

\begin{obs}[Adapted reflexive tensors on adapted covers, compare Observation~\ref{obs:3-14}]\label{obs:4-12}%
  If the $q$-morphism $γ$ of Setting~\ref{set:4-1} is adapted, then $γ^* \{D\}$
  is integral and
  \[
    \Sym^{[n]}_{\cC} Ω^p_{(X,D,γ)} = \Sym^{[n]} Ω^{[p]}_{(X,D,γ)}\quad
    \text{for all } p,n ∈ ℕ⁺.
    \eqno \qed
  \]
\end{obs}

\begin{warning}[No surjectivity in Observation~\ref{obs:4-12}]\label{warning:4-13}%
  The sheaves $Ω^{[•]}_{(X,D,γ)}$ in Observation~\ref{obs:4-12} need not be
  locally free.  The natural morphisms
  \[
    \Sym^• Ω^{[•]}_{(X,D,γ)} → \Sym^{[•]} Ω^{[•]}_{(X,D,γ)}
  \]
  are neither injective nor surjective in general; notice that the left side
  might well contain torsion!  In contrast to Observation~\ref{obs:3-13}, we
  cannot conclude that the symmetric multiplication maps of
  Observation~\ref{obs:4-10} are surjective.
\end{warning}

\subsection{Functoriality}
\approvals{Erwan & yes\\ Stefan & yes}

The functoriality statements of Section~\ref{sec:3-6} also have direct
analogues.  In line with Warning~\ref{warning:4-13} above, there is a caveat
here, stemming from the fact that the reflexive hull construction does not
commute with pull-back.  We highlight this issue in Warning~\vref{warning:4-16},
as it will become central when we define and discuss morphisms of $\cC$-pairs in
Section~\ref{sec:7}ff, in the second part of this paper.

\begin{obs}[Functoriality in $q$-morphisms, compare Observation~\ref{obs:3-15}]\label{obs:4-14}%
  In Setting~\ref{set:4-1}, assume that the morphism $γ$ factors into a sequence of
  $q$-morphisms,
  \[
    \begin{tikzcd}
      \what{X} \ar[r, "α"'] \ar[rr, bend left=20, "γ"] & \what{Y} \ar[r, "β"'] & X.
    \end{tikzcd}
  \]
  Given numbers $n, p ∈ ℕ⁺$, there exists a commutative diagram as follows,
  \[
    \begin{tikzcd}[column sep=4cm, baseline=(current bounding box.south east)]
      α^{[*]} \Sym^{[n]} Ω^{[p]}_{\what{Y}}( \log β^* ⌊D⌋) \ar[r, hook, "\text{equals $\diff α$ over nc locus}"] & \Sym^{[n]} Ω^{[p]}_{\what{X}}\bigl( \log γ^* ⌊D⌋\bigr) \\
      α^{[*]} \Sym^{[n]}_{\cC} Ω^{[p]}_{(X,D,β)} \ar[r, hook] \ar[u, hook, "\text{inclusion }α^{[*]}(ι_{n,p})"] & \Sym^{[n]}_{\cC} Ω^{[p]}_{(X,D,γ)}.  \ar[u, hook, "\text{inclusion }ι_{n,p}"']
    \end{tikzcd}
    \eqno \qed
  \]
\end{obs}

\begin{obs}[Functoriality in adapted morphisms, compare Observation~\ref{obs:3-16}]\label{obs:4-15}%
  If the morphism $β$ of Observation~\ref{obs:4-14} is adapted for $(X,D)$,
  then $γ$ is likewise adapted.  The natural morphisms $α^{[*]} \Sym^{[•]}_{\cC}
  Ω^{[•]}_{(X,D,β)} ↪ \Sym^{[•]}_{\cC} Ω^{[•]}_{(X,D,γ)}$ are isomorphic in this
  case.  \qed
\end{obs}

\begin{warning}[Reflexive pull-back in the functoriality statement]\label{warning:4-16}%
  In the setting of Observation~\ref{obs:4-14}, there exist natural sheaf
  morphisms
  \begin{equation}\label{eq:4-16-1}
    α^* \Sym^{[•]}_{\cC} Ω^{[•]}_{(X,D,β)} → α^{[*]} \Sym^{[•]}_{\cC} Ω^{[•]}_{(X,D,β)}
  \end{equation}
  that are however neither injective nor surjective in general; notice that $α^*
  \Sym^{[•]}_{\cC} Ω^{[•]}_{(X,D,β)}$ might well contain torsion!  This will
  become important.  For later reference, we note a few settings where
  \eqref{eq:4-16-1} is isomorphic indeed.
  \begin{enumerate}
    \item\label{il:4-16-2} The morphism \eqref{eq:4-16-1} is isomorphic if
    $\Sym^{[•]}_{\cC} Ω^{[•]}_{(X,D,β)}$ is locally free.

    \item The morphism \eqref{eq:4-16-1} is isomorphic if $α$ is flat.
  \end{enumerate}
  On the positive side, observe that if the $q$-morphism $β$ is adapted and
  $Ω^{[•]}_{(X,D,β)}$ is locally free, then
  \[
    \Sym^{[•]}_{\cC} Ω^{[•]}_{(X,D,β)} \overset{\text{Obs.~\ref{obs:4-12}}}{=}
    \Sym^{[•]} Ω^{[•]}_{(X,D,β)} \overset{\text{loc.~free}}{=}
    \Sym^• Ω^{[•]}_{(X,D,β)}
  \]
  is locally free, so that Item~\ref{il:4-16-2} applies.
\end{warning}

As in Section~\ref{sec:3-6}, we note that a partial converse of
Observation~\ref{obs:4-14} holds true.

\begin{obs}[Test for adapted reflexive tensors, compare Observation~\ref{obs:3-17}]\label{obs:4-17}%
  In the setting of Observation~\ref{obs:4-14}, let $n$ and $p ∈ ℕ⁺$ be two
  numbers, let $U ⊆ \img α ⊆ \what{Y}$ be open and let $σ ∈ \Bigl(ℳ_{\what{Y}} ⊗
  \Sym^{[n]}Ω^{[p]}_{\what{Y}}\Bigr)(U)$ be any meromorphic reflexive tensor on
  $U$.  Then, the following are equivalent.
  \begin{enumerate}
    \item The section $σ$ is an adapted reflexive tensor.  More precisely: the
    meromorphic reflexive tensor $σ$ is a section of the subsheaf
    \[
      \Sym^{[n]}_{\cC} Ω^{[p]}_{(X,D,β)} ⊆ ℳ_{\what{Y}} ⊗ \Sym^{[n]}Ω^{[p]}_{\what{Y}}.
    \]

    \item The pull-back of $σ$ is an adapted reflexive tensor.  More precisely:
    the meromorphic reflexive tensor $(\diff α)(σ)$ is a section of the subsheaf
    \[
      \Sym^{[n]}_{\cC} Ω^{[p]}_{(X,D,γ)} ⊆ ℳ_{\what{Y}} ⊗ \Sym^{[n]} Ω^{[p]}_{\what{X}}.  \eqno \qed
    \]
  \end{enumerate}
\end{obs}

\begin{consequence}[Trace morphism, compare Consequence~\ref{cons:3-18}]\label{cons:4-18}%
  In the setting of Observation~\ref{obs:4-14}, the trace map
  \[
    α_* Ω^{[•]}_{\what{X}}\bigl( \log γ^* D\bigr)
    \xrightarrow{\trace} Ω^{[•]}_{\what{Y}}( \log β^* D)
  \]
  maps adapted reflexive differentials to adapted reflexive differentials.  More
  precisely, there exist commutative diagrams as follows,
  \[
    \begin{tikzcd}[column sep=2cm, baseline=(current bounding box.south east)]
      α_* Ω^{[•]}_{(X,D,γ)} \ar[r, hook, "\text{Obs.~\ref{obs:4-8}}"] \ar[d, "\text{restr.~of }\trace"'] & α_* Ω^{[•]}_{\what{X}}\bigl( \log γ^* D\bigr) \ar[d, "\trace"] \\
      Ω^{[•]}_{(X,D,β)} \ar[r, hook, "\text{Obs.~\ref{obs:4-8}}"'] & Ω^{[•]}_{\what{Y}}( \log β^* D)
    \end{tikzcd}
    \eqno \qed
  \]
\end{consequence}

\subsection{Galois linearization}
\approvals{Erwan & yes\\ Stefan & yes}

Unsurprisingly, the linearization morphisms discussed in Section~\ref{sec:3-7}
also extend from $\what{X}⁺$ to $\what{X}$.

\begin{obs}[Linearisation, compare Observation~\ref{obs:3-19}]\label{obs:4-19}%
  Assume Setting~\ref{set:4-1} and write $G := \Aut_{𝒪}\bigl(\what{X}/X\bigr)$
  for the relative automorphism group.  Then, all sheaves $\Sym^{[•]}_{\cC}
  Ω^{[•]}_{(X,D,γ)}$ of Definition~\ref{def:4-3} carry natural
  $G$-linearisations that are compatible with the natural $G$-linearisations of
  \begin{gather*}
    \Sym^{[n]} γ^{[*]} Ω^{[p]}_X, \quad γ^{[*]} \Sym^{[n]} Ω^{[p]}_X(\log D), \quad \Sym^{[n]} Ω^{[p]}_{\what{X}} \bigl(\log γ^*D \bigr), \\
    \Sym^{[n]} Ω^{[p]}_{\what{X}} \bigl(\log γ^*⌊D⌋ \bigr).
  \end{gather*}
  The inclusions of Observation~\ref{obs:4-8} are morphisms of $G$-sheaves.
  \qed
\end{obs}

\begin{lem}[Invariant push-forward, compare Lemma~\ref{lem:3-20}]\label{lem:4-20}%
  In the setting of Observation~\ref{obs:4-14}, assume that the $q$-morphisms
  $α$, $β$ and $γ$ are covers, and that $α$ is Galois with group $G$.  Given
  numbers $n$, $p ∈ ℕ⁺$, the natural morphism between $G$-invariant push-forward
  sheaves induced by the bottom row of the commutative diagram in
  Observation~\ref{obs:4-14},
  \begin{equation}\label{eq:4-20-1}
    \Sym^{[n]}_{\cC} Ω^{[p]}_{(X,D,β)} = \Bigl( α_* α^{[*]} \Sym^{[n]}_{\cC} Ω^{[p]}_{(X,D,β)} \Bigr)^G ↪ \Bigl( α_* \Sym^{[n]}_{\cC} Ω^{[p]}_{(X,D,γ)} \Bigr)^G,
  \end{equation}
  is isomorphic.
\end{lem}
\begin{proof}
  Lemma~\ref{lem:3-20} guarantees that the sheaf morphism \eqref{eq:4-20-1} is
  isomorphic over the big open set $\what{X}⁺ ⊆ \what{X}$.  On the other
  hand, recall\footnote{The paper \cite{GKKP11} formulates this result in the
  algebraic setting.  The proof of \cite[Lem.~A.4]{GKKP11} works without change
  also for normal analytic varieties.} from \cite[Lem.~A.4]{GKKP11} that the
  invariant push-forward sheaves in \eqref{eq:4-20-1} are both reflexive.  The
  isomorphism will therefore extend from $\what{X}⁺$ to $\what{X}$.
\end{proof}

% !TEX root = orbiAlb1
%
% Do not edit the following line.  The text is automatically updated by
% subversion.
%
\svnid{$Id: 05-pullBack.tex 911 2024-09-29 18:09:56Z kebekus $}
\selectlanguage{british}

\section{Pull-back over uniformizable pairs}
\label{sec:5}

\subsection{Motivation}
\approvals{Erwan & yes \\ Stefan & yes}
\label{sec:5-1}

This section establishes pull-back properties of adapted reflexive tensors;
these will be instrumental when we define ``morphisms of $\cC$-pairs'' in
Sections~\ref{sec:7}--\ref{sec:8} below.  To motivate the somewhat technical
discussion, let us first recall Reid's construction of the Albanese for
projective varieties with rational singularities, \cite[Prop.~2.3]{Reid83a} and
\cite[Sect.~2.4]{BS95}.

\begin{reminder}[Albanese for algebraic varieties with rational singularities]
  Given a complex, projective variety $\what{X}$ with rational singularities,
  Reid considers a resolution of singularities, $φ : Y → \what{X}$ and takes the
  Albanese of $Y$.  The assumption that $\what{X}$ has rational singularities
  implies that all $1$-differentials on $Y$ are trivial on $φ$-fibres.  This in
  turn yields a factorization,
  \begin{equation}\label{eq:5-1-1}
    \begin{tikzcd}[column sep=2cm]
      Y \ar[d, "φ\text{, resolution}"'] \ar[r, "\alb(Y)"] & \Alb(Y) \\
      \what{X}, \ar[ur, bend right=10, "∃!"']
    \end{tikzcd}
  \end{equation}
  and shows that $\Alb(Y)$ does not depend on the choice of the resolution.  It
  is therefore reasonable to take $\Alb(Y)$ as the Albanese of $\what{X}$.
\end{reminder}

For the forthcoming construction of an ``Albanese for $\cC$-pairs'', we would
like to emulate Reid's argument in a setting where $\what{X}$ is a cover of a
locally uniformizable $\cC$-pair $(X,D)$.  But covers need not have rational
singularities, so that we cannot expect a factorization of the Albanese as in
\eqref{eq:5-1-1} above!  We will however show that differentials on $Y$ that are
\emph{adapted} outside the $φ$-exceptional locus are trivial on $φ$-fibres.  More
generally, we show that any adapted differential on
\[
  \what{X}_{\reg} ≅ Y ∖ φ\text{-exceptional set}
\]
extends to a differential form on $Y$ that is trivial on $φ$-fibres.  To give an
adapted differential on $\what{X}_{\reg}$ it is of course equivalent to give an
adapted reflexive differential $\what{X}$.  By the end of the day, we will thus
construct a ``pull-back map''
\[
  H⁰ \bigl(\what{X},\, Ω^{[1]}_{(X,D,•)} \bigr) → H⁰\bigl( Y,\, Ω¹_Y \bigr).
\]
This section aims to construct pull-back maps more generally, for arbitrary
tensors and arbitrary morphisms $φ$ from manifolds to $\what{X}$.

\subsection{Main results}
\approvals{Erwan & yes \\ Stefan & yes}
\label{sec:5-2}

To formulate our results precisely and to set the stage for the remainder of the
present section, consider the following situation.

\begin{setting}[Smooth space mapping to cover of $X$]\label{set:5-2}%
  Let $(X, D_X)$ be a locally uniformizable $\cC$-pair.  Let $(Y, D_Y)$ be a log
  pair.  Assume that $(Y, D_Y)$ is nc and consider a sequence of morphisms
  \[
    \begin{tikzcd}[column sep=2cm]
      Y \arrow[r, "φ"] & \what{X} \arrow[r, "γ\text{, $q$-morphism}"] & X,
    \end{tikzcd}
  \]
  where $\supp φ^* γ^* ⌊D_X⌋ ⊆ \supp D_Y$.
\end{setting}

\begin{rem}
  We do not assume that the variety $\what{X}$ of Setting~\ref{set:5-2} is
  smooth.  The morphism $φ$ may take its image in the singular locus of
  $\what{X}$.
\end{rem}

\CounterStep{}Maintain Setting~\ref{set:5-2}.  Following ideas and methods of
\cite{MR3084424}, we aim to construct ``natural'' pull-back morphisms
\begin{equation}\label{eq:5-4-1}
  d_{\cC} φ : φ^* \Sym^{[•]}_{\cC} Ω^{[•]}_{(X,D_X,γ)} → \Sym^• Ω^•_Y(\log D_Y)
\end{equation}
that compare adapted reflexive tensors on $\what{X}$ with logarithmic Kähler
tensors on the manifold $Y$.

\begin{rem}[Pull-back for a uniformized variety]%
  Assume Setting~\ref{set:5-2}.  If the $q$-morphism $γ$ is a uniformization,
  then Item~\ref{il:4-6-5} of Example~\ref{ex:4-6} identifies adapted reflexive
  tensors on $\what{X}$ with logarithmic Kähler tensors,
  \[
    \Sym^{[•]}_{\cC} Ω^{[•]}_{(X,D_X,γ)} = \Sym^• Ω^•_{\what{X}}(\log γ^* ⌊D_X⌋).
  \]
  The most natural choice for the pull-back morphisms \eqref{eq:5-4-1} is the
  standard pull-back of logarithmic Kähler differentials and tensors.
\end{rem}

\begin{rem}[Optimality and possible generalizations]
  This section works in Setting~\ref{set:5-2}, where $(X,D)$ is locally
  uniformizable.  If one is interested only in adapted reflexive differentials
  rather than adapted reflexive tensors, it is conceivable that a pull-back
  morphism as in \eqref{eq:5-4-1} exists under less restrictive conditions.  We
  discuss possible generalizations in Section~\ref{sec:15-1} near the end of
  this paper.
\end{rem}

\subsection{Construction of pull-back maps in the uniformizable case}
\approvals{Erwan & yes \\ Stefan & yes}

We begin with an explicit construction for a pull-back morphism, at least in the
setting where $X$ is uniformizable.

\begin{construction}[Pull-back of sections in the uniformizable case]\label{cons:5-7}%
  In Setting~\ref{set:5-2}, assume that $(X, D_X)$ is uniformizable.  Choose a
  uniformization $u : X_u \twoheadrightarrow X$ and consider a
  diagram
  \[
    \begin{tikzcd}[column sep=3cm, row sep=1cm]
      \wcheck{Y}⁺ \arrow[r, "\wcheck{φ}"] \ar[d, two heads, "s⁺"'] & \wcheck{X} \arrow[d, two heads, "t"] \ar[r, "\wcheck{γ}"] & X_u \ar[d, two heads, "u\text{, uniformization}"]\\
      Y⁺ \arrow[r, "φ⁺ := φ|_{Y⁺}"'] & \what{X} \arrow[r, "γ\text{, $q$-morphism}"'] & X,
    \end{tikzcd}
  \]
  constructed as follows.
  \begin{itemize}
    \item Choose a component $\wcheck{X}$ of the normalized fibre product $X_u
    ⨯_X \what{X}$.

  \item Choose a component $\wcheck{Y}$ of the normalized fibre product $Y
    ⨯_{\what{X}} \wcheck{X}$ and denote the natural morphism by $s : \wcheck{Y}
    → Y$

  \item Let $Y⁺ ⊆ Y$ be the maximal open set over which $(\wcheck{Y}, s^* D_Y)$
  is nc and denote the preimage by $\wcheck{Y}⁺ := s^{-1}(Y⁺)$.
  \end{itemize}
  To begin the construction in earnest, observe that there are natural
  morphisms,
  \begin{align}
    (s⁺)^*(φ⁺)^* \Sym^{[•]}_{\cC} Ω^{[•]}_{(X,D_X,γ)} & = \wcheck{φ}^*t^* \Sym^{[•]}_{\cC} Ω^{[•]}_{(X,D_X,γ)} && \text{commutativity} \nonumber \\
    & → \wcheck{φ}^* t^{[*]} \Sym^{[•]}_{\cC} Ω^{[•]}_{(X,D_X,γ)} && \text{natural} \label{eq:5-7-1} \\
    & → \wcheck{φ}^* \Sym^{[•]}_{\cC} Ω^{[•]}_{(X,D_X,γ◦t)} && \text{Observation~\ref{obs:4-14}.} \nonumber
  \intertext{Secondly, recall}
    \Sym^{[•]}_{\cC} Ω^{[•]}_{(X,D_X,γ◦t)} &= \Sym^{[•]}_{\cC} Ω^{[•]}_{(X,D_X,u◦\wcheck{γ})} && \text{commutativity} \nonumber \\
    & = \wcheck{γ}^{[*]} \Sym^{[•]}_{\cC} Ω^{[•]}_{(X,D_X,u)} && \text{Observation~\ref{obs:4-15}} \label{eq:5-7-2} \\
    & = \wcheck{γ}^{*} \Sym^• Ω^•_{X_u}(\log u^* ⌊D_X⌋) && \text{Example~\ref{ex:4-6}}.  \nonumber
  \intertext{As a consequence, pull-back of logarithmic Kähler tensors yields natural embeddings}
    \wcheck{φ}^* \Sym^{[•]}_{\cC} Ω^{[•]}_{(X,D_X,γ◦t)} & = (\wcheck{γ}◦\wcheck{φ})^{*} \Sym^• Ω^•_{X_u}(\log u^* ⌊D_X⌋) && \text{\eqref{eq:5-7-2}} \nonumber \\
    & ↪ \Sym^• Ω^•_{\wcheck{Y}⁺}(\log s^* D_Y) && \text{pull-back} \label{eq:5-7-3} \\
    & = (s⁺)^* \Sym^• Ω^•_{Y⁺}(\log D_Y) && \text{branching of }s⁺.  \nonumber
  \intertext{Combining \eqref{eq:5-7-1} with \eqref{eq:5-7-3} and taking push-forward, we obtain maps as follows,}
    (φ⁺)^* \Sym^{[•]}_{\cC} Ω^{[•]}_{(X,D_X,γ)} & → (s⁺)_*(s⁺)^*(φ⁺)^* \Sym^{[•]}_{\cC} Ω^{[•]}_{(X,D_X,γ)} && \text{natural} \nonumber \\
    & → (s⁺)_*(s⁺)^* \Sym^• Ω^•_{Y⁺}(\log D_Y) && \text{\eqref{eq:5-7-3}}◦\text{\eqref{eq:5-7-1}} \nonumber \\
    & → \Sym^• Ω^•_{Y⁺}(\log D_Y) && \text{trace}.  \nonumber
  \end{align}
  Given that $Ω^•_Y(\log D_Y)$ is locally free and $Y⁺ ⊆ Y$ is big, these maps
  extend to the desired pull-back morphisms $d_{\cC}φ$ of the form promised in
  \eqref{eq:5-4-1} above.
\end{construction}

We leave the proof of the following fact to the reader.

\begin{fact}[Canonicity]
  The pull-back morphisms $d_{\cC} φ$ of Construction~\ref{cons:5-7} do not
  depend on any of the choices made in the construction.  \qed
\end{fact}

\subsection{Construction of pull-back maps in general}
\label{sec:5-4}
\approvals{Erwan & yes \\ Stefan & yes}

Construction~\ref{cons:5-7} evidently commutes with restrictions to open subsets
of domain and target, which allows extending the setup from the uniformizable to
the locally uniformizable case.

\begin{fact}[Pull-back over locally uniformizable pairs]\label{fact:5-9}%
  In Setting~\ref{set:5-2}, there exist unique sheaf morphisms
  \[
    d_{\cC} φ : φ^* \Sym^{[•]}_{\cC} Ω^{[•]}_{(X,D_X,γ)} → \Sym^• Ω^•_Y(\log D_Y)
  \]
  such that for every uniformizable open subset $X⁺ ⊆ X$ with preimages
  $\what{X}⁺ ⊆ \what{X}$ and $Y⁺ ⊆ Y$, the restrictions $d_{\cC} φ|_{Y⁺}$ agree
  with the pull-back morphisms $d_{\cC} \bigl(φ|_{Y⁺}\bigr)$ of
  Construction~\ref{cons:5-7}.  \qed
\end{fact}

\begin{defn}[Pull-back over locally uniformizable pairs]
  We refer to the pull-back morphisms $d_{\cC} •$ of Fact~\ref{fact:5-9} as the
  \emph{pull-back for adapted reflexive tensors over the locally uniformizable
  pair $(X, D_X)$}\index{pull-back for adapted reflexive tensors}.
\end{defn}

\subsection{Universal properties}
\approvals{Erwan & yes \\ Stefan & yes}
\label{sec:5-5}

Construction~\ref{cons:5-7} enjoys a number of fairly obvious properties whose
proofs are conceptually easy, but lengthy to write down.  To keep the size of
this already long paper within reason, we leave the proofs of the following
facts to the reader.

\begin{fact}[Compatibility with Kähler differentials]\label{fact:5-11}%
  In Setting~\ref{set:5-2}, let $σ ∈ H⁰\bigl( \what{X},\, \Sym^n
  Ω^p_{\what{X}}\bigr)$ be a Kähler tensor, with associated reflexive tensor
  $σ_r ∈ H⁰\bigl( \what{X},\, \Sym^{[n]} Ω^{[p]}_{\what{X}}\bigr)$.  If $σ_r$ is
  adapted, then the composed morphism
  \begin{align*}
    H⁰\bigl( \what{X},\, \Sym^{[n]}_{\cC} Ω^{[p]}_{(X,D_X,γ)}\bigr) & \xrightarrow{φ^*} H⁰\bigl( Y,\, φ^* \Sym^{[n]}_{\cC} Ω^{[p]}_{(X,D_X,γ)}\bigr) \\
    & \xrightarrow{H⁰\bigl( d_{\cC} φ \bigr)} H⁰\bigl( Y,\, \Sym^n Ω^p_Y(\log D_Y)\bigr)
  \end{align*}
  maps the adapted reflexive tensor $σ_r$ to
  \[
    (dφ) (σ) ∈ H⁰\bigl( Y,\, \Sym^n Ω^p_Y \bigr) ⊆ H⁰\bigl( Y,\, \Sym^n Ω^p_Y(\log D_Y)\bigr).  \eqno \qed
  \]
\end{fact}

To avoid any potential confusion, we recall that the assumption ``$σ_r$
adapted'' in Fact~\ref{fact:5-11} is equivalent to the assumption that $σ_r$ is
contained in the subspace of adapted reflexive tensors,
\[
  σ_r ∈ H⁰\bigl( \what{X},\, \Sym^{[n]}_{\cC} Ω^{[p]}_{(X,D_X,γ)}\bigr) ⊆ H⁰\bigl( \what{X},\, \Sym^{[n]} Ω^{[p]}_{\what{X}}\bigr).
\]
The term $(dφ) (σ)$ is the classic pull-back of Kähler tensors.

\begin{fact}[Functoriality]\label{fact:5-12}
  Let $(X, D_X)$ be a locally uniformizable $\cC$-pair and let $(Y_•, D_{Y_•})$
  be nc log pairs.  Assume that a commutative diagram of the following form is
  given,
  \[
    \begin{tikzcd}[column sep=2cm]
      Y_1 \ar[r, "φ_1"] \ar[d, "α"'] & \what{X}_1 \ar[r, "γ_1\text{, $q$-morphism}"] \ar[d, "β"] & X \ar[d, equal] \\
      Y_2 \ar[r, "φ_2"'] & \what{X}_2 \ar[r, "γ_2\text{, $q$-morphism}"'] & X,
    \end{tikzcd}
  \]
  where
  \[
    \supp φ_•^* γ_•^* ⌊D_X⌋ ⊆ \supp D_{Y_•}
    \quad\text{and}\quad
    \supp α^* D_{Y_2} ⊆ \supp D_{Y_1}.
  \]
  Then, the pull-back morphisms form a
  commutative diagram of sheaves on $Y_1$, as follows
  \[
    \begin{tikzcd}[column sep=2cm, baseline=(current bounding box.south east)]
      φ_1^* \: \Sym^{[•]}_{\cC} Ω^{[•]}_{(X,D_X,γ_1)} \ar[r, "d_{\cC} φ_1"] & \Sym^• Ω^•_{Y_1}(\log D_{Y_1}) \\
      φ_1^* β^{[*]} \: \Sym^{[•]}_{\cC} Ω^{[•]}_{(X,D_X,γ_2)} \ar[u, "φ_1^*(\text{Obs.~\ref{obs:4-14}})"] & α^* \Sym^• Ω^•_{Y_2}(\log D_{Y_2}) \ar[u, "dα"'] \\
      φ_1^* β^* \: \Sym^{[•]}_{\cC} Ω^{[•]}_{(X,D_X,γ_2)} \ar[u, "\text{natl.}"] \ar[r, equal] & α^* φ_2^* \: \Sym^{[•]}_{\cC} Ω^{[•]}_{(X,D_X,γ_2)} \ar[u, "α^*(d_{\cC} φ_2)"']
    \end{tikzcd}
    \eqno \qed
  \]
\end{fact}

\begin{fact}[Open immersions]
  In Setting~\ref{set:5-2}, assume that the morphism $φ : Y → \what{X}$ is an
  open immersion, so we may view $Y$ as an open subset of $\what{X}$.  The
  pull-back morphisms $d_{\cC} φ$ are then equal to the composition of the
  following sequence of sheaf morphisms,
  \begin{align*}
    \Sym^{[•]}_{\cC} Ω^{[•]}_{(X,D_X,γ)}\bigr|_Y & ↪ \Sym^{[•]} Ω^{[•]}_{\what{X}} \bigl(\log γ^*⌊D_X⌋ \bigr)\bigr|_Y && \text{Observation~\ref{obs:4-8}} \\
    & ↪ \Sym^{[•]} Ω^{[•]}_Y(\log D_Y) = \Sym^• Ω^•_Y(\log D_Y),
  \end{align*}
  where the last inclusion is induced by the assumption that $\supp γ^*⌊D_X⌋
  ⊆ \supp D_Y$.  \qed
\end{fact}

\begin{fact}[Standard operations]\label{fact:5-14}%
  The pull-back morphisms $d_{\cC} φ$ of Fact~\ref{fact:5-9} commute with
  (reflexive) wedge products, symmetric products and exterior derivatives.  \qed
\end{fact}

\subsubsection{Consequences of the universal properties}
\approvals{Erwan & yes \\ Stefan & yes}

We highlight a few cases that will later become relevant.  The following
propositions are direct consequences of the compatibility between $d_{\cC}$ and
the pull-back of Kähler tensors, as stated in Fact~\vref{fact:5-11}.

\begin{prop}[Smooth base spaces]\label{prop:5-15}%
  In Setting~\ref{set:5-2}, assume that the space $X$ is smooth and $D_X = 0$,
  so that $\Sym^{[•]}_{\cC} Ω^{[•]}_{(X,0,γ)} = γ^* \Sym^• Ω^•_X ⊆ \Sym^•
  Ω^•_{\what{X}}$ is a sheaf of Kähler tensors.  If $(Y,D_Y)$ is nc, then the
  pull-back morphisms $d_{\cC} φ$ equal the standard pull-back of Kähler
  tensors.  More precisely, there exist commutative diagrams as follows,
  \[
    \begin{tikzcd}[column sep=1cm, baseline=(current bounding box.south east)]
      φ^* \Sym^{[•]}_{\cC} Ω^{[•]}_{(X,0,γ)} \ar[rr, "d_{\cC} φ"] \ar[d, equal] && \Sym^• Ω^•_Y(\log D_Y) \ar[d, equal] \\
      φ^*γ^* \Sym^• Ω^•_X \ar[r, "d (γ◦ φ)"'] & \Sym^• Ω^•_Y \ar[r, hook] & \Sym^• Ω^•_Y(\log D_Y).
    \end{tikzcd}
    \eqno \qed
  \]
\end{prop}

\begin{prop}[Uniformizations]
  In Setting~\ref{set:5-2}, assume that the cover $γ$ is a uniformization, so
  that $\Sym^{[•]}_{\cC} Ω^{[•]}_{(X,0,γ)} = \Sym^• Ω^•_{\what{X}}(\log γ^*
  ⌊D⌋)$ is a sheaf of logarithmic Kähler tensors.  If $(Y,D_Y)$ is nc, then the
  pull-back morphisms $d_{\cC} φ$ equal the standard pull-back of Kähler
  differentials.  More precisely, there exist commutative diagrams as follows,
  \[
    \begin{tikzcd}[column sep=2cm, baseline=(current bounding box.south east)]
      φ^* \Sym^{[•]}_{\cC} Ω^{[•]}_{(X,0,γ)} \ar[r, "d_{\cC} φ"] \ar[d, equal] & \Sym^• Ω^•_Y(\log D_Y) \ar[d, equal] \\
      \Sym^• Ω^•_{\what{X}}(\log γ^* ⌊D⌋) \ar[r, "d φ"'] & \Sym^• Ω^•_Y(\log D_Y).
    \end{tikzcd}
    \eqno \qed
  \]
\end{prop}

\begin{prop}[Constant morphism]\label{prop:5-17}
  In Setting~\ref{set:5-2}, assume that $φ$ is constant.  Then, the pull-back
  morphisms $d_{\cC} φ$ are zero.
\end{prop}
\begin{proof}
  Write $\{x\} := \img φ$.  Applying Fact~\ref{fact:5-12} to the diagram
  \[
    \begin{tikzcd}[column sep=2cm]
      Y \ar[r, "φ"] \ar[d, "\const"'] & \what{X} \ar[r, "γ\text{, $q$-morphism}"] \ar[d, equal] & X \ar[d, equal] \\
      \{x\} \ar[r, hook] & \what{X} \ar[r, "γ"'] & X,
    \end{tikzcd}
  \]
  we obtain identifications $d_{\cC} φ = d \const = 0$.
\end{proof}

\subsubsection{Uniqueness}
\approvals{Erwan & yes \\ Stefan & yes}

We mentioned in Section~\ref{sec:5-2} that the pull-back morphisms $d_{\cC}$ are
uniquely determined by ``functoriality'' and ``compatibility with pull-back of
Kähler differentials''.  While true, this might require some explanation.  To
begin, observe that compatibility with the pull-back of Kähler differentials
alone does \emph{not} determine the pull-back morphism $d_{\cC} φ$ in all
possible settings --- with the notation of Setting~\ref{set:5-2}, think of a
case where the image of $φ$ is entirely contained in the singular locus of
$\what{X}$.

It is however true that compatibility with the pull-back of Kähler differentials
and functoriality \emph{together} determine the collection of pull-back morphisms,
\[
  \bigl(d_{\cC} φ \bigr)_{(φ \text{ from a sequence of morphisms as in Setting~\ref{set:5-2}})}.
\]
Precise statements and proofs are not hard to give, but will be lengthy and
painful to spell out.  Rather than going into too much detail here, we refer the
reader to \cite[Sect.~6.4]{MR3084424} and \cite[Sect.~14]{KS18} that discuss
completely analogous situations.

% !TEX root = orbiAlb1
%
% Do not edit the following line.  The text is automatically updated by
% subversion.
%
\svnid{$Id: 06-invariants.tex 911 2024-09-29 18:09:56Z kebekus $}
\selectlanguage{british}

\section{Invariants of \texorpdfstring{$\cC$}{ C}-pairs}
\label{sec:6}
\subversionInfo
\approvals{Erwan & yes \\ Stefan & yes}

Almost every invariant defined for compact Kähler manifolds (or logarithmic
Kähler pairs) has an analogue in the setting of $\cC$-pairs.  This section
introduces two invariants of particular importance: irregularities and
Kodaira-Iitaka dimensions for sheaves of tensors.

\subsection{Irregularities}
\approvals{Erwan & yes \\ Stefan & yes}
\label{sec:6-1}

For $\cC$-pairs, adapted differentials take the role that ordinary differentials
play for ordinary spaces.  Accordingly, there exists a meaningful notion of
``irregularity'' for $\cC$-pairs.  It goes without saying that the irregularity
is of fundamental importance when we discuss $\cC$-analogues of the Albanese in
the follow-up paper \cite{orbialb2}.

\begin{defn}[Irregularity, augmented irregularity]
  Let $(X, D)$ be a compact $\cC$-pair and let $γ : \what{X} \twoheadrightarrow
  X$ be any cover.  We refer to the number
  \[
    q(X,D,γ) := h⁰ \Bigl( \what{X},\, Ω^{[1]}_{(X,D,γ)} \Bigr)
  \]
  as the \emph{irregularity of $(X,D,γ)$}\index{irregularity of a $\cC$-pair and
  cover}.  The number
  \[
    q⁺(X, D) := \sup \bigl\{ q(X,D,γ) \,|\, γ \text{ a cover} \bigr\} ∈ ℕ ∪ \{∞\}
  \]
  is called \emph{augmented irregularity of the $\cC$-pair $(X,
  D)$}\index{augmented irregularity of a $\cC$-pair}.
\end{defn}

We do not fully understand how the irregularities $q(X,D,γ)$ depend on the
covering map $γ$.  Section~\ref{sec:15-2} gathers several open questions
there.  Before turning to a $\cC$-analogue of the Kodaira-Iitaka dimension in
the next subsection, we highlight a few cases where the irregularities can be
computed.

\begin{lem}[Projective manifolds of small codimension]\label{lem:6-2}%
  Let $X$ be a projective manifold.  If $X$ admits an embedding $X ⊆ ℙ^n$ with
  $n < 2·\dim X$, then $q⁺(X,0) = 0$.
\end{lem}
\begin{proof}
  If $q⁺(X,0) > 0$, then there exists a Galois cover $γ : \what{X}
  \twoheadrightarrow X$ with group $G$ and a non-trivial section $σ$ in
  $Ω^{[1]}_{(X,0,γ)} = γ^* Ω¹_X$.  The pluri-differential
  \[
    \prod_{g ∈ G} g^* σ ∈ H⁰ \bigl( \what{X},\, \Sym^{\#G} Ω^{[1]}_{(X,0,γ)}\bigr)
  \]
  is $G$-invariant, not trivial, and by Lemma~\ref{lem:4-20} gives a non-trivial
  section in $\Sym^{\#G} Ω¹_X$.  It has, however, been shown in \cite{MR1144433}
  that $h⁰\bigl( X,\, \Sym^N Ω¹_X \bigr) = 0$ for every number $N ∈ ℕ$.
\end{proof}

To construct a more interesting example, recall from \cite[Thm.~1.5]{MR3953506}
that every projective, klt variety $X$ with numerically trivial canonical class
admits a cover $\wtilde{X} \twoheadrightarrow X$, étale in codimension one, and
a decomposition
\[
  \wtilde{X} ≅ A ⨯ \prod_j Y_j ⨯ \prod_k Z_k
\]
where $A$ is an Abelian variety, where the $Y_j$ are (possibly singular)
Calabi-Yau varieties and the $Z_k$ are (possibly singular) irreducible
symplectic varieties.  We refer the reader to \cite[Sect.~1.4 and
Def.~1.3]{MR3988092} for a discussion and for the definition of ``singular
Calabi-Yau'' and ``singular irreducible symplectic''.

\begin{lem}
  Let $Y$ be singular Calabi-Yau or singular irreducible symplectic.  If $φ : X
  → Y$ is any birational morphism between normal, projective varieties, and if
  $D ∈ \Div(X)$ is a $φ$-exceptional divisor that makes $(X,D)$ a $\cC$-pair,
  then $q⁺(X,D) = 0$.
\end{lem}
\begin{proof}
  Assume that $q⁺(X,D) > 0$.  As in the proof of Lemma~\ref{lem:6-2}, we can
  then construct a pluri-differential form on $X_{\reg} ∖ \supp D$, hence a
  non-trivial section $σ ∈ H⁰\bigl( Y⁺,\, \Sym^• Ω¹_{Y⁺}\bigr)$, where
  \[
    Y⁺ := Y_{\reg} ∖ \text{indeterminacy locus of } φ^{-1}.
  \]
  Since $Y⁺ ⊆ Y$ is a big open subset, $σ$ induces a non-trivial reflexive
  pluri-differential $σ' ∈ H⁰\bigl( Y,\, \Sym^{[•]} Ω^{[1]}_Y\bigr)$.  However,
  it has been shown in \cite[Thm.~1.11]{MR3988092} that no such reflexive
  pluri-differential exist if $Y$ is singular Calabi-Yau or singular irreducible
  symplectic.
\end{proof}

We refer the reader to \cite{MR3127814} for more on the relation between
existence of pluri-differentials and the geometry of the underlying space.

\begin{example}[Unbounded Irregularities]
  Let $X$ be a compact Riemann surface of general type, and let $γ : \what{X} →
  X$ be any étale cover.  Then,
  \[
    q(X,0,γ)
    = h⁰ \bigl( \what{X},\, Ω^{[1]}_{(X,0,γ)} \bigr)
    = h⁰ \bigl( \what{X},\, γ^* Ω¹_X \bigr)
    = h⁰ \bigl( \what{X},\, Ω¹_{\what{X}} \bigr)
    = g(\what{X}).
  \]
  Given that étale covers of arbitrarily high degrees exist, we find that $q⁺(X,
  0) = ∞$.
\end{example}

\subsection{The $\cC$-Kodaira-Iitaka dimension}
\approvals{Erwan & yes \\ Stefan & yes}
\label{sec:6-2}

This section introduces the $\cC$-Kodaira-Iitaka dimension for rank-one sheaves
of adapted tensors.  We refer the reader to \cite[Sect.~4]{MR2860268} for a
related construction, for references, and proper attributions.

\begin{defn}[$\cC$-product sheaves]\label{def:6-5}%
  Let $(X,D)$ be a $\cC$-pair and let $γ : \what{X} → X$ be a $q$-morphism.
  Assume we are given numbers $n, d, p ∈ ℕ⁺$ and a coherent subsheaf of adapted
  reflexive tensors,
  \[
    ℱ ⊆ \Sym^{[n]}_{\cC} Ω^{[p]}_{(X, D, γ)}.
  \]
  Using the inclusion
  \[
    \Sym^{[d]} ℱ
    ⊆ \Sym^{[d]} \Sym^{[n]}_{\cC} Ω^{[p]}_{(X, D, γ)}
    \overset{\text{Obs.~\ref{obs:4-11}}}{⊆} \Sym^{[d·n]}_{\cC} Ω^{[p]}_{(X, D, γ)},
  \]
  we define the \emph{$d^{\text{th}}$ $\cC$-product sheaf of
  $ℱ$}\index{C-product sheaf@$\cC$-product sheaf} as
  \[
    \Sym^{[d]}_{\cC} ℱ := \text{saturation of } \Sym^{[d]} ℱ
    \text{ in } \Sym^{[d·n]}_{\cC} Ω^{[p]}_{(X, D, γ)}.
  \]
\end{defn}

\begin{rem}[Elementary properties]
  As the saturation of a coherent sheaf within a reflexive sheaf, the product
  sheaf $\Sym^{[d]}_{\cC} ℱ$ of Definition~\ref{def:6-5} is always reflexive.  If
  $ℱ$ has rank one, then $\Sym^{[d]}_{\cC} ℱ$ also has rank one.
\end{rem}

\begin{defn}[$\cC$-Kodaira-Iitaka dimension]%
  Let $(X,D)$ be a compact $\cC$-pair and let $γ : \what{X} \twoheadrightarrow
  X$ be a cover.  Given numbers $n, p ∈ ℕ⁺$ and a coherent, rank-one subsheaf of
  adapted reflexive tensors,
  \[
    ℱ ⊆ \Sym^{[n]}_{\cC} Ω^{[p]}_{(X, D, γ)},
  \]
  consider the set
  \[
    M := \left\{ m ∈ ℕ\, \left|\, h⁰\bigl( \what{X},\, \Sym^{[m]}_{\cC} ℱ \bigr) > 0 \right.  \right\}.
  \]
  If $M = ∅$, we say that the sheaf $ℱ$ has \emph{$\cC$-Kodaira-Iitaka dimension
  minus infinity} and write $κ_{\cC}(ℱ) = -∞$.  Otherwise, consider the natural
  meromorphic maps
  \[
    φ_m : \what{X} \dasharrow ℙ\left( H⁰\Bigl( \what{X},\, \Sym^{[m]}_{\cC} ℱ \Bigr)^\vee
    \right), \text{ for each } m ∈ M
  \]
  and define the \emph{$\cC$-Kodaira-Iitaka dimension}\index{C-Kodaira-Iitaka
  dimension@$\cC$-Kodaira-Iitaka dimension} as
  \[
    κ_{\cC}(ℱ) = \max_{m ∈ M} \left\{ \dim \overline{φ_m(\what{X})}\right\},
  \]
  where $\overline{φ_m(\what{X})}$ denotes the Zariski closure of $φ_m(\what{X})
  ⊆ ℙ^•$.
\end{defn}

\begin{rem}
  Recall from Reminder~\vref{remi:2-8} that the $φ_m$ are meromorphic indeed, so
  that $φ_m(X) ⊆ ℙ^•$ are constructible.  This implies that the $\max$ in the
  definition of $κ_{\cC}(ℱ)$ is a maximum and that $κ_{\cC}(ℱ) ≤ \dim X$.
\end{rem}

\begin{warning}
  Unlike the standard Kodaira-Iitaka dimension, the $\cC$-Kodaira-Iitaka
  dimension is defined only for subsheaves of adapted reflexive differentials.
  Its value is generally not an invariant of the sheaf alone, and will often
  depend on the embedding into $\Sym^{[•]}_{\cC} Ω^{[•]}_{(X, D, γ)}$.
\end{warning}

\subsection{Bogomolov-Sommese vanishing and special pairs}
\approvals{Erwan & yes \\ Stefan & yes}

Campana has observed in \cite[Sect.~3.5]{MR2831280} that the classic vanishing
theorem of Bogomolov and Sommese, \cite[Cor.~6.9]{EV92} carries over to
$\cC$-pairs with simple normal crossing boundary.  Using extension theorems for
differential forms on log canonical spaces, Patrick Graf generalized Campana's
observation substantially in his thesis \cite{GrafThesis}.  While Graf works
with projective varieties, his arguments carry over to the setting of compact
Kähler spaces\footnote{Graf's thesis relies on the paper \cite{GKKP11}, which
provides the relevant extension theorems for differential forms in the algebraic
setting.  The newer paper \cite{KS18} establishes analogous results in the
analytic setting.}.

\begin{thm}[\protect{Bogomolov-Sommese vanishing on $X$, \cite[Thm.~1.2]{Gra13}}]\label{thm:6-10}%
  \index{Bogomolov-Sommese vanishing!on a $\cC$-pair}Let $(X,D)$ be a log
  canonical $\cC$-pair where $X$ is compact Kähler.  If $ℱ ⊆ Ω^{[p]}_{(X, D,
  \Id_X)}$ is coherent of rank one, then $κ_{\cC}(ℱ) ≤ p$.  \qed
\end{thm}

We say that a pair is special if the inequality in the Bogomolov-Sommese
vanishing theorem is strict.

\begin{defn}[Bogomolov sheaf on $X$]\label{def:6-11}%
  Let $(X,D)$ be a log canonical $\cC$-pair where $X$ is compact Kähler.  A
  \emph{Bogomolov sheaf on $X$}\index{Bogomolov sheaf} is a coherent sheaf $ℱ ⊆
  Ω^{[p]}_{(X, D, \Id_X)}$ of rank one such that if $κ_{\cC}(ℱ) = p$.
\end{defn}

\begin{defn}[Special $\cC$-pair]\label{def:6-12}%
  Let $(X,D)$ be a log canonical $\cC$-pair where $X$ is compact Kähler.  The
  pair $(X,D)$ is called \emph{special}\index{special $\cC$-pair} when there are
  no Bogomolov sheaves.
\end{defn}

\begin{warning}
  We remark that Definition~\ref{def:6-11} differs from Campana's.  In
  \cite[Déf.~5.17]{MR2831280}, Campana defines ``specialness'' in terms of
  meromorphic fibrations $X \dasharrow Y$ onto orbifolds of general type.  For
  pairs with snc boundary, he shows in \cite[Cor.~3.13]{MR2831280} that the two
  definitions agree.
\end{warning}

\begin{rem}
  The importance of special pairs comes from the existence of the \emph{core
  map} constructed by Campana \cite[Théo.~10.1]{MR2831280} for smooth
  $\cC$-pairs: Given a smooth $\cC$-pair $(X,D)$ where $X$ is compact Kähler,
  there exists a unique fibration $c_{(X,D)}: (X,D) \dashrightarrow C(X,D)$ with
  special generic fibres and orbifold base of general type.  Campana's
  construction therefore splits any smooth $\cC$-pair into two antithetic parts,
  of ``special'' and ``general'' type.
\end{rem}
 
Under assumptions that are substantially stronger than those of
Theorem~\ref{thm:6-10}, an analogue of the Bogomolov-Sommese vanishing theorem
will hold for sheaves of adapted reflexive tensors on arbitrary covers of $X$.
We include a full statement for future reference and refer the reader to
Sections~\ref{sec:15-2} and \ref{sec:15-4} for questions concerning potential
generalizations.

\begin{prop}[Bogomolov-Sommese vanishing on covers of $X$]\label{prop:6-15}%
  \index{Bogomolov-Sommese vanishing!on a cover of a $\cC$-pair}Let $(X,D)$ be a
  locally uniformizable $\cC$-pair where $X$ is compact Kähler.  Let $γ :
  \what{X} \twoheadrightarrow X$ be a cover.  If $ℱ ⊆ Ω^{[p]}_{(X, D, γ)}$ is
  coherent of rank one, then $κ_{\cC}(ℱ) ≤ p$.
\end{prop}
\begin{proof}
  For the reader's convenience, we subdivide the proof into relatively
  independent steps.

  \subsubsection*{Step 1: Setup}
  
  Let $π : Y \twoheadrightarrow \what{X}$ be a strong log resolution of the pair
  $(\what{X}, γ^* ⌊D⌋)$ and consider the reduced divisor
  \[
    D_Y := \Bigl(π^*γ^* ⌊D⌋\Bigr)_{\red} ∈ \Div(Y).
  \]
  The pair $(Y, D_Y)$ is then snc, and Fact~\ref{fact:5-9} provides us with
  pull-back maps
  \[
    d_{\cC} π : π^* \Sym^{[n]}_{\cC} Ω^{[p]}_{(X,D,γ)} → \Sym^n Ω^p_Y (\log D_Y).
  \]
  We consider the saturated images of the $\cC$-product sheaves
  $\Sym^{[n]}_{\cC} ℱ$ and write
  \[
    ℱ^n_Y := \text{saturation of } (d_{\cC} π) \Bigl(π^* \Sym^{[n]}_{\cC} ℱ\Bigr) \text{ in } \Sym^n Ω^p_Y (\log D_Y).
  \]
  There are two things that we can say immediately.
  \begin{enumerate}
    \item The sheaves $ℱ^n_Y$ are reflexive of rank one.  Since $Y$ is smooth,
    this implies that $ℱ^n_Y$ are invertible, \cite[Lem.~1.1.5]{OSS}.
  
    \item\label{il:6-15-2} The compatibility of pull-back and reflexive
    symmetric products asserts that
    \[
      \bigl(ℱ¹_Y\bigr)^{⊗ n} = \Sym^n ℱ¹_Y \overset{\text{Fact~\ref{fact:5-14}}}{⊆} ℱ^n_Y, \quad \text{for every } n ∈ ℕ⁺.
    \]
  \end{enumerate}

  \subsubsection*{Step 2: Relation between $ℱ¹_Y$ and $ℱ^n_Y$}

  Following ideas of Patrick Graf \cite{GrafThesis,Gra13} and simplifying some
  of his arguments, we will show in this step that the inclusions in
  Item~\ref{il:6-15-2} are in fact equalities,
  \begin{equation}\label{eq:6-15-3}
    \bigl(ℱ¹_Y\bigr)^{⊗ n} = ℱ^n_Y, \quad \text{for every } n ∈ ℕ⁺.
  \end{equation}
  Since the sheaves on both sides of \eqref{eq:6-15-3} are locally free, it
  suffices to show equality on a suitable big open subset of $Y$.  To this end,
  recall from the construction of the saturation that $ℱ¹_Y$ appears on the left
  of an exact sequence of coherent sheaves on $Y$,
  \[
    0 → ℱ¹_Y → Ω^p_Y (\log D_Y) → 𝒬 → 0,
  \]
  where $𝒬$ is torsion free, and hence locally free over a suitable big, open
  subset $Y° ⊆ Y$, see \cite[Cor.~on p.~75]{OSS}.  Since short exact sequences
  of locally free sheaves are locally split, the sheaf $ℱ¹_Y|_{Y°}$ is locally a
  direct summand of $Ω^p_Y (\log D_Y)|_{Y°}$.  But then $\bigl(ℱ¹_Y\bigr)^{⊗
  n}|_{Y°}$ is locally a direct summand of $\Sym^n Ω^p_Y (\log D_Y)|_{Y°}$.  The
  subsheaf $\bigl(ℱ¹_Y\bigr)^{⊗ n}|_{Y°} ⊆ \Sym^n Ω^p_Y (\log D_Y)|_{Y°}$ is
  therefore saturated, hence equal to $ℱ^n_Y|_{Y°}$.  Equality~\eqref{eq:6-15-3}
  thus follows.

  \subsubsection*{Step 3: End of proof}

  By construction, sections of the $\cC$-product sheaves $\Sym^{[n]}_{\cC} ℱ$
  induce section of $ℱ^n_Y$,
  \[
    H⁰\Bigl( \what{X},\, \Sym^{[m]}_{\cC} ℱ \Bigr)
    ↪ H⁰\Bigl( Y,\, ℱ^n_Y \Bigr)
    \overset{\text{\eqref{eq:6-15-3}}}{=} H⁰\Bigl( Y,\, (ℱ¹_Y)^{⊗ n} \Bigr).
  \]
  If $h⁰\bigl( \what{X},\, \Sym^{[m]}_{\cC} ℱ \bigr) > 0$, then the associated
  meromorphic mappings are related,
  \[
    \begin{tikzcd}[row sep=1cm]
      Y \ar[r, dashed, "ψ_m"] \ar[d, two heads, "π"'] & ℙ\left( H⁰\Bigl( X,\, ℱ^n_Y \Bigr)^\vee \right) \ar[d, dashed, two heads, "\text{projection}"] \\
      \what{X} \ar[r, dashed, "φ_m"'] & ℙ\left( H⁰\Bigl( \what{X},\, \Sym^{[m]}_{\cC} ℱ \Bigr)^\vee \right),
    \end{tikzcd}
  \]
  so that $\dim \overline{φ_m(\what{X})} ≤ \dim \overline{ψ_m(Y)}$.  In summary,
  we find that $κ_{𝒞}(ℱ) ≤ κ(ℱ¹_Y)$.  The classic Bogomolov-Sommese vanishing
  theorem, \cite[Cor.~6.9]{EV92}, asserts that $κ(ℱ¹_Y) ≤ p$.
\end{proof}

\subsection{Conjectures on the geometry of special pairs}
\approvals{Erwan & yes \\ Stefan & yes}

The class of special $\cC$-pairs is supposed to generalize rational or elliptic
curves, which suggests the following conjectures made by Campana
\cite[Conj.~13.10, 13.15, 13.23]{MR2831280}.

\begin{conjecture}
  Let $(X,D)$ be a smooth $\cC$-pair where $X$ is compact Kähler.
  \begin{enumerate}
  \item If $⌊D⌋=0$ and $(X,D)$ is special, then the orbifold fundamental group
  $π_1(X,D)$ is almost Abelian.

  \item The orbifold Kobayashi pseudo-distance $d_{(X,D)}$ vanishes identically
  if and only if $(X,D)$ is special.

  \item If $(X,D)$ is projective defined over a number field $k$, then there
  exists a finite extension $k' ⊃ k$ such that rational points $(X,D)(k')$ are
  dense if and only if $(X,D)$ is special.
  \end{enumerate}
\end{conjecture}
 
Campana has shown that a special compact Kähler manifold $X$ has a surjective
Albanese map \cite[Prop.~5.3]{Cam04} which implies in particular that its
classical augmented irregularity (computed with étale covers) satisfies
$\wtilde{q}(X) ≤ \dim X$.  We formulate a conjecture generalizing this property
to special $\cC$-pairs.  It will be discussed in a sequel \cite{orbialb2} of
this article.

\begin{conjecture}
  Let $(X,D)$ be a log canonical $\cC$-pair where $X$ is compact Kähler.  If
  $(X,D)$ is special, then the augmented irregularity satisfies $q⁺(X, D)≤ \dim
  X$.
\end{conjecture}

% !TEX root = orbiAlb1

\phantomsection\addcontentsline{toc}{part}{Morphisms of \texorpdfstring{$\cC$}{C}-pairs}

%
% Do not edit the following line.  The text is automatically updated by
% subversion.
%
\svnid{$Id: 07-diagrams.tex 848 2024-07-15 06:02:59Z kebekus $}
\selectlanguage{british}

\section{Diagrams admitting pull-back}
\subversionInfo
\label{sec:7}
\approvals{Erwan & yes \\ Stefan & yes}

We feel that the sheaves of adapted reflexive differentials, and in particular
the $\cC$-cotangent sheaf, are the key objects that make Campana's theory
useful.  Accordingly, we define a morphism of $\cC$-pairs as a morphism of
varieties that allows pull-back of adapted reflexive differentials, in a manner
that is compatible with the standard pull-back of Kähler differentials, and
hence with the pull-back maps introduced in Section~\ref{sec:5-4} above.

\subsection{Diagrams admitting pull-back}
\approvals{Erwan & yes \\ Stefan & yes}

The requirement that pull-back be ``compatible with the standard pull-back of
Kähler differentials and with the pull-back maps introduced in
Section~\ref{sec:5-4}'' is conceptually straightforward, but gets somewhat
technical to write down correctly.  To avoid any potential of confusion, we
clarify assumptions and notation explicitly, remind of the relevant facts, and
formulate the main definition in great detail.

\begin{setting}[Commutative diagram of $q$-morphisms]\label{set:7-1}%
  Assume we are given $\cC$-pairs $(X, D_X)$ and $(Y, D_Y)$ and a commutative
  diagram of the following form,
  \begin{equation}\label{eq:7-1-1}
    \begin{tikzcd}[column sep=large]
      \what{X} \ar[r, "\what{φ}"] \ar[d, "a\text{, $q$-morphism}"'] & \what{Y} \ar[d, "b\text{, $q$-morphism}"] \\
      X \ar[r, "φ"'] & Y.
    \end{tikzcd}
  \end{equation}
  Following the conventions of Sections~\ref{sec:2-3} and \ref{sec:2-5} write
  \begin{itemize}
    \item $X° := X ∖ \supp ⌊D_X⌋$ for the open part of $(X,D_X)$,

    \item $Y° := Y ∖ \supp ⌊D_Y⌋$ for the open part of $(Y,D_Y)$, and

    \item $Y^{\lu} ⊆ Y$ for the maximal open subset over which $(Y,D_Y)$ is
    locally uniformizable.
  \end{itemize}
  We assume that
  \begin{equation}\label{eq:7-1-2}
    φ(X°) ⊆ Y° \quad\text{and}\quad \img φ ∩ Y^{\lu} ≠ ∅.
  \end{equation}
\end{setting}

\begin{rem}[$q$-morphisms and covers]
  We stress that the morphisms $a$ and $b$ in \eqref{eq:7-1-1} need not be
  surjective.  In other words, we ask that $a$ and $b$ are \emph{$q$-morphisms}
  (and hence open, not necessarily surjective), but not necessarily adapted
  \emph{covers} (which would imply surjective).
\end{rem}

The discussion of Setting~\ref{set:7-1} involves a number of additional and
auxiliary objects, such as pull-back divisors and preimage sets.  For the
reader's convenience, we use the following notation consistently throughout the
present section.

\begin{notation}[Divisors, open sets and restrictions in Setting~\ref{set:7-1}]\label{not:7-3}%
  In Setting~\ref{set:7-1}, consider the reduced divisor
  \[
    \what{D}_X := \bigl(a^* ⌊D_X⌋\bigr)_{\red} ∈ \Div(\what{X}).
  \]
  We consider the preimage set
  \[
    \what{Y}⁺ := b^{-1} \bigl(Y^{\lu}\bigr) ⊆ \what{Y}
    \quad\text{and let}\quad
    \what{X}⁺ ⊆ \what{φ}^{-1}\bigl(\what{Y}⁺\bigr) ⊆ \what{X}
  \]
  be the maximal open subset where $(\what{X}, \what{D}_X)$ is nc.  For brevity,
  denote the restrictions by
  \[
    \bigl(\what{X}⁺, \what{D}_X⁺\bigr) := \bigl(\what{X}, \what{D}_X \bigr) \bigr|_{\what{X}⁺}
    \quad\text{and}\quad
    \what{φ}⁺ := \what{φ}|_{\what{X}⁺} : \what{X}⁺ → \what{Y}⁺.
  \]
  The following diagram summarizes the situation,
  \[
    \begin{tikzcd}
      \what{X}⁺ \ar[d, "a|_{\what{X}⁺}"'] \ar[rrr, bend left=20, "\what{φ}⁺"] \ar[r, phantom, pos=0.65, "⊆"] & \what{X} \ar[r, "\what{φ}"'] \ar[d, "a"'] & \what{Y} \ar[d, "b"] \ar[r, phantom, pos=0.45, "⊇"] & \what{Y}⁺ \ar[d, "b|_{\what{Y}⁺}"] \\
      φ^{-1}\bigl(Y^{\lu}\bigr) \ar[r, phantom, "⊆"] & X \ar[r, "φ"'] & Y \ar[r, phantom, "⊇"] & Y^{\lu}.
    \end{tikzcd}
  \]
\end{notation}

\begin{obs}[Pull-back morphisms in Setting~\ref{set:7-1}]\label{obs:7-4}%
  Setting~\ref{set:7-1} reproduces the setup discussed in Section~\ref{sec:5},
  where we introduced the pull-back map for adapted reflexive differentials.
  Using Notation~\ref{not:7-3}, observe that conditions~\eqref{eq:7-1-2} of
  Setting~\ref{set:7-1} guarantee that
  \[
    \supp \bigl(\what{φ}⁺\bigr)^* b^* ⌊D_Y⌋ ⊆ \supp \what{D}_X⁺
    \quad\text{and}\quad
    \what{X}⁺ ≠ ∅.
  \]
  Modulo some differences in the notation, we are therefore in the setting
  spelled out in \vref{set:5-2}.  Fact~\ref{fact:5-9} will therefore apply to
  give canonical pull-back morphisms
  \begin{equation}\label{eq:7-4-1}
    d_{\cC} \what{φ}⁺ : \bigl(\what{φ}⁺\bigr)^* \Sym^{[n]}_{\cC} Ω^{[•]}_{(Y,D_Y,b)} → \Sym^n Ω^•_{\what{X}⁺}(\log \what{D}_X⁺)
  \end{equation}
  that enjoy all properties spelled out in Fact~\ref{fact:5-11}--\ref{fact:5-14}
  above.
\end{obs}

\begin{rem}[Pull-back and adapted reflexive differentials]
  For the upcoming discussion, recall from Observation~\ref{obs:4-8} that the
  targets of the pull-back morphisms \eqref{eq:7-4-1} contain the sheaves of
  adapted reflexive tensors,
  \[
    \Sym^{[•]}_{\cC} Ω^{[•]}_{(X,D_X,a)} \bigr|_{\what{X}⁺} ⊆ \Sym^• Ω^•_{\what{X}⁺}(\log \what{D}_X⁺).
  \]
\end{rem}

\begin{defn}[{Diagrams admitting pull-back of adapted tensors}]\label{def:7-6}%
  Assume Setting~\ref{set:7-1} and use Notation~\ref{not:7-3}.  Given numbers
  $n, p ∈ ℕ⁺$, we say that \emph{$\what{φ}$ admits pull-back of adapted
  reflexive $(n,p)$-tensors}\index{admitting pull-back!of adapted reflexive
  tensors} or \emph{Diagram~\eqref{eq:7-1-1} admits pull-back of adapted
  reflexive $(n,p)$-tensors} if there exists a sheaf morphism
  \[
    η : \what{φ}^* \Bigl( \Sym^{[n]}_{\cC} Ω^{[p]}_{(Y, D_Y, b)} \Bigr) → \Sym^{[n]}_{\cC} Ω^{[p]}_{(X, D_X, a)}
  \]
  whose restriction to $\what{X}⁺$ agrees with the pull-back morphism $d_{\cC}
  \what{φ}⁺$.  In other words, there exists a factorization of $d_{\cC}
  \what{φ}⁺$ as follows,
  \[
    \begin{tikzcd}[column sep=1.2cm]
      \what{φ}^* \bigl( \Sym^{[n]}_{\cC} Ω^{[p]}_{(Y, D_Y, b)} \bigr)\bigr|_{\what{X}⁺} \ar[r, "η|_{\what{X}⁺}"'] \ar[rr, bend left=10, "d_{\cC} \what{φ}⁺"] & \Sym^{[n]}_{\cC} Ω^{[p]}_{(X, D_X, a)}\bigr|_{\what{X}⁺} \ar[r, hook, "\text{Obs.~\ref{obs:4-8}}"'] & \Sym^n Ω^p_{\what{X}⁺}(\log \what{D}⁺_X),
    \end{tikzcd}
  \]
\end{defn}

\begin{notation}[Diagrams admitting pull-back of adapted differentials]
  Assume the setting of Definition~\ref{def:7-6}.  We say that
  \emph{Diagram~\eqref{eq:7-1-1} admits pull-back of adapted reflexive
  differentials}\index{admitting pull-back!of adapted reflexive differentials}
  if it admits pull-back of adapted reflexive $(1,p)$-tensors, for every $p ∈
  ℕ⁺$.
\end{notation}

\begin{warning}[Not enough to consider $p = 1$]%
  Assume the setup of Definition~\ref{def:7-6}.  When proving that $\what{φ}$
  admits pull-back of adapted reflexive differentials, one might be tempted to
  hope that it suffices to check that $\what{φ}$ admits pull-back of adapted
  reflexive differentials only for $p=1$.  This is not the case in general.  The
  sheaf $Ω^{[1]}_{(Y, D_Y, b)}$ is typically \emph{not} locally free, the
  natural morphisms
  \[
    \bigwedge^p Ω^{[1]}_{(Y, D_Y, b)} → Ω^{[p]}_{(Y, D_Y, b)}
  \]
  are typically \emph{not} surjective, and sections in $Ω^{[p]}_{(Y, D_Y, b)}$
  can typically \emph{not even locally} be written as products of sections in
  $Ω^{[1]}_{(Y, D_Y, b)}$.  We refer the reader to Example~\vref{ex:8-7} for a
  concrete example.
\end{warning}

\subsection{Elementary properties}
\approvals{Erwan & yes \\ Stefan & yes}

The following observations are elementary but useful.  We include them for later
reference.

\begin{obs}[Uniqueness]
  If it exists at all, then the morphism $η$ of Definition~\ref{def:7-6} is
  unique.  \qed
\end{obs}

\begin{obs}[Compatibility with pull-back of Kähler tensors]\label{obs:7-10}%
  In the setting of Definition~\ref{def:7-6}, assume that we are given numbers
  $n, p ∈ ℕ⁺$ for which the pull-back morphism $η$ exists.  Let $σ ∈ H⁰\bigl(
  \what{Y},\, \Sym^n Ω^p_{\what{Y}}\bigr)$ be a Kähler tensor, with pull-back $τ
  ∈ H⁰\bigl( \what{X},\, \Sym^n Ω^p_{\what{X}}\bigr)$ and denote the associated
  reflexive tensors by
  \[
    σ_r ∈ H⁰\bigl(\what{Y},\, \Sym^{[n]} Ω^{[p]}_{\what{Y}}\bigr)
    \quad\text{and}\quad
    τ_r ∈ H⁰\bigl(\what{X},\, \Sym^{[n]} Ω^{[p]}_{\what{X}}\bigr).
  \]
  If $σ_r$ is adapted, then Fact~\ref{fact:5-11} implies that $τ_r$ is adapted,
  and that the composed morphism
  \begin{align*}
    H⁰\bigl( \what{Y},\, \Sym^{[n]}_{\cC} Ω^{[p]}_{(Y,D_Y,b)}\bigr) & \xrightarrow{\what{φ}^*} H⁰\bigl( \what{X},\, φ^* \Sym^{[n]}_{\cC} Ω^{[p]}_{(Y,D_Y,b)}\bigr) \\
    & \xrightarrow{H⁰(η)} H⁰\bigl( \what{X},\, \Sym^{[n]}_{\cC} Ω^{[p]}_{(X,D_X,a)}\bigr)
  \end{align*}
  maps $σ_r$ to $τ_r$.  \qed
\end{obs}

\begin{obs}[Local nature]\label{obs:7-11}%
  In the setting of Definition~\ref{def:7-6}, assume that we are given numbers
  $n, p ∈ ℕ⁺$.  Let $(\what{U}_i)_{i ∈ I}$ and $(\what{V}_j)_{j ∈ J}$ be open
  coverings of $\what{X}$ and $\what{Y}$, respectively.  Then, the morphism
  $\what{φ}$ admits pull-back of adapted reflexive $(n,p)$-tensors if and only
  if every restricted morphism
  \[
    \what{φ}|_{\what{U}_i ∩ \what{φ}^{-1}(\what{V}_i)} : \what{U}_i ∩ \what{φ}^{\,-1}\bigl(\what{V}_i\bigr) → \what{V}_i
  \]
  admits pull-back of adapted reflexive $(n,p)$-tensors.  \qed
\end{obs}

If the pull-back map $η$ exists for given numbers $n, p ∈ ℕ⁺$, it can be seen as
a section of the subsheaf
\[
  \sHom_{\what{X}} \Bigl( \what{φ}^* \Sym^{[n]}_{\cC} Ω^{[p]}_{(Y, D_Y, b)},\, \Sym^{[n]}_{\cC} Ω^{[p]}_{(X, D_X, a)} \Bigr),
\]
which is reflexive since $\Sym^{[n]}_{\cC} Ω^{[p]}_{(X, D_X, a)}$ is.  To give a
section of this sheaf over $\what{X}$, it is thus equivalent to give a section
over any big open subset.

\begin{obs}[Removing small subsets]\label{obs:7-12}%
  In the setting of Definition~\ref{def:7-6}, assume that we are
  given numbers $n, p ∈ ℕ⁺$.  Let $\what{U} ⊆ \what{X}$ be a big open subset.
  Then, the morphism $\what{φ}$ admits pull-back of adapted reflexive
  $(n,p)$-tensors if and only if the morphism $\what{φ}|_{\what{U}}$ admits
  pull-back of adapted reflexive $(n,p)$-tensors.  \qed
\end{obs}

% !TEX root = orbiAlb
%
% Do not edit the following line.  The text is automatically updated by
% subversion.
%
\svnid{$Id: 08-morphisms.tex 911 2024-09-29 18:09:56Z kebekus $}
\selectlanguage{british}

\section{Morphisms of \texorpdfstring{$\cC$}{ C}-pairs}
\subversionInfo
\label{sec:8}
\approvals{Erwan & yes \\ Stefan & yes}

As motivated in the introduction, we define a $\cC$-morphism as a morphism where
every diagram admits pull-back of adapted reflexive differentials.  We impose
Condition~\eqref{eq:7-1-2} to ensure that the word ``pull-back of adapted
reflexive differentials'' carries meaning.

\begin{defn}[Morphisms of $\cC$-pairs]\label{def:8-1}%
  Given $\cC$-pairs $(X, D_X)$ and $(Y, D_Y)$ and a morphism $φ : X → Y$ with
  \[
    φ(X°) ⊆ Y° \quad\text{and}\quad \img φ ∩ Y^{\lu} ≠ ∅,
  \]
  call $φ$ a \emph{morphism between $\cC$-pairs $(X, D_X)$ and $(Y,
  D_Y)$}\index{morphism!of $\cC$-pairs} if every commutative diagram of form
  \eqref{eq:7-1-1} admits pull-back of adapted reflexive differentials.
\end{defn}

\begin{rem}[Notation used in Definition~\ref{def:8-1}]
  Definition~\ref{def:8-1} uses the standard notation where $X° ⊆ X$ and $Y° ⊆
  Y$ denote the open parts of $(X,D_X)$ and $(Y,D_Y)$, and $Y^{\lu} ⊆ Y$ is the
  maximal open subset over which $(Y,D_Y)$ is locally uniformizable.
\end{rem}

\begin{notation}[$\cC$-morphisms]
  In the setting of Definition~\ref{def:8-1}, we will often write $φ : (X, D_X)
  → (Y, D_Y)$ to indicate that a given morphism $φ : X → Y$ is a morphism
  between the $\cC$-pairs $(X, D_X)$ and $(Y, D_Y)$.  We use the word
  \emph{$\cC$-morphism}\index{C-morphism@$\cC$-morphism} for brevity.
\end{notation}

\begin{warning}[Pull-back differentials vs.~pull-back of tensors]
  Note that Definition~\ref{def:8-1} asks for pull-back of adapted reflexive
  differentials and \emph{not} for the more general pull-back of adapted
  reflexive \emph{tensors}.  This is a deliberate design decision, reflecting
  the fact that differentials, rather than tensors, are the objects that carry
  geometric meaning.  Section~\ref{sec:13} discusses criteria to guarantee that
  some $\cC$-morphisms do indeed induce pull-back of adapted reflexive
  \emph{tensors}.
\end{warning}

The following criterion is a direct consequence of Observations~\ref{obs:7-11}
and \ref{obs:7-12}.

\begin{obs}[Local nature and removing small subsets]\label{obs:8-5}%
  Given $\cC$-pairs $(X, D_X)$ and $(Y, D_Y)$ and a morphism $φ : X → Y$ such
  that $φ(X°) ⊆ Y°$ and $\img φ ∩ Y^{\lu} ≠ ∅$, the following conditions are
  equivalent.
  \begin{enumerate}
    \item The morphism $φ$ is a $\cC$-morphism.

    \item There exist open coverings $(U_i)_{i ∈ I}$ and $(V_j)_{j ∈ J}$ of $X$
    and $Y$, respectively, such that every restricted morphism
    \[
      φ|_{U_i ∩ φ^{-1}(V_j)} : U_i ∩ φ^{-1}(V_j) → V_j
    \]
    is a $\cC$-morphism between the restricted $\cC$-pairs
    \[
      \bigl(U_i ∩ φ^{-1}(V_j),\, D_X ∩ U_i ∩ φ^{-1}(V_j) \bigr)
      \quad\text{and}\quad
      (V_j, D_Y ∩ V_j).
    \]

    \item There exists a big open subset $U ⊂ X$ such that the restriction
    $φ|_U$ is a $\cC$-morphism $(U, D_X ∩ U) → (Y, D_Y)$.  \qed
  \end{enumerate}
\end{obs}

\subsection{First example}
\approvals{Erwan & yes \\ Stefan & yes}

To check if a given morphism is a $\cC$-morphism, Definition~\ref{def:8-1}
requires us in principle to consider all diagrams of form \eqref{eq:7-1-1} and
to check if they admit pull-back of adapted reflexive differentials, for all
numbers $p$.  This can be cumbersome.  We will therefore postpone the discussion
of interesting examples until useful criteria are established in the subsequent
Section~\ref{sec:9}.  For now, we only mention one example, which we work out in
detail.

\begin{example}[Morphisms to a manifold without boundary]\label{ex:8-6}
  If $(X,D_X)$ is any $\cC$-pair and $φ : X → Y$ is any morphism to a manifold
  $Y$, then $φ$ is a $\cC$-morphism between the pairs $(X,D_X)$ and $(Y,0)$.
  For a proof, assume that a diagram of form \eqref{eq:7-1-1} is given,
  \[
    \begin{tikzcd}[column sep=large]
      \what{X} \ar[r, "\what{φ}"] \ar[d, "a\text{, $q$-morphism}"'] & \what{Y} \ar[d, "b\text{, $q$-morphism}"] \\
      X \ar[r, "φ"'] & Y.
    \end{tikzcd}
  \]
  We need to show that $\what{φ}$ admits pull-back of adapted reflexive
  differentials.  To begin, observe that
  \[
    a^{[*]} Ω^{•}_X = Ω^{[•]}_{(X,0,a)} ⊆ Ω^{[•]}_{(X,D_X,a)}
    \quad\text{while}\quad
    Ω^{[•]}_{(Y,0,b)} = b^* Ω^{•}_Y.
  \]
  Pull-back of adapted reflexive differentials is therefore a matter of
  pulling-back Kähler differentials.  To make this precise, follow
  Notation~\ref{not:7-3} so that $\what{X}⁺$ is the maximal open set of
  $\what{X}$ where $(\what{X}, \what{D}_X)$ is nc.  Proposition~\vref{prop:5-15}
  will then identify the pull-back morphisms for adapted reflexive
  differentials,
  \[
    d_{\cC} \what{φ}^{\:+} : \underbrace{(\what{φ}^{\:+})^* Ω^{[•]}_{(Y,0,b)}}_{= \what{φ}^*b^* Ω^•_Y|_{\what{X}⁺} = a^*φ^* Ω^•_Y|_{\what{X}⁺}} → Ω^•_{\what{X}⁺}(\log \what{D}_X⁺),
  \]
  with the pull-back map of Kähler
  differentials,
  \[
    \begin{tikzcd}[column sep=large]
      a^*φ^* Ω^•_Y|_{\what{X}⁺} \ar[r, "a^*(dφ)"'] \ar[rrr, bend left=10, "d_{\cC} \what{φ}⁺"] & a^* Ω^•_X|_{\what{X}⁺} \ar[r, hook, "da"'] & Ω^•_{\what{X}⁺} \ar[r, hook] & Ω^•_{\what{X}⁺}(\log \what{D}_X⁺).
    \end{tikzcd}
  \]
  It follows that
  \[
    \img d_{\cC} \what{φ}^{\:+} ⊆ \img da|_{\what{X}⁺} = Ω^{[•]}_{(X,0,a)}|_{\what{X}⁺} ⊆ Ω^{[•]}_{(X,D_X,a)}|_{\what{X}⁺}.
  \]
  It remains to verify that the morphisms $d_{\cC} \what{φ}^{\:+}$ extend from
  $\what{X}⁺$ to morphisms
  \[
    η : \what{φ}^{\:*} Ω^{[•]}_{(Y, 0, b)} → Ω^{[•]}_{(X, D_X, a)}
  \]
  that are defined on all of $\what{X}$.  Extension is however clear, given that
  $Ω^{[•]}_{(Y, 0, b)}$ is locally free and $\what{X}⁺ ⊂ \what{X}$ is a big
  subset.
\end{example}

\subsection{First Counterexamples}
\approvals{Erwan & yes \\ Stefan & yes}

For singular varieties, Definition~\ref{def:8-1} is quite restrictive.  The
reader might find it surprising that a morphism between varieties does not
always give a morphism of $\cC$-pairs, even if target and domain are equipped
with the empty boundary.

\begin{example}[Pull-back for 1-differentials, not for 2-differentials]\label{ex:8-7}%
  Consider a construction described in \cite[Appendix~B]{KS18}: Let $E$ be an
  elliptic curve and let $L ∈ \Pic(E)$ be ample.  Let $Y$ be the affine cone
  over $E$ with conormal bundle $L$ and let $φ : X → Y$ be the minimal
  resolution, obtained by blowing up the vertex.  Considering the trivial pairs
  $(X,0)$ and $(Y,0)$, a diagram of the form \eqref{eq:7-1-1} is then given as
  \[
    \begin{tikzcd}[column sep=2cm, row sep=1cm]
      X \ar[r, two heads, "\what{φ} = φ"] \ar[d, hook, two heads, "a\text{, identity}"'] & Y \ar[d, hook, two heads, "b\text{, identity}"] \\
      X \ar[r, two heads, "φ\text{, resolution}"'] & Y.
    \end{tikzcd}
  \]
  It follows from \cite[Prop.~B.2]{KS18} that $\what{φ}$ admits pull-back of
  reflexive 1-differentials, but not of reflexive 2-differentials.  In
  particular, $φ$ does \emph{not} give a morphism between the $\cC$-pairs
  $(X,0)$ and $(Y,0)$.
\end{example}

\begin{example}[Resolution of the $A_1$-singularity]\label{ex:8-8}%
  Let $\what{Y} := 𝔸²$ and let $b : \what{Y} → Y$ be the quotient morphism for
  the action of the multiplicative group $± 1$, so that $Y$ has a unique
  singular point, which is of $A_1$ type.  We claim that the minimal resolution
  morphism $φ : X → Y$ is \emph{not} a morphism between the $\cC$-pairs $(X,0)$
  and $(Y,0)$.  In order to see this, construct a diagram as in
  \eqref{eq:7-1-1},
  \begin{equation}\label{eq:8-8-1}
    \begin{tikzcd}[row sep=1cm]
      \what{X} \ar[rrr, two heads, "\what{φ}\text{, blow-up}"] \ar[d, two heads, "a\text{, quotient}"'] &&& 𝔸² \ar[d, two heads, "b\text{, quotient}"] \ar[r, phantom, "="] & \what{Y} \\
      X \ar[rrr, two heads, "φ\text{, blow-up}"'] &&& \factor{𝔸²}{± 1} \ar[r, phantom, "="] & Y
    \end{tikzcd}
  \end{equation}
  The morphism $a$ is two-to-one and ramified exactly along the
  $\what{φ}$-exceptional curve in $\what{X}$.  A direct application of the
  definitions shows
  \[
    Ω^{[1]}_{(X, 0, a)} = a^* Ω¹_X ⊂ Ω¹_{\what{X}} \quad\text{while}\quad
    Ω^{[1]}_{(Y, 0, b)} = Ω¹_{\what{Y}}.
  \]
  Note however that the differential $d \what{φ} : \what{φ}^* Ω¹_{\what{Y}} →
  Ω¹_{\what{X}}$ does \emph{not} take its image in $a^* Ω¹_X$.
\end{example}

\begin{rem}[Kummer K3s]\label{rem:8-9}%
  Example~\ref{ex:8-8} shows in particular that the contraction morphism $φ : X
  \twoheadrightarrow Y = \factor{A}{±1}$ from a Kummer K3 surface to its
  associated torus quotient \emph{is not} a morphism between the $\cC$-pairs $(X,
  0)$ and $(Y, 0)$.  This observation will be of critical importance when we
  discuss the Albanese of a $\cC$-pair in the forthcoming paper \cite{orbialb2}.
\end{rem}

We continue this Example~\ref{ex:8-8} in Section~\vref{sec:10}, once suitable
criteria for $\cC$-morphisms have been established.  It will turn out that the
morphisms of Example~\ref{ex:8-8} and Remark~\ref{rem:8-9} do induce
$\cC$-morphisms once the correct, natural multiplicities of the exceptional sets
are taken into account.

% !TEX root = orbiAlb
%
% Do not edit the following line.  The text is automatically updated by
% subversion.
%
\svnid{$Id: 09-criteria.tex 911 2024-09-29 18:09:56Z kebekus $}
\selectlanguage{british}

\section{Criteria for \texorpdfstring{$\cC$}{ C}-morphisms}
\label{sec:9}
\subversionInfo
\approvals{Erwan & yes \\ Stefan & yes}
 
Throughout the present section, we consider Setting~\ref{set:7-1} in the special
case where the $q$-morphisms are \emph{(adapted) covers} and in particular
surjective.  We formulate our setup precisely and fix notation.

\begin{setting}\label{set:9-1}%
  Let $(X, D_X)$ and $(Y, D_Y)$ be two $\cC$-pairs and assume that there exists
  a commutative diagram as follows,
  \[
    \begin{tikzcd}[column sep=large]
      \what{X} \ar[r, "\what{φ}"] \ar[d, two heads, pos=0.4, "a\text{, arbitrary cover}"'] & \what{Y} \ar[d, two heads, pos=0.4, "b\text{, adapted cover}"] \\
      X \ar[r, "φ"'] & Y,
    \end{tikzcd}
  \]
  where $φ(X°) ⊆ Y°$ and $\img φ ∩ Y^{\lu} ≠ ∅$.  Use
  Notation~\ref{not:7-3}/Observation~\ref{obs:7-4} and consider the canonical
  pull-back morphisms
  \begin{equation}\label{eq:9-1-1}
    d_{\cC} \what{φ}^{\,+} : (\what{φ}^{\,+})^* Ω^{[•]}_{(Y,D_Y,b)} → Ω^•_{\what{X}⁺}(\log \what{D}_X⁺)
  \end{equation}
  Setting~\ref{set:9-1} ends here.
\end{setting}

If the morphism $\what{φ}$ of Setting~\ref{set:9-1} admits pull-back of adapted
reflexive differentials, one might be tempted to hope that $φ$ is then a
$\cC$-morphism.  This is not the case in general.  The following example shows
that \emph{one cannot check that a given morphism is a $\cC$-morphism by looking
at one pair of covers only, even if both covers are adapted}.

\begin{example}[Not enough to check one pair of adapted morphisms]\label{ex:9-2}
  Let $φ : X → Y$ be the minimal resolution of the $A_1$-singularity, as
  discussed in Example~\ref{ex:8-8}.  Take $a := \Id_X$ and $b := \Id_Y$ and
  note that the identity morphisms are adapted covers for the pairs $(X, 0)$ and
  $(Y, 0)$.  We are thus in Setting~\ref{set:9-1}.  It is then very obvious that
  $Ω^{[•]}_{(X,0,\Id_X)} = Ω^{•}_X$ and $Ω^{[•]}_{(Y,0,\Id_Y)} = Ω^{[•]}_Y$, and
  the existence of pull-back morphisms $π^* Ω^{[•]}_Y → Ω^•_X$ is precisely the
  content of the extension theorem for the $A_1$-singularity, see
  \cite[Thm.~1.1]{GKK08}, \cite[Thm.~1.5]{GKKP11} or \cite[Cor.~1.8]{KS18}.  But
  we have seen in Example~\ref{ex:8-8} that $φ$ is not a $\cC$-morphism.
\end{example}

In spite of the negative Example~\ref{ex:9-2}, there do exist relevant settings
where a look at \emph{one} adapted cover and \emph{one} value of $p$ suffices to
guarantee that a given morphism of varieties is in fact a morphism of
$\cC$-pairs.  The following proposition and its corollary identify two such
cases.

\begin{prop}[Criterion for $\cC$-morphisms]\label{prop:9-3}%
  In Setting~\ref{set:9-1}, assume that $Ω^{[1]}_{(Y, D_Y, b)}$ is locally free
  and that there exists a sheaf morphism
  \[
    η¹: \what{φ}^{\,*} Ω^{[1]}_{(Y, D_Y, b)} → Ω^{[1]}_{(X, D_X, a)}
  \]
  whose restriction to $\what{X}⁺$ agrees with the canonical pull-back morphisms
  $d_{\cC}\what{φ}^{\,+}$.  Then, $φ$ is a $\cC$-morphism between the pairs $(X,
  D_X)$ and $(Y, D_Y)$.
\end{prop}

Proposition~\ref{prop:9-3} will be shown in Section~\ref{sec:9-2} below.

\begin{cor}[Criterion for $\cC$-morphisms]\label{cor:9-4}%
  In Setting~\ref{set:9-1}, assume that $\what{Y}$ is smooth of dimension two
  and that there exists a sheaf morphism
  \[
    η¹: \what{φ}^{\,*} Ω^{[1]}_{(Y, D_Y, b)} → Ω^{[1]}_{(X, D_X, a)}
  \]
  whose restriction to $\what{X}⁺$ agrees with the canonical pull-back morphisms
  $d_{\cC}\what{φ}^{\,+}$.  Then, $φ$ is a $\cC$-morphism between the pairs $(X,
  D_X)$ and $(Y, D_Y)$.
\end{cor}
\begin{proof}
  Recall from \cite[Lem.~1.1.10]{OSS} that the reflexive sheaf $Ω^{[1]}_{(Y,
  D_Y, b)}$ is locally free and apply Proposition~\ref{prop:9-3}.
\end{proof}

\subsection{Elementary criteria for $\cC$-morphisms}
\approvals{Erwan & yes \\ Stefan & yes}

The proof of Proposition~\ref{prop:9-3} relies on the following lemma.  Since
the lemma and the subsequent criterion for a morphism of varieties to be a
$\cC$-morphism will be used several times in the sequel, we found it worth the
while to spell out all details.

\begin{lem}[Test for pull-back of adapted reflexive differentials]\label{lem:9-5}%
  Assume we are given two $\cC$-pairs, $(X, D_X)$ and $(Y, D_Y)$, and a
  commutative diagram of morphisms between normal, analytic varieties,
  \begin{equation}\label{eq:9-5-1}
    \begin{tikzcd}[column sep=large, row sep=large]
      \wcheck{X} \ar[r, "\wcheck{φ}"] \ar[d, two heads, "α\text{, cover}"'] & \wcheck{Y} \ar[d, two heads, "\:β\text{, cover}"] \\
      \what{X} \ar[r, "\what{φ}"] \ar[d, "a\text{, $q$-morphism}"'] & \what{Y} \ar[d, "b\text{, $q$-morphism}"] \\
      X \ar[r, "φ"'] & Y.
    \end{tikzcd}
  \end{equation}
  If $\wcheck{φ}$ admits pull-back of adapted reflexive differentials, then
  $\what{φ}$ admits pull-back of adapted reflexive differentials.
\end{lem}

We do not believe that the converse of Lemma~\ref{lem:9-5} holds in general.
Proposition~\ref{prop:9-3} and Corollary~\ref{cor:9-4} identify special
situations where a converse can be shown to hold.

\begin{proof}[Proof of Lemma~\ref{lem:9-5}]
  Assuming that $\wcheck{φ}$ admits pull-back of adapted reflexive
  differentials, we need to show that $\what{φ}$ admits pull-back of adapted
  reflexive differentials.  To spell things out: assuming we are given sheaf
  morphisms
  \[
    \wcheck{η} : \wcheck{φ}^* Ω^{[•]}_{(Y,D_Y,b◦ β)} → Ω^{[•]}_{(X,D_X,a◦α)}
  \]
  whose restrictions to $\wcheck{X}⁺$ agree with the pull-back morphisms
  $d_{\cC} \wcheck{φ}⁺$, we need to construct appropriate morphisms
  \[
    \what{η} : \what{φ}^{\,*} Ω^{[•]}_{(Y,D_Y,b)} → Ω^{[•]}_{(X,D_X,a)}
  \]
  whose restrictions to $\what{X}⁺$ agree with the pull-back morphisms $d_{\cC}
  \what{φ}^{\,+}$.  In analogy to Construction~\ref{cons:5-7}, we consider sheaf
  morphisms on $\wcheck{X}$,
  \begin{align}
    α^* \what{φ}^{\,*} Ω^{[•]}_{(Y,D_Y,b)} & = \wcheck{φ}^* β^* Ω^{[•]}_{(Y,D_Y,b)} && \txt{commutativity} \nonumber \\
    & → \wcheck{φ}^* β^{[*]} Ω^{[•]}_{(Y,D_Y,b)} && \text{natural} \label{eq:9-5-2} \\
    & → \wcheck{φ}^* Ω^{[•]}_{(Y,D_Y,b◦β)} && \text{Observation~\ref{obs:4-14}} \nonumber \\
    & → Ω^{[•]}_{(X,D_X,a◦α)} && \wcheck{η}, \nonumber \\
    \intertext{and take $\what{η}$ as the composition}
    \what{φ}^{\,*} Ω^{[•]}_{(Y,D_Y,b)} & → α_* α^* \what{φ}^{\,*} Ω^{[•]}_{(Y,D_Y,b)} && \text{natural} \\
    & → α_* Ω^{[•]}_{(X,D_X,a◦α)} && α_* \text{\eqref{eq:9-5-2}} \\
    & → Ω^{[•]}_{(X,D_X,a)} && \text{Consequence~\ref{cons:4-18}}
  \end{align}
  We leave it to the reader to apply Fact~\ref{fact:5-12} (``Functoriality'') in
  order to check that the restrictions of $\what{η}$ to $\what{X}⁺$ indeed agree
  with $d_{\cC} \what{φ}^{\,+}$.
\end{proof}

The following criterion is a direct consequence of Lemma~\ref{lem:9-5}.  In a
nutshell, it asserts that to check if a given morphism is a $\cC$-morphism, it
suffices to restrict attention to those diagrams of the form \eqref{eq:7-1-1}
where the $q$-morphisms $a$ and $b$ are adapted.

\begin{cor}[Elementary criterion for $\cC$-morphisms]\label{cor:9-6}%
  Given $\cC$-pairs $(X, D_X)$ and $(Y, D_Y)$ and a morphism $φ : X → Y$, assume
  that $φ(X°) ⊆ Y°$ and $\img φ ∩ Y^{\lu} ≠ ∅$.  Also, assume that every
  commutative diagram of the form
  \[
    \begin{tikzcd}[column sep=large]
      \what{X} \ar[r, "\what{φ}"] \ar[d, "a\text{, adapted}"'] & \what{Y} \ar[d, "b\text{, adapted}"] \\
      X \ar[r, "φ"'] & Y
    \end{tikzcd}
  \]
  admits pull-back of adapted reflexive differentials.  Then, $φ$ is a
  $\cC$-morphism between the $\cC$-pairs $(X, D_X)$ and $(Y, D_Y)$.
\end{cor}
\begin{proof}
  Observation~\ref{obs:8-5} (``being a $\cC$-morphism is local on $X$ and
  $Y$'') and Lemma~\ref{lem:2-36} (``strongly adapted covers exist locally'')
  allow assuming without loss of generality that $X$ and $Y$ admit strongly
  adapted covers.  As a consequence, we find that every $q$-morphism to $X$ and
  $Y$ can be refined to an adapted morphism via an elementary fibre product
  construction.  More precisely, every commutative diagram of form
  \eqref{eq:7-1-1} can be extended to a diagram of the form \eqref{eq:9-5-1},
  with the additional property that $α◦a$ and $β◦b$ are adapted for $(X, D_X)$
  and $(Y, D_Y)$, respectively.  Lemma~\ref{lem:9-5} asserts that to prove that
  $\what{φ}$ admits pull-back of adapted differentials, it suffices to show that
  $\wcheck{φ}$ admits pull-back of adapted differentials.  That, however, holds
  by assumption.
\end{proof}

\subsection{Proof of Proposition~\ref*{prop:9-3}}
\approvals{Erwan & yes \\ Stefan & yes}
\label{sec:9-2}

Given a diagram as in Setting~\ref{set:7-1},
\[
  \begin{tikzcd}
    \what{X}° \ar[d, "a°\text{, $q$-morphism}"'] \ar[r, "\what{φ}°"] & \what{Y}° \ar[d, "b°\text{, $q$-morphism}"] \\
    X \ar[r, "φ"'] & Y,
  \end{tikzcd}
\]
we need to show that $\what{φ}°$ admits pull-back of adapted reflexive
differentials.  With this in mind, consider common covers of $\what{Y}°$ and
$\what{Y}$, and $\what{X}°$ and $\what{X}$ respectively,
\begin{align*}
  \wcheck{Y} & := \text{normalisation of a component of } \what{Y}° ⨯_Y \what{Y} \\
  \wcheck{X} & := \text{normalisation of a component of } \bigl( \what{X}° ⨯_X \what{X} \bigr) ⨯_{\what{Y}°} \wcheck{Y}.
\end{align*}
The following diagram summarizes the situation,
\begin{equation}
  \begin{tikzcd}[row sep=1cm]
    \wcheck{X} \ar[d, two heads, "α°"'] \ar[r, equals] & \wcheck{X} \ar[d, "α"'] \ar[rrrr, "\wcheck{φ}"] &&&& \wcheck{Y} \ar[d, "β"] \ar[r, equals] & \wcheck{Y} \ar[d, two heads, "β°"] \\
    \what{X}° \ar[d, "a°"'] & \what{X} \ar[d, pos=0.6, two heads, "a"'] \ar[rrrr, "\what{φ}"] &&&& \what{Y} \ar[d, pos=0.6, two heads, "b"] & \what{Y}° \ar[d, "b°"] \ar[from=llllll, bend right=20, crossing over, "\what{φ}°"'] \\
    X \ar[r, equals] & X \ar[rrrr, "φ"'] &&&& Y \ar[r, equals] & Y.
  \end{tikzcd}
\end{equation}
As before, we use Lemma~\ref{lem:9-5} (``Test for pull-back of adapted reflexive
differentials'') and find that it suffices to show that $\wcheck{φ}$ admits
pull-back of adapted reflexive differentials.  In order to construct the
relevant morphisms
\[
  \wcheck{η}: \wcheck{φ}^{\,*} Ω^{[•]}_{(Y, D_Y, β◦ b)} → Ω^{[•]}_{(X, D_X, α◦ a)},
\]
consider the identifications
\begin{align}
  \wcheck{φ}^* Ω^{[•]}_{(Y,D_Y,b◦β)} & = \wcheck{φ}^* β^{[*]} Ω^{[•]}_{(Y,D_Y,b)} && \text{Observation~\ref{obs:4-15}, $b$ adapted} \nonumber \\
  & = \wcheck{φ}^* β^* Ω^{[•]}_{(Y,D_Y,b)} && \text{local freeness} \label{eq:9-6-2} \\
  & = α^* \what{φ}^{\,*} Ω^{[•]}_{(Y,D_Y,b)} && \text{commutativity.} \nonumber
  \intertext{In case $• = 1$, the last sheaf admits morphisms as follows,}
  α^* \what{φ}^{\,*} Ω^{[1]}_{(Y,D_Y,b)} & → α^* Ω^{[1]}_{(X,D_X,a)} && α^* η¹ \nonumber \\
  & → α^{[*]} Ω^{[1]}_{(X,D_X,a)} && \text{natural} \label{eq:9-6-3} \\
  & → Ω^{[1]}_{(X,D_X,a◦α)} && \text{Observation~\ref{obs:4-14}.} \nonumber
  \intertext{In $• = p$ is arbitrary, take $\wcheck{η}$ as the composed morphism}
  \wcheck{φ}^{\,*} Ω^{[p]}_{(Y, D_Y, β◦b)} & = α^* \what{φ}^{\,*} Ω^{[p]}_{(Y,D_Y,b)} && \text{\eqref{eq:9-6-2}} \nonumber \\
  & = Λ^p α^* \what{φ}^{\,*} Ω^{[p]}_{(Y,D_Y,b)} && \text{local freeness} \nonumber \\
  & → Λ^p Ω^{[1]}_{(X,D_X,a◦α)} && Λ^p \text{\eqref{eq:9-6-3}} \nonumber \\
  & → Ω^{[p]}_{(X,D_X,a◦α)} && \text{natural} \nonumber
\end{align}
As before, we leave it to the reader to verify that the morphisms $\wcheck{η}$
agree over $\wcheck{X}⁺$ with the canonical pull-back morphisms
$d_{\cC}\wcheck{φ}⁺$.  \qed

% !TEX root = orbiAlb1
%
% Do not edit the following line.  The text is automatically updated by
% subversion.
%
\svnid{$Id: 10-examples.tex 852 2024-07-15 08:20:56Z kebekus $}
\selectlanguage{british}

\section{Examples and counterexamples}
\subversionInfo
\label{sec:10}

\subsection{Resolution of the $A_1$-singularity}
\approvals{Erwan & yes \\ Stefan & yes}

To illustrate the use of Proposition~\ref{prop:9-3} (``Criterion for
$\cC$-morphisms''), we continue our discussion on the resolution of the
$A_1$-singularity.

\begin{example}[Resolution of the $A_1$-singularity]\label{ex:10-1}%
  One concrete example of Setting~\ref{set:9-1} is given in
  Diagram~\eqref{eq:8-8-1} of Example~\vref{ex:8-8}.  Continuing the notation
  of the example, denote the $φ$-exceptional locus by $E ⊊ X$ and observe that
  the two-to-one cover $a$ is adapted for the pair $\Bigl(X,
  \frac{1}{2}·E\Bigr)$.  A direct application of the definitions shows
  \[
    Ω^{[1]}_{\bigl(X, \frac{1}{2}·E, a\bigr)} = Ω¹_{\what{X}}
    \quad\text{and}\quad
    Ω^{[1]}_{(Y, 0, b)} = Ω¹_{\what{Y}},
  \]
  so that $Ω^{[1]}_{(Y, 0, b)}$ is locally free and that there exists a sheaf
  morphism
  \[
    \diff \what{φ}° \::\: \what{φ}^{\,*} Ω^{[1]}_{(Y, 0, b)} → Ω^{[1]}_{\bigl(X, \frac{1}{2}·E, a\bigr)}
  \]
  that agrees with the standard pull-back of Kähler differentials, namely
  pull-back of Kähler differentials itself.  Proposition~\ref{prop:9-3}
  therefore applies to say that $φ$ yields a morphism of $\cC$-pairs,
  \[
    φ : \Bigl(X, \textstyle{\frac{1}{2}}·E \Bigr) → \Bigl( \factor{𝔸²}{±1},\, 0\Bigr).
  \]
\end{example}

\begin{rem}[Kummer K3s]
  Example~\ref{ex:10-1} shows in particular that the contraction morphism $φ : X
  → Y := \factor{A}{±1}$ from a Kummer K3 surface to its associated torus
  quotient induces a $\cC$-morphism between $\Bigl(X, \frac{1}{2}·E \Bigr)$ and
  $(Y, 0)$, if $E ⊂ X$ is the $φ$-exceptional locus.  Again, this observation
  will be of critical importance when we construct the Albanese of a $\cC$-pair
  in a forthcoming paper.
\end{rem}

\subsection{Inclusion of boundary components}
\approvals{Erwan & yes \\ Stefan & yes}

$\cC$-morphisms may take their images inside the boundary divisors of the target
space.  The following example shows the simplest setting.

\begin{example}[Inclusion of boundary components]
  Consider a snc $\cC$-pair $(X, D)$ where the Weil $ℚ$-divisor $D$ is of the
  form
  \[
    D = \sum_i \frac{m_i-1}{m_i}·D_i, \quad\text{all }m_i ∈ ℕ^{≥ 2}.
  \]
  Pick one component $D_0$ and define
  \[
    D^c_0 := \sum_{i ≠ 0} \frac{m_i-1}{m_i}·D_i|_{D_0} ∈ ℚ\Div(D_0).
  \]
  The $\cC$-pair $(D_0, D^c_0)$ is then snc, and we claim that the inclusion $ι
  : D_0 → X$ is a $\cC$-morphism between the pairs $(D_0, D^c_0)$ and $(X,D)$.
  Since the claim is local, we may assume without loss of generality that $X =
  𝔹 ⊂ ℂ^n$ is the unit ball, that $\supp D ⊊ X$ is a union of hyperplanes, and
  that we are given a uniformization
  \[
    γ : \what{X} → X, \quad (x_1, x_2, …, x_n) ↦ \bigl(x^{a_1}_1, x^{a_2}_2, …, x^{a_n}_n\bigr),
  \]
  where $\what{X} := 𝔹$ is again the unit ball.  Consider the preimage
  $\what{D}_0 := γ^{-1}(D_0) ⊊ \what{X}$ with its reduced structure and observe
  that $γ|_{\what{D}_0} : \what{D}_0 → D_0$ is again a uniformization.  We
  obtain a diagram
  \[
    \begin{tikzcd}[column sep=2cm, row sep=1cm]
      \what{D}_0 \ar[r, hook, "ι\text{,inclusion}"] \ar[d, two heads, "γ|_{\what{D}_0}\text{, uniformization}"'] & \what{X} \ar[d, two heads, "γ\text{, uniformization}"] \\
      D_0 \ar[r, hook, "\text{inclusion}"'] & X.
    \end{tikzcd}
  \]
  Observing that
  \[
    Ω^{[•]}_{(D_0, D^c_0,γ|_{\what{D}_0})} = Ω^{•}_{\what{D}_0}
    \quad\text{and}\quad
    Ω^{[•]}_{(X, D, γ)} = Ω^{•}_{\what{X}}
  \]
  and that $d_{\cC} ι : ι^* Ω^{[•]}_{(X, D, γ)} →
  Ω^{•}_{\what{D}_0}$ is the restriction of Kähler differentials,
  Proposition~\ref{prop:9-3} yields the claim.
\end{example}

\subsection{Comparison of divisors}
\approvals{Erwan & yes \\ Stefan & yes}

We continue with a criterion for the identity morphism to be a $\cC$-morphism
under a change of divisors.  While perhaps trivial, the criterion is so useful
that it deserves to be mentioned and carefully proven.

\begin{prop}[Comparison of divisors]\label{prop:10-4}%
  Let $(X, D_{1,X})$ and $(X, D_{2,X})$ be two $\cC$-pairs on the same
  underlying space.  Then, the following statements are equivalent.
  \begin{enumerate}
    \item\label{il:10-4-1} The identity morphism $\Id_X$ is a $\cC$-morphism
    between $(X, D_{1,X})$ and $(X, D_{2,X})$.
    
    \item\label{il:10-4-2} We have $D_{1,X} ≥ D_{2,X}$.
  \end{enumerate}
\end{prop}
\begin{proof}[Proof of Proposition~\ref{prop:10-4}, \ref{il:10-4-1} $⇒$ \ref{il:10-4-2}]
  If $x ∈ X$ is any point where $\bigl(X, D_{1,X}+D_{2,X} \bigr)$ is nc, there
  exists a $q$-morphism $γ : \what{X} → X$ where $\bigl(\what{X}, γ^*
  (D_{1,X}+D_{2,X})\bigr)$ is nc, where $γ$ is adapted for $(X, D_{1,X})$ and
  for $(X, D_{2,X})$, and contains $x$ in its image.  We obtain a diagram as
  follows,
  \[
    \begin{tikzcd}
      \what{X} \ar[r, "\Id_{\what{X}}"] \ar[d, "γ\text{, adapted $q$-morphism}"'] & \what{X} \ar[d, "γ\text{, adapted $q$-morphism}"] \\
      X \ar[r, "\Id_X"'] & X.
    \end{tikzcd}
  \]
  The morphism $\Id_{\what{X}}$ admits pull-back of adapted reflexive
  differentials by assumption.  In other words, there exists an inclusion of
  sheaves, $Ω¹_{(X, D_{2,X}, γ)} ⊆ Ω¹_{(X, D_{1,X}, γ)}$.  The definition of
  $Ω¹_{(X, D_{•,X}, γ)}$ will then imply that the inequality $D_{1,X} ≥ D_{2,X}$
  holds over the open set $\img(γ)$.  The claim follows since $x$ is an
  arbitrary point in a big open subset of $X$.
\end{proof}
\begin{proof}[Proof of Proposition~\ref{prop:10-4}, \ref{il:10-4-2} $⇒$ \ref{il:10-4-1}]
  Given a diagram as follows
  \[
    \begin{tikzcd}
      \what{X}_1 \ar[r, "\what{φ}"] \ar[d, "γ_1\text{, $q$-morphism}"'] & \what{X}_2 \ar[d, "γ_2\text{, $q$-morphism}"] \\
      X \ar[r, "\Id_X"'] & X,
    \end{tikzcd}
  \]
  we need to show that $\what{φ}$ admits pull-back of adapted reflexive
  differentials.  A few simplifications can be made without loss of generality.
  To begin, observe that $\what{φ}$ is $q$-morphism and hence open by
  Reminder~\vref{remi:2-19}.  Replacing $\what{X}_2$ with the image of
  $\what{φ}$, we may therefore assume without loss of generality that $\what{φ}$
  is surjective.  We can then invoke Lemma~\ref{lem:9-5}, replace $γ_2$ by
  $γ_2◦\what{φ}$ and assume without loss of generality that $γ_1$ and $γ_2$ are
  equal, and that $\what{φ}$ is the identity map on $\what{X}_1 = \what{X}_2$.
  
  With these simplifications in place, it is clear that $\what{φ}$ admits
  pull-back of adapted reflexive differentials if and only if we have inclusions
  \[
    Ω^{[p]}_{(X, D_{2,X}, γ)} ⊆ Ω^{[p]}_{(X, D_{1,X}, γ)} \quad \text{for every }p.
  \]
  The inequality $D_{1,X} ≥ D_{2,X}$ will however guarantee that.
\end{proof}

\subsection{$\cC$-Resolutions of singularities}
\approvals{Erwan & yes \\ Stefan & yes}

Regretfully, it follows almost immediately from the definition of $\cC$-morphism
that resolutions of singularities do not always exist.

\begin{defn}[$\cC$-Resolution of singularities]\label{def:10-5}%
  Let $(X, D)$ be a $\cC$-pair.  A \emph{$\cC$-resolution of
  singularities}\index{C-resolution of singularities@$\cC$-resolution of
  singularities} is a proper, bimeromorphic $\cC$-morphism $π : (\wtilde{X},
  \wtilde{D}) → (X, D)$ where $(\wtilde{X}, \wtilde{D})$ is snc and $π_*
  \wtilde{D} = D$.
\end{defn}

\begin{prop}[Necessary criterion for existence $\cC$-resolutions of singularities]\label{prop:10-6}%
  Let $X$ be a normal analytic variety.  Assume that $X$ is Gorenstein and that
  a $\cC$-resolution of singularities $π : (\wtilde{X}, \wtilde{D}) → (X, 0)$
  exists.  Then, $X$ is log canonical.  In particular, $X$ has Du Bois
  singularities.
\end{prop}

Once appropriate criteria are established, we generalize
Proposition~\ref{prop:10-6} in Corollary~\ref{cor:13-5} to locally
$ℚ$-Gorenstein $\cC$-pairs with non-trivial boundary.

\begin{proof}[Proof of Proposition~\ref{prop:10-6}]
  Let $E ∈ \Div(\wtilde{X})$ denote the $π$-exceptional divisor, with its
  reduced structure.  Consider the trivial diagram
  \[
    \begin{tikzcd}[column sep=large]
      \wtilde{X} \ar[r, "π"] \ar[d, "\Id_{\wtilde{X}}\text{, $q$-morphism}"'] & X \ar[d, "\Id_X\text{, $q$-morphism}"] \\
      \wtilde{X} \ar[r, "π"'] & X
    \end{tikzcd}
  \]
  Definition~\ref{def:8-1} guarantees that $φ$ admits pull-back of adapted
  reflexive differentials,
  \begin{align*}
    π^* ω_X & = π^* \Bigl( Ω^{[\dim X]}_{(X, 0, \Id_X)} \Bigr) && \text{Example~\ref{ex:4-6}} \\
    & → Ω^{[\dim X]}_{(\wtilde{X}, \wtilde{D}, \Id_{\wtilde{X}})} && \text{pull-back} \\
    & = ω_{\wtilde{X}}(\log ⌊\wtilde{D}⌋) ⊆ ω_{\wtilde{X}}(\log E) = ω_{\wtilde{X}}(E) && \text{Example~\ref{ex:4-6}}.
  \end{align*}
  By definition, \cite[Sec.~2.3]{KM98}, this means that
  $\operatorname{discrep}(X,0) ≥ -1$ so that $X$ is log canonical.
\end{proof}

\begin{example}[$\cC$-pair without $\cC$-resolution of singularities]
  For a concrete example of a variety that is Gorenstein but not log canonical,
  let $Y ⊊ ℙ²$ be any general type curve, and let $X ⊂ 𝔸³$ be the affine cone
  over $Y$ with normal bundle $𝒪_{ℙ²}(1)|_Y$, as discussed in
  \cite[App.~B]{KS18}.  The variety $X$ is then normal.  As a hypersurface in
  $𝔸³$, it is also Gorenstein.  However, we have seen in \cite[Prop.~B.3]{KS18}
  that $X$ is not log canonical.
\end{example}

For future reference, we note the following variant of
Proposition~\ref{prop:10-6}, which relates the existence of $\cC$-resolutions of
singularities to the notion of ``weakly rational'' singularities, as introduced
in \cite[Sect.~1.4 and Def.~A.1]{KS18}.

\begin{prop}[Necessary criterion for the existence of special $\cC$-resolutions of singularities]\label{prop:10-8}%
  Let $X$ be a $\cC$-pair with $⌊D⌋ = 0$.  Assume that there exists a
  $\cC$-resolution of singularities,
  \[
    π : (\wtilde{X}, \wtilde{D}) → (X, D),
  \]
  where $⌊\wtilde{D}⌋ = 0$.  Then, $X$ has weakly rational singularities in the
  sense of \cite[Def.~A.1]{KS18}.
\end{prop}

\begin{rem}[Rational and weakly rational singularities]
  If the space $X$ of Proposition~\ref{prop:10-8} is Cohen-Macaulay, recall from
  \cite[Sect.~1.4 and references there]{KS18} that $X$ has weakly rational
  singularities if and only if it has rational singularities.
\end{rem}

\begin{proof}[Proof of Proposition~\ref{prop:10-8}]
  As in the proof of Proposition~\ref{prop:10-6}, we obtain a pull-back map for
  adapted reflexive differentials,
  \[
    π^* ω_X
    = π^* \Bigl( Ω^{[\dim X]}_{(X, D, \Id_X)} \Bigr)
    → Ω^{[\dim X]}_{(\wtilde{X}, \wtilde{D}, \Id_{\wtilde{X}})}
    = ω_{\wtilde{X}}.
  \]
  In the language of \cite[Sect.~1.4]{KS18}, this implies that the
  Grauert-Riemenschneider sheaf on $X$ equals its dualizing sheaf,
  \[
    ω^{\operatorname{GR}}_X \overset{\text{def.}}{=} π_* ω_{\wtilde{X}} = ω_X,
  \]
  which is reflexive.  By definition, this means that $X$ is weakly rational
  singularities.
\end{proof}

\begin{prop}[Sufficient criterion for existence of $\cC$-resolutions of singularities]\label{prop:10-10}%
  Let $(X,D)$ be a locally uniformizable $\cC$-pair.  Then, a $\cC$-resolution
  of singularities exists.
\end{prop}

\begin{rem}[Proposition~\ref{prop:10-10} is not optimal]
  Proposition~\ref{prop:10-10} is far from optimal.  It is certainly possible to
  bound the coefficients of the resolution pair.
\end{rem}

\begin{proof}[Proof of Proposition~\ref{prop:10-10}]
  \CounterStep{}Let $π : Y → X$ be a strong log resolution of the pair $(X,D)$,
  with exceptional divisor $E ⊂ Y$.  Consider the divisor $D_Y := π^{-1}_* D + E
  ∈ \Div Y$, where $π^{-1}_* D$ denotes the strict transform.  We claim that $π$
  is a morphism between the $\cC$-pairs $\bigl(Y, D_Y\bigr)$ and $(X,D)$.
  Recalling from Observation~\ref{obs:8-5} that the claim is local on $X$, we
  assume without loss of generality that $X$ is uniformizable, so that a diagram
  of the following form exists,
  \begin{equation}\label{eq:10-12-1}
    \begin{tikzcd}[row sep=1cm, column sep=2cm]
      \what{Y} \ar[r, "\what{π}"] \ar[d, two heads, "γ_Y"'] & \what{X} \ar[d, two heads, "γ_X\text{, uniformization}"] \\
      Y \ar[r, "π\text{, resolution}"'] & X,
    \end{tikzcd}
  \end{equation}
  where $\what{Y}$ is obtained as the normalization of a suitable component in
  the fibre product $Y ⨯_X \what{X}$.  Further simplifications are possible:
  Observation~\ref{obs:8-5} allows assuming without loss of generality that $Y$
  and $\what{Y}$ are smooth, and that the divisors $D_Y$ and $γ^*_Y D_Y$ have
  smooth support.

  Since $γ_X$ uniformizes, we have seen in Example~\ref{ex:4-6} that
  $Ω^{[1]}_{(X,D,γ_X)}$ is locally free.  Proposition~\ref{prop:9-3} therefore
  applies to show that $π$ is a morphism of $\cC$-pairs if and only if
  Diagram~\eqref{eq:10-12-1} admits pull-back of adapted reflexive
  differentials.  That, however, follows almost immediately from our choice of a
  boundary divisor on $Y$ and from Fact~\vref{fact:5-9}.  To be precise, recall
  that Fact~\ref{fact:5-9} equips us with a pull-back map
  \begin{align}
    \what{π}^{\,*} Ω¹_{\what{X}} \bigl(\log γ^*_X ⌊D_X⌋ \bigr) & = \what{π}^{\,*} Ω^{[1]}_{(X,D_X,γ_X)} && γ_Y\text{ uniformizes} \nonumber \\
    & → Ω¹_{\what{Y}} \bigl(\log \what{π}^{\,*} γ^*_X ⌊D_X⌋ \bigr) && \text{pull-back }d_{\cC} φ \label{eq:10-12-2}\\
    & = Ω¹_{\what{Y}} \bigl(\log γ^*_Y π^* ⌊D_X⌋ \bigr) && \text{commutativity} \nonumber \\
    & ↪ Ω¹_{\what{Y}} \bigl(\log γ^*_Y ⌊D_Y⌋\bigr) && \text{choice of } D_Y.  \nonumber
  \end{align}
  To prove that Diagram~\eqref{eq:10-12-1} admits pull-back of adapted reflexive
  differentials, we need to show that the composed map \eqref{eq:10-12-2} takes
  its image in
  \begin{equation}\label{eq:10-12-3}
    Ω¹_{(Y,D_Y,γ_Y)} ⊆ Ω¹_{\what{Y}} \bigl(\log γ^*_Y ⌊D_Y⌋\bigr).
  \end{equation}
  This is clear over the complement of $E$, where $π$ and $\what{π}$ are
  isomorphic.  This is also clear over the complement of $π^{-1}_* D$, where
  $D_Y$ is reduced.  Recalling the assumption that $D_Y$ has smooth support,
  observe that $E$ and $π^{-1}_* D$ are disjoint, so that \eqref{eq:10-12-3}
  holds everywhere.
\end{proof}

% !TEX root = orbiAlb1
%
% Do not edit the following line.  The text is automatically updated by
% subversion.
%
\svnid{$Id: 11-functoriality.tex 852 2024-07-15 08:20:56Z kebekus $}
\selectlanguage{british}

\section{Functoriality}
\subversionInfo
\approvals{Erwan & yes \\ Stefan & yes}
\label{sec:11}

Observe that $\cC$-pairs form no category because they cannot be composed.  If
\begin{equation}\label{eq:11-0-1}
  (X, D_X) \xrightarrow{φ_1} (Y, D_Y) \xrightarrow{φ_2} (Z, D_Z)
\end{equation}
is a sequence of morphisms of $\cC$-pairs, the composition $φ_2◦φ_1$ need not be
a $\cC$-morphism between $(X, D_X)$ and $(Z, D_Z)$, for the simple reason that
the image of the composed morphism might be disjoint from the open set $Z^{\lu}
⊆ Z$ where $(Z, D_Z)$ is locally uniformizable.  The following proposition
asserts that this is the only obstacle for functoriality.  It implies in
particular that locally uniformizable $\cC$-pairs form a category.

\begin{prop}[Weak functoriality]\label{prop:11-1}%
  Given a sequence of morphisms between $\cC$-pairs as in \eqref{eq:11-0-1},
  assume that
  \[
    \img (φ_2◦φ_1) ∩ Z^{\lu} ≠ ∅.
  \]
  Then, $φ_2◦φ_1$ is a morphism between the $\cC$-pairs $(X, D_X)$ and $(Z,
    D_Z)$.
\end{prop}
\begin{proof}
  We apply the elementary criterion for $\cC$-morphisms spelled out in
  Corollary~\ref{cor:9-6} above: assuming that we have a diagram
  \[
    \begin{tikzcd}[column sep=large]
      \what{X} \ar[r, "\what{φ}"] \ar[d, "a\text{, adapted}"'] & \what{Z} \ar[d, "c\text{, adapted}"] \\
      X \ar[r, "φ_2◦φ_1"'] & Z,
    \end{tikzcd}
  \]
  we need to show that $\what{φ}$ admits pull-back of adapted reflexive
  differentials.  The question is local on $X$.  We can therefore shrink all
  relevant spaces and assume that $Y$ is Stein.  Lemma~\ref{lem:2-36} will then
  yield an adapted cover $\overline{Y} \twoheadrightarrow Y$.  We can thus set
  \begin{align*}
    \wcheck{Y} & := \text{normalisation of a component of } \overline{Y} ⨯_Z \what{Z}, \\
    \wcheck{X} & := \text{normalisation of a component of } \wcheck{Y} ⨯_Y \what{X},
  \end{align*}
  and obtain a diagram as follows,
  \[
    \begin{tikzcd}[column sep=large]
      \wcheck{X} \ar[d, two heads, "α"'] \ar[r, "\wcheck{φ}_1"] & \wcheck{Y} \ar[dd, two heads, near start, "β"'] \ar[r, "\wcheck{φ}_2"] & \wcheck{Z} \ar[d, equals, "γ"] \\
      \what{X} \ar[d, "a"'] & & \what{Z} \ar[d, "c"] \ar[from=ll, crossing over, near start, "\what{φ}"'] \\
      X \ar[r, "φ_1"'] & Y \ar[r, "φ_2"'] & Z.
    \end{tikzcd}
  \]
  Lemma~\ref{lem:9-5} applies to this setting, so that we only need to show that
  $\wcheck{φ}_2◦\wcheck{φ}_1$ admits pull-back of adapted reflexive
  differentials.  But each morphism $\wcheck{φ}_{•}$ admits pull-back of adapted
  reflexive differentials individually.
\end{proof}

% !TEX root = orbiAlb1
%
% Do not edit the following line.  The text is automatically updated by
% subversion.
%
\svnid{$Id: 12-quotients.tex 852 2024-07-15 08:20:56Z kebekus $}
\selectlanguage{british}

\section{Existence of categorical quotients}
\subversionInfo
\approvals{Erwan & yes \\ Stefan & yes}
\label{sec:12}

$\cC$-pairs admit a natural notion of a categorical quotient under the action of
a finite group.  The following definitions are direct analogues of the classic
definitions for normal varieties and should not come as a surprise.

\begin{notation}[Group action on $\cC$-pair]
  \index{group action on $\cC$-pair} Let $(X, D_X)$ be a $\cC$-pair and let $G$
  be a finite group.  A \emph{$G$-action on $(X, D_X)$}\index{group action on
  $\cC$-pair} is a $G$-action on $X$ that stabilizes the divisor $D_X$.
\end{notation}

\begin{rem}
  The requirement that the action stabilizes $D_X$ is equivalent to the
  requirement that $g^* D_X = D_X$ for all $g ∈ G$.  We do not require that the
  $G$-action stabilizes the components of $D_X$ individually, nor that it fixes
  them pointwise.
\end{rem}

\begin{defn}[Categorical quotients of $\cC$-pairs]\label{def:12-3}%
  Let $G$ be a finite group that acts on a $\cC$-pair $(X, D_X)$.  A
  \emph{categorical quotient of $(X, D_X)$ by $G$}\index{quotient of
  $\cC$-pair}\index{categorical quotient of $\cC$-pair} is a surjective morphism
  of $\cC$-pairs,
  \[
    q : (X, D_X) \twoheadrightarrow (Q, D_Q),
  \]
  whose underlying morphism of varieties is constant on $G$-orbits and that
  satisfies the following universal property: if $φ : (X,D_X) → (Y,D_Y)$ is any
  morphism of $\cC$-pairs whose underlying morphism of varieties is constant on
  $G$-orbits, then $φ$ factorizes via $q$,
  \[
    \begin{tikzcd}[column sep=large]
      (X, D_X) \ar[r, two heads, "q"'] \ar[rr, bend left=15, "φ"] & (Q, D_Q) \ar[r, "∃!  ψ"'] & (Y,D_Y).
    \end{tikzcd}
  \]
\end{defn}

At the level of underlying spaces, the quotient is simply the categorical
quotient of a normal analytic space, \cite[Thm.~4]{MR0084174}.  The following
construction equips the quotient space with a suitable divisor.

\begin{construction}\label{cons:12-4}%
  Let $(X, D_X)$ be a $\cC$-pair and let $G$ be a finite group that acts on
  $(X,D_X)$.  Set $Q := X/G$ and take $q : X → Q$ as the quotient morphism.  For
  any prime Weil divisor $H ∈ \Div(Q)$, choose a component of the preimage $H' ⊂
  \supp q^ *H$ and set
  \[
    m_H := (\mult_{H'} q^* H) · (\mult_{\cC, H'} D_X).
  \]
  Observe that the number $m_H$ does not depend on the choice of $H'$ and set
  \[
    D_Q := \sum_{H ∈ \Div(Q) \text{ prime}} \frac{m_H - 1}{m_H}·H ∈
    ℚ\Div(Q).
  \]
  As before, we stick to the convention that
  \[
    ∞·(\text{positive, finite}) = ∞,
    \quad ∞ - (\text{finite}) = ∞,
    \quad\text{and}\quad ∞/∞ = 1.
  \]
\end{construction}

\begin{thm}[Existence of quotients]\label{thm:12-5}%
  In the setting of Construction~\ref{cons:12-4}, the morphism $q$ is a morphism
  of $\cC$-pairs, $q: (X, D_X) → (Q, D_Q)$, and this morphism is a categorical
  quotient.
\end{thm}

The proof of Theorem~\ref{thm:12-5} is elementary, but the somewhat delicate
notion of $\cC$-morphism does require some attention and makes the proof a
little more technical than we would have preferred.  For the reader's
convenience, we defer the proof until Subsection~\ref{sec:12-4} and discuss a
few properties of the quotient construction first.

\begin{notation}
  The universal property implies immediately that categorical quotients are
  unique up to unique isomorphism.  We will therefore speak of ``the'' quotient
  and denote the quotient $\cC$-pair by the symbol $(X, D_X)/G$.
\end{notation}

\subsection{Functoriality}
\approvals{Erwan & yes \\ Stefan & yes}

The universal property in the definition of ``quotient'' will often be used in
the following form, which we formulate separately for the reader's convenience.

\begin{prop}[Functoriality]\label{prop:12-7}%
  Let $G$ be a finite group that acts on two $\cC$-pairs, $(X, D_X)$ and $(Y,
  D_Y)$.  Let $γ : (X, D_X) → (Y, D_Y)$ be a $G$-equivariant morphism of
  $\cC$-pairs.  Then, there exists an induced morphism between categorical
  quotients, and a commutative diagram of morphisms between $\cC$-pairs as
  follows,
  \[
    \begin{tikzcd}[row sep=large]
      (X, D_X) \arrow[r, "γ"] \arrow[d, two heads, "q_X\text{, quotient}"'] & (Y, D_Y) \arrow[d, two heads, "q_Y\text{, quotient}"] \\
      \factor{(X, D_X)}{G} \arrow[r, "γ^G"'] & \factor{(Y, D_Y)}{G}.
    \end{tikzcd}
  \]
\end{prop}

Proposition~\ref{prop:12-7} is an almost immediate consequence of the following
lemma, which we show in Section~\ref{sec:12-3} below.

\begin{lem}[Quotients of uniformizable pairs]\label{lem:12-8}%
  Quotients of uniformizable pairs are uniformizable.  Quotients of locally
  uniformizable pairs are locally uniformizable.
\end{lem}

\begin{proof}[Proof of Proposition~\ref{prop:12-7}]
  Proposition~\ref{prop:11-1} (``Weak functoriality'') and Lemma~\ref{lem:12-8}
  imply that $q_Y◦γ$ is a morphism of $\cC$-pairs.  Since $q_Y◦γ$ is constant on
  the orbits of the $G$-action on $X$, it will factor via $q_X$.
\end{proof}

\subsection{Quotients as adapted covers}
\approvals{Erwan & yes \\ Stefan & yes}

The following two observations are frequently useful.  The elementary proofs are
left to the reader.

\begin{obs}[Adapted covers vs.~quotients]\label{obs:12-9}%
  Let $(X, D_X)$ be a $\cC$-pair and let $γ : \what{X} \twoheadrightarrow X$ be
  an adapted cover that is Galois with group $G$.  Write
  \[
    D_{\what{X}} := \bigl( γ^* ⌊D_X⌋ \bigr)_{\red} %
    \quad\text{and}\quad %
    (X, D'_X) := \factor{(\what{X}, D_{\what{X}})}{G}.
  \]
  Then, $D'_X ≥ D_X$.  Proposition~\vref{prop:10-4} allows formulating this
   inequality by saying that the identity induces a morphism of $\cC$-pairs,
  \[
    \Id_X : (X, D'_X) → (X, D_X).  \eqno \qed
  \]
\end{obs}

\begin{obs}[Quotients vs.~adapted covers]\label{obs:12-10}%
  Let $G$ be a finite group that acts on a log pair $(X,D_X)$ and let $γ : X
  \twoheadrightarrow X/G$ be the quotient morphism.  Then, $γ$ is a strongly
  adapted cover for the quotient $\cC$-pair $(Q, D_Q) := (X, D_X)/G$.  The
  $\cC$-cotangent sheaf equals $Ω^{[1]}_{(Q, D_Q, γ)} = Ω^{[1]}_X(\log D_X)$.
  \qed
\end{obs}

\subsection{Adapted covers, uniformizations and quotients.  Proof of Lemma~\ref*{lem:12-8}}
\label{sec:12-3}
\approvals{Erwan & yes \\ Stefan & yes}

The following lemma is key to the proofs of Theorem~\ref{thm:12-5},
Proposition~\ref{prop:12-7}, and Lemma~\ref{lem:12-8}.  It might also be of
independent interest.

\begin{lem}[Adapted covers and quotients]\label{lem:12-11}%
  In the setting of Construction~\ref{cons:12-4}, let $γ : \what{X} → X$ be a
  $q$-morphism.  If $γ$ is adapted for the pair $(X, D_X)$, then $q◦γ : \what{X}
  → Q$ is adapted for the pair $(Q, D_Q)$ and
  \[
    Ω^{[p]}_{(X,D_X,γ)} = Ω^{[p]}_{(Q, D_Q, q◦γ)} \quad\text{for every number }p.
  \]
\end{lem}
\begin{proof}
  The assertion that $q◦γ$ is adapted follows directly from the choices made in
  Construction~\ref{cons:12-4}.

  Two reflexive sheaves $\what{X}$ agree if they agree on a big open set.
  Removing codimension-two subsets from all relevant varieties, we can therefore
  assume without loss of generality that all spaces are smooth and that all
  divisors have smooth support.  We are therefore in the setting of
  Definition~\ref{def:3-2} (``Bundle of adapted tensors in the nc case'') where
  the relevant $\cC$-cotangent sheaves are given as
  \begin{align}
    \label{eq:12-11-1} Ω¹_{(X,D_X,γ)} & = \ker \Bigl( γ^* Ω¹_X(\log D_X) \xrightarrow{γ^*\text{(residue morphism)}} 𝒪_{γ^* D_{X,\orb}} \Bigr).  \\
    \label{eq:12-11-2} Ω¹_{(Q,D_Q,q◦ γ)} & = \ker \Bigl( γ^*q^* Ω¹_Q(\log D_Q) \xrightarrow{γ^*q^*\text{(residue morphism)}} 𝒪_{γ^*q^* D_{Q,\orb}} \Bigr).
  \end{align}
  It suffices to consider the case $p=1$ and to show equality of these sheaves
  only; equality for all other values of $p$ will follow by local freeness.

  The construction of the divisor $D_Q$ has two immediate consequences.  To
  formulate them properly, let $R$ be the reduced divisor on $X$ obtained as the
  union of those components of the ramification divisor that are not contained
  in the finite part of $D_X$,
  \[
    R := (\supp \Ramification q) ∖ \supp D_{X,\orb} \quad ∈ \Div(X).
  \]
  With this notation in place, it follows from construction that
  \[
    \Branch q ⊆ \supp D_Q \quad\text{and}\quad q^* D_{Q,\orb} = D_{X,\orb} +
    R.
  \]
  The inclusion implies in particular that $q^* Ω¹_Q(\log D_Q) = Ω¹_X(\log
  D_X+R)$, and \eqref{eq:12-11-2} simplifies to
  \begin{align*}
    Ω¹_{(Q,D_Q,q◦γ)} & = \ker \Bigl( γ^* Ω¹_X(\log D_X + R) \xrightarrow{γ^*\text{(residue morphism)}} 𝒪_{γ^* (D_{X,\orb}+R)} \Bigr) \\
                     & = \ker \Bigl( γ^* Ω¹_X(\log D_X + R) \xrightarrow{γ^*\text{(residue morphism)}} 𝒪_{γ^* D_{X,\orb}} ⊕ 𝒪_{γ^* R} \Bigr);
  \end{align*}
  for the last equality, we use that $\supp D_{X,\orb}$ and $\supp R$ are
  disjoint by our smoothness assumption.  Recalling the standard fact that
  $Ω¹_X(\log D_X)$ is described as the kernel of the following residue morphism,
  \[
    Ω¹_X(\log D_X) = \ker \Bigl( Ω¹_X(\log D_X + R) \xrightarrow{\text{(residue morphism)}} 𝒪_{γ^* R} \Bigr)
  \]
  the claim now follows.
\end{proof}

\begin{proof}[Proof of Lemma~\ref*{lem:12-8}]
  We consider the uniformizable case only.  Let $(X, D_X)$ be a $\cC$-pair, let
  $G$ be a finite group that acts on $(X,D_X)$ and let $(Q,D_Q)$ be the quotient
  of Construction~\ref{cons:12-4}.  Finally, consider a cover $γ: \what{X} → X$
  where $\bigl(\what{X}, (γ^*⌊D_X⌋)_{\reg}\bigr)$ is nc.  The following
  statements are then equivalent.
  \begin{align*}
    γ \text{ uniformizes } (X,D_X) & ⇔ Ω^{[•]}_{(X,D_X,γ)} = Ω^{•}_{\what{X}}(\log γ^*⌊D_X⌋) && \text{Observation~\ref{obs:4-9}}\\
    & ⇔ Ω^{[•]}_{(Q,D_Q,q◦γ)} = Ω^{•}_{\what{X}}(\log γ^*⌊D_X⌋) && \text{Lemma~\ref{lem:12-11}} \\
    & ⇔ Ω^{[•]}_{(Q,D_Q,q◦γ)} = Ω^{•}_{\what{X}} \bigl(\log (q◦γ)^* ⌊D_Q⌋\bigr) && \text{Construction~\ref{cons:12-4}} \\
    & ⇔ q◦γ \text{ uniformizes } (Q, D_Q).  && \text{Observation~\ref{obs:4-9}}
  \end{align*}
  The claim thus follows.
\end{proof}

\subsection{Existence of quotients, proof of Theorem~\ref*{thm:12-5}}
\label{sec:12-4}
\approvals{Erwan & yes \\ Stefan & yes}

Maintain setting and assumptions of Theorem~\ref{thm:12-5} and
Construction~\ref{cons:12-4}.  The theorem asserts that $q$ is a morphism of
$\cC$-pairs and that the universal property holds.  Even though the proofs of
these two statements are similar, we prefer to present the arguments separately,
in two separated steps.

\subsection*{Step 1: The quotient morphism is a morphism of $\cC$-pairs}
\approvals{Erwan & yes \\ Stefan & yes}

We work with Corollary~\ref{cor:9-6} (``Elementary criterion for
$\cC$-morphisms'') and assume that a commutative diagram of the following form
is given,
\[
  \begin{tikzcd}[column sep=3cm]
    \what{X} \ar[r, "\what{q}"] \ar[d, "γ\text{, adapted}"'] & \what{Q} \ar[d, "δ\text{, adapted}"] \\
    X \ar[r, two heads, "q\text{, quotient morphism}"'] & Q.
  \end{tikzcd}
\]
In order to show that the quotient morphism $q$ is a morphism of $\cC$-pairs, we
need to show that $\what{q}$ admits pull-back of adapted reflexive
differentials.  That, however, follows now almost directly,
\begin{align*}
  Ω^{[•]}_{(X,D_X,γ)} & = Ω^{[•]}_{(Q, D_Q, q◦γ)} && \text{Lemma~\ref{lem:12-11}} \\
  & = Ω^{[•]}_{(Q, D_Q, δ◦\what{q})} && \text{commutativity} \\
  & = \what{q}^{\:[*]}Ω^{[•]}_{(Q, D_Q, δ)} && \text{Observation~\ref{obs:4-14}}.
\end{align*}

\subsection*{Step 2: Universal property}
\approvals{Erwan & yes \\ Stefan & yes}

Let $φ : (X,D_X) → (Y,D_Y)$ be any morphism of $\cC$-pairs whose underlying
morphism of varieties is constant on $G$-orbits.  Then, the universal property
of the classic categorical quotients asserts that $φ$ factorizes via $q$ as a
morphism of analytic varieties,
\[
  \begin{tikzcd}[column sep=large]
    X \ar[r, two heads, "q"'] \ar[rr, bend left=15, "φ"] & Q \ar[r, "∃!  ψ"'] & Y,
  \end{tikzcd}
\]
and we need to show that $ψ$ induces a morphism between $\cC$-pairs $(Q,D_Q)$
and $(Y,D_Y)$.  As before, we work with Corollary~\ref{cor:9-6} and assume that
a diagram of the following form is given,
\[
  \begin{tikzcd}[column sep=3cm]
    & \what{Q} \ar[r, "\what{ψ}"] \ar[d, "δ\text{, adapted}"'] & \what{Y} \ar[d, "ε\text{, adapted}"] \\
    X \ar[r, two heads, "q\text{, quotient morphism}"'] & Q \ar[r, "ψ"'] & Y.
  \end{tikzcd}
\]
We need to show that $\what{ψ}$ admits pull-back of adapted reflexive
differentials.  As before, we note that this question is local on $Q$.  We may
therefore shrink $Q$, employ Lemma~\ref{lem:2-36} (``Strongly adapted covers
exist locally'') and assume without loss of generality that there exists an
adapted cover $\overline{X} \twoheadrightarrow X$.  Let $\wcheck{X} =
\wcheck{Q}$ be a suitable component of the normalized fibre product $\bigl(
\overline{X} ⨯_Q \what{Q} \bigr)^{\operatorname{norm}}$.  We obtain an extended
diagram as follows,
\[
  \begin{tikzcd}[column sep=3cm, row sep=large]
    \wcheck{X} \ar[dd, "γ\text{, adapted}"'] \ar[r, "\wcheck{q} \: = \: \Id", equal] & \wcheck{Q} \ar[d, two heads, "δ'\text{, finite}"'] \ar[r, "\wcheck{ψ}"] & \wcheck{Y} \ar[d, equal] \\
    & \what{Q} \ar[r, "\what{ψ}"] \ar[d, "δ\text{, adapted}"'] & \what{Y} \ar[d, "ε\text{, adapted}"] \\
    X \ar[r, two heads, "q\text{, quotient morphism}"'] & Q \ar[r, "ψ"'] & Y.
  \end{tikzcd}
\]
Recall Lemma~\ref{lem:9-5} (``Test for pull-back of adapted reflexive
differentials'').  To show that $\what{ψ}$ admits pull-back of adapted reflexive
differentials, it suffices to show that $\wcheck{ψ}$ admits pull-back of adapted
reflexive differentials: there exist sheaf morphisms
\[
  η : \wcheck{ψ}^* \Bigl( Ω^{[•]}_{(Y, D_Y, c)} \Bigr) → Ω^{[•]}_{(Q, D_Q, δ◦δ')}
\]
that generically agree with the standard pull-back $d_{\cC} \wcheck{ψ}⁺$ of
adapted reflexive differentials.  But we know by assumption that
$\wcheck{ψ}◦\wcheck{q}$ admits pull-back of adapted reflexive differentials:
There exist sheaf morphisms
\begin{align*}
  \wcheck{ψ}^* \Bigl( Ω^{[•]}_{(Y, D_Y, c)} \Bigr) & = (\wcheck{ψ}◦ q)^* \Bigl( Ω^{[•]}_{(Y, D_Y, c)} \Bigr) && \wcheck{q} = \Id \\
  & → Ω^{[•]}_{(X, D_X, γ)} && \wcheck{ψ}◦\wcheck{q}\text{ admits pull-back} \\
  & = Ω^{[•]}_{(Q, D_Q, q◦γ)} && \text{Lemma~\ref{lem:12-11}} \\
  & = Ω^{[•]}_{(Q, D_Q, δ◦δ')} && \text{Commutativity}
\end{align*}
The claim thus follows.  \qed

% !TEX root = orbiAlb1
%
% Do not edit the following line.  The text is automatically updated by
% subversion.
%
\svnid{$Id: 13-pullBackAndMMP.tex 911 2024-09-29 18:09:56Z kebekus $}
\selectlanguage{british}

\section{Pull-back and the Minimal Model Program}
\subversionInfo
\label{sec:13}

\subsection{Pull-back of adapted reflexive tensors}
\approvals{Erwan & yes \\ Stefan & yes}

In Section~\ref{sec:8}, we defined morphisms of $\cC$-pairs are morphisms of
varieties where every diagram of form \eqref{eq:7-1-1} admits pull-back of
adapted reflexive differentials.  This section discusses criteria to guarantee
that diagrams also admit pull-back of adapted reflexive \emph{tensors}.  The main
result, formulated in Proposition~\ref{prop:13-1} below, relates $\cC$-morphisms
to notions of minimal model theory, and gives severe restrictions for the
existence of resolutions of singularities in the context of $\cC$-pairs; we
discuss these issues in Section~\ref{sec:13-2} right after formulating the main
result.

\begin{prop}[Criterion for pull-back of adapted reflexive tensors]\label{prop:13-1}%
  Given a morphism $φ : (X, D_X) → (Y, D_Y)$ of $\cC$-pairs, let $p ∈ ℕ⁺$ be any
  number.  Assume that there exists an open covering $Y = ∪ Y_i$ and adapted
  covers $γ_i : \what{Y}_i \twoheadrightarrow Y_i$ where the sheaves
  $Ω^{[p]}_{(Y, D_Y, γ_i)}$ of adapted reflexive differentials are locally free.
  Then, every diagram of form \eqref{eq:7-1-1} admits pull-back of adapted
  reflexive $(n,p)$-tensors, for every $n ∈ ℕ⁺$.
\end{prop}

We will prove Proposition~\ref{prop:13-1} in Section~\vref{sec:13-3}.

\subsection{Relation to the Minimal Model Program}
\approvals{Erwan & yes \\ Stefan & yes}
\label{sec:13-2}

We highlight one case where the assumptions of Proposition~\ref{prop:13-1} are
known to hold.

\begin{prop}[Pull-back for morphisms to $ℚ$-Gorenstein pairs]\label{prop:13-2}%
  Let $φ : (X, D_X) → (Y, D_Y)$ be a morphism of $\cC$-pairs where $(Y,D_Y)$ is
  locally $ℚ$-Gorenstein, where $\dim X = \dim Y$, and where $φ$ is generically
  finite.  If $m ∈ ℕ⁺$ is any number, then there exists a pull-back map
  \begin{equation}\label{eq:13-2-1}
    φ^* \Bigl( ω_Y^{⊗m} ⊗ 𝒪_Y\bigl(⌊m·D_Y⌋\bigr) \Bigr)^{\vee\vee}
    \xrightarrow{\text{pull-back}}
    \Bigl( ω_X^{⊗m} ⊗ 𝒪_X\bigl(⌊m·D_X⌋\bigr) \Bigr)^{\vee\vee}
  \end{equation}
  whose restriction to
  \[
    \bigl(Y_{\reg} ∖ \supp D_Y \bigr) ∩ φ^{-1}(Y_{\reg} ∖ \supp D_Y)
  \]
  agrees with the standard pull-back map of Kähler differentials and their
  symmetric powers.
\end{prop}

\begin{rem}\label{rem:13-3}%
  Proposition~\ref{prop:13-2} does not assume that $(X,D_X)$ is locally
  $ℚ$-Gorenstein.  If canonical divisors exist, then \eqref{eq:13-2-1} can be
  written in the compact form
  \[
    φ^* 𝒪_Y\bigl(m·K_Y+⌊m·D_Y⌋\bigr)
    \xrightarrow{\text{pull-back}} 𝒪_X\bigl(m·K_X+⌊m·D_X⌋\bigr),
  \]
  which might be more familiar to the algebraic geometer.
\end{rem}

\begin{proof}[Proof of Proposition~\ref{prop:13-2}]
  \CounterStep{}Cover $Y$ by Stein open sets, $Y = ∪_i Y_i$, so that canonical
  divisors $K_{Y_i}$ exist on each of the $Y_i$.  Recall from
  Lemma~\ref{lem:2-36} that the Stein spaces $Y_i$ admit strongly adapted covers
  $γ_i : \what{Y}_i \twoheadrightarrow Y_i$.  The pull-back divisors $γ^*_i
  (K_{Y_i}+D_Y)$ are then locally $ℚ$-Cartier Weil divisors on the $\what{Y}_i$.
  Shrinking the $Y_i$ if necessary, we may assume without loss of generality
  that $γ^*_i (K_{Y_i}+D_Y)$ are $ℚ$-Cartier and that suitable multiples are
  Cartier and linearly trivial,
  \begin{equation}\label{eq:13-4-1}
    𝒪_{\what{Y}_i}\bigl( m_i·γ^*_i(K_{Y_i}+D_Y) \bigr) ≅ 𝒪_{\what{Y}_i}, \quad \text{for suitable } m_i ∈ ℕ⁺.
  \end{equation}
  Following \cite[Cor.~1.9]{Reid79} or \cite[Sect.~3.5--3.7]{Reid87}, the
  isomorphisms~\eqref{eq:13-4-1} can be used to construct index-one-covers, that
  is, cyclic covers $β_i : \wcheck{Y}_i \twoheadrightarrow \what{Y}_i$ where
  $(γ_i◦ β_i)^* (K_{Y_i}+D_Y)$ is Cartier.  Set $α_i := γ_i◦β_i$, $n := \dim X$
  and observe that
  \[
    Ω^{[n]}_{(Y,D_Y,α_i)} ≅ 𝒪_{\wcheck{Y}_i}\bigl( α^*_i (K_{Y_i}+D_Y) \bigr)
  \]
  is locally free.  Apply Proposition~\ref{prop:13-1} to the trivial diagram
  \[
    \begin{tikzcd}[column sep=large]
      X \ar[r, "φ"] \ar[d, "\Id_X\text{, $q$-morphism}"'] & Y \ar[d, "\Id_Y\text{, $q$-morphism}"] \\
      X \ar[r, "φ"'] & Y
    \end{tikzcd}
  \]
  and recall from Example~\ref{ex:4-6} that domain and range of associated the
  pull-back map
  \[
    η : \what{φ}^* \Bigl( \Sym^{[m]}_{\cC} Ω^{[n]}_{(Y, D_Y, \Id_Y)} \Bigr) → \Sym^{[m]}_{\cC} Ω^{[d]}_{(X, D_X, \Id_X)}
  \]
  are identified as
  \begin{align*}
    \Sym^{[m]}_{\cC} Ω^{[n]}_{(Y, D_Y, \Id_Y)} & = \bigl(ω_Y^{⊗m}⊗𝒪_Y(⌊m·D_Y⌋)\bigr)^{\vee\vee} \\
    \Sym^{[m]}_{\cC} Ω^{[n]}_{(X, D_X, \Id_X)} & = \bigl(ω_X^{⊗m}⊗𝒪_X(⌊m·D_X⌋)\bigr)^{\vee\vee}
  \end{align*}
  The proof is thus finished.
\end{proof}

Proposition~\ref{prop:13-2} relates the notion of a $\cC$-morphism to the notion
of discrepancies used in birational geometry, cf.~\cite[Chapt.~I.1]{Reid87} or
\cite[Def.~2.22]{KM98}.  Instead of going into details, we note only one
immediate consequence, which refines Proposition~\vref{prop:10-6}.

\begin{cor}[$\cC$-morphisms and canonical singularities, compare Proposition~\ref{prop:10-6}]\label{cor:13-5}%
  Let $(X,D)$ be a $\cC$-pair that is locally $ℚ$-Gorenstein.  If a
  $\cC$-resolution of singularities exists, then $(X,D)$ is log canonical.  In
  particular, $X$ has Du~Bois singularities.
\end{cor}
\begin{proof}
  Since the question is local on $X$, we may assume without loss of generality
  that $(X,D)$ is $ℚ$-Gorenstein and that a canonical divisor exists.  Let $π :
  (\wtilde{X}, \wtilde{D}) → (X,D)$ be a $\cC$-resolution of singularities, with
  $π$-exceptional divisor $E ∈ \Div(\wtilde{X})$.  Let $m ∈ ℕ⁺$ be a number such
  that $m· D$ is integral and $m·(K_X+D)$ is Cartier.  Following
  Remark~\ref{rem:13-3}, we obtain a pull-back map
  \begin{align*}
    π^* 𝒪_X\bigl(m·K_X+m·D\bigr) & → 𝒪_{\wtilde{X}}\bigl(m·K_{\wtilde{X}}+ m·\wtilde{D}\bigr) && \text{pull-back} \\
    & ⊆ 𝒪_{\wtilde{X}}\bigl(m·K_{\wtilde{X}}+ m·π^{-1}_* D + m·E\bigr),
  \end{align*}
  where $π^{-1}_* D$ denotes the strict transform.  As in the proof of
  Proposition~\ref{prop:10-6}, this means that $\operatorname{discrep}(X,D) ≥
  -1$, so that $(X,D)$ is log canonical as claimed.
\end{proof}

\begin{rem}[Converse of Corollary~\ref{cor:13-5}?]
  For the time being, we are unsure if a converse of Corollary~\ref{cor:13-5} holds
  and refer the reader to Section~\ref{sec:15-4} for questions and a more
  detailed discussion.
\end{rem}

\subsection{Proof of Proposition~\ref*{prop:13-1}}
\approvals{Erwan & yes \\ Stefan & yes}
\label{sec:13-3}

\CounterStep{}We prove Proposition~\ref{prop:13-1} under the simplifying
assumption that there exists one single adapted cover $γ_0 : \what{Y}_0
\twoheadrightarrow Y$ where $Ω^{[p]}_{(Y,D_Y,γ_0)}$ is locally free.  The proof
in the general case is conceptually identical, but notationally more involved.
We assume that a number $n ∈ ℕ⁺$ and a diagram of the form \eqref{eq:7-1-1} are given,
\begin{equation}\label{eq:13-7-1}
  \begin{tikzcd}[column sep=large]
    \what{X} \ar[r, "\what{φ}"] \ar[d, "a\text{, $q$-morphism}"'] & \what{Y} \ar[d, "b\text{, $q$-morphism}"] \\
    X \ar[r, "φ"'] & Y.
  \end{tikzcd}
\end{equation}
Set
\begin{align*}
  \wcheck{Y} & := \text{Galois closure of a connected component of the normalized fibre product } \what{Y}_0 ⨯_Y \what{Y} \\
  \wcheck{X} & := \text{connected component of the normalized fibre product } \wcheck{Y} ⨯_{\what{Y}} \what{X}
\end{align*}
We obtain an extension of Diagram~\eqref{eq:13-7-1}, as follows,
\[
  \begin{tikzcd}[column sep=large, row sep=1cm]
    \wcheck{X} \ar[r, "\wcheck{φ}"] \ar[d, two heads, "\what{a}\text{, Galois cover}"'] & \wcheck{Y} \ar[d, two heads, "\what{b}\text{, Galois cover}"] \ar[rr, equals] && \wcheck{Y} \ar[d, "\what{γ}_0\text{, $q$-morphism}"] \\
    \what{X} \ar[r, "\what{φ}"] \ar[d, "a\text{, $q$-morphism}"'] & \what{Y} \ar[d, "b\text{, $q$-morphism}"] && \what{Y}_0 \ar[d, two heads, "γ_0\text{, adapted cover}"] \\
    X \ar[r, "φ"'] & Y \ar[rr, equals] && Y.
  \end{tikzcd}
\]
The assumption that $φ$ is a morphism of $\cC$-pairs equips us with a pull-back
morphism of adapted reflexive differentials,
\[
  η : \what{φ}^{\,*} Ω^{[p]}_{(Y,D_Y,b)} → Ω^{[p]}_{(X,D_X,a)}.
\]
Using the assumption that $Ω^{[p]}_{(Y,D_Y,γ_0)}$ is locally free, we find that
\[
  Ω^{[p]}_{(Y, D_Y, b◦ \what{b})} = Ω^{[p]}_{(Y,D_Y, γ_0◦\what{γ}_0)} = \what{γ}^{\,[*]}_0 Ω^{[p]}_{(Y,D_Y, γ_0)} = \what{γ}^{\,*}_0 Ω^{[p]}_{(Y,D_Y, γ_0)}
\]
is likewise locally free.  The morphism $η$ therefore induces a pull-back
morphism for the sheaves of reflexive adapted $(n,p)$-tensors on $\wcheck{Y}$,
for every $n ∈ ℕ$:
\begin{align}
  \wcheck{φ}^* \Sym^{[n]}_{\cC} Ω^{[p]}_{(Y,D_Y,b◦\what{b})} & = \wcheck{φ}^* \Sym^{[n]} Ω^{[p]}_{(Y,D_Y,b◦\what{b})} && b◦\what{b}\text{ adapted, Obs.~\ref{obs:4-12}} \nonumber \\
  & = \wcheck{φ}^* \Sym^n Ω^{[p]}_{(Y,D_Y,b◦\what{b})} && \text{locally free} \nonumber \\
  & = \Sym^n \wcheck{φ}^* Ω^{[p]}_{(Y,D_Y,b◦\what{b})} && \text{natural} \label{eq:13-7-2} \\
  & → \Sym^n Ω^{[p]}_{(X,D_X,a◦\what{a})} && \Sym^n η \nonumber \\
  & → \Sym^{[n]} Ω^{[p]}_{(X,D_X,a◦\what{a})} && \text{natural} \nonumber \\
  & → \Sym^{[n]}_{\cC} Ω^{[p]}_{(X,D_X,a◦\what{a})} && \text{Observation~\ref{obs:4-8}}.  \nonumber
  \intertext{This in turn yields a morphism between push-forward sheaves,}
  \what{φ}^{\,*} \what{b}_* \Sym^{[n]}_{\cC} Ω^{[p]}_{(Y,D_Y,b◦\what{b})} & → \what{a}_* \wcheck{φ}^* \Sym^{[n]}_{\cC} Ω^{[p]}_{(Y,D_Y,b◦\what{b})} && \text{natural} \nonumber \\
  & → \what{a}_* \Sym^{[N]}_{\cC} Ω^{[p]}_{(X,D_X,a◦\what{a})}.  && \what{a}_*\text{\eqref{eq:13-7-2}} \label{eq:13-7-3}
  \intertext{Recall from Observation~\ref{obs:4-19} that all morphisms in \eqref{eq:13-7-2}
  are morphisms of Galois-linearized sheaves.  The resulting map
  \eqref{eq:13-7-3} will therefore map Galois-invariant sections to
  Galois-invariant sections.  We obtain the desired pull-back map as
  follows,}
  \what{φ}^{\,*} \Sym^{[n]}_{\cC} Ω^{[p]}_{(Y,D_Y,b)} & = \what{φ}^{\,*} \left(\what{b}_* \Sym^{[n]}_{\cC} Ω^{[p]}_{(Y,D_Y,b◦\what{b})}\right)^{\Gal(\what{b})} && \text{Lemma~\ref{lem:4-20}} \nonumber \\
  & → \left( \what{a}_* \Sym^{[n]} Ω^{[p]}_{(X,D_X,a◦\what{a})} \right)^{\Gal(\what{a})} && \text{\eqref{eq:13-7-3}}^{\Gal} \label{eq:13-7-4} \\
  & = \Sym^{[n]}_{\cC} Ω^{[p]}_{(X,D_X,a)} && \text{Lemma~\ref{lem:4-20}.} \nonumber
\end{align}
We leave it to the reader to check that the restriction of \eqref{eq:13-7-4} to
$\what{X}⁺$ agrees with the pull-back morphism $d_{\cC} \what{φ}⁺$, so that
Diagram~\ref{eq:13-7-1} really does admit pull-back of adapted reflexive
$(n,p)$-tensors, in the sense of Definition~\ref{def:7-6}.  \qed

% !TEX root = orbiAlb1
%
% Do not edit the following line.  The text is automatically updated by
% subversion.
%
\svnid{$Id: 14-literature.tex 911 2024-09-29 18:09:56Z kebekus $}
\selectlanguage{british}

\section{Relation to the literature}
\subversionInfo
\approvals{Erwan & yes \\ Stefan & yes}
\label{sec:14}

Morphisms between $\cC$-pairs have already been discussed in the literature, but
typically only in special settings and for $\cC$-pairs that satisfy additional
assumptions.  While all these notions overlap with Definition~\ref{def:8-1},
there is often no implication, unless we are in the simplest setting of
$\cC$-pairs with snc boundary.  For completeness and future reference, this
section compares Definition~\ref{def:8-1} with three of the more prominent
definitions used in the literature.

\subsection{Orbifold morphisms in the sense of Campana}
\approvals{Erwan & yes \\ Stefan & yes}

Campana uses the following definition in several papers.

\begin{defn}[\protect{Orbifold morphism in the sense of Campana, \cite[Déf.~2.3]{MR2831280}}]\label{def:14-1}%
  \index{orbifold morphism!in the sense of Campana}Let $(X, D_X)$ and $(Y, D_Y)$
  be two $\cC$-pairs, where $Y$ is $ℚ$-factorial.  A morphism $φ : X → Y$ is
  called \emph{orbifold morphism} if $\img φ ⊄ \supp D_Y$ and if every pair of
  prime Weil divisors, $Δ_X ∈ \Div(X)$ and $Δ_Y ∈ \Div(Y)$ with $\img φ ⊄ \supp
  Δ_Y$ and $Δ_X ⊂ φ^{-1}(Δ_Y)$ satisfies the inequality
  \begin{equation}\label{eq:14-1-1}
    \bigl( \mult_{Δ_X} φ^* Δ_Y\bigr)·\bigl(\mult_{\cC, Δ_X} D_X\bigr) ≥
    \mult_{\cC, Δ_Y} D_Y.
  \end{equation}
\end{defn}

\begin{rem}
  The $ℚ$-factoriality assumption in Definition~\ref{def:14-1} guarantees that a
  meaningful pull-back $φ^* Δ_Y ∈ ℚ\Div(X)$ exists.
\end{rem}

Definition~\ref{def:14-1} has the drawback that it is not local in the analytic
topology.  The following examples illustrate some problems.

\begin{figure}
  \begin{tikzpicture}
    \node (tb) at (0,0) {\includegraphics[width=5cm]{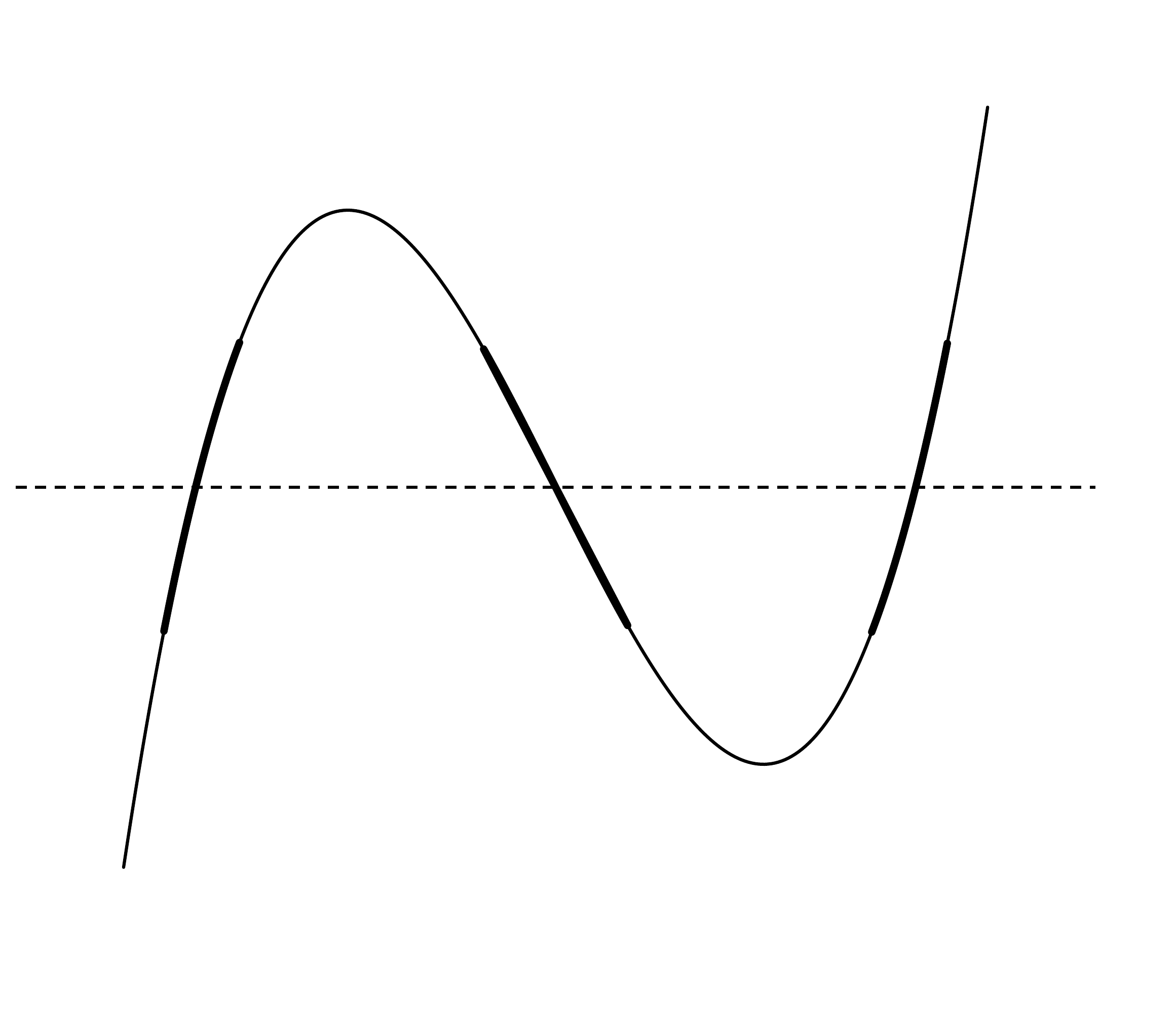}};
    \node (t) at (6,0) {\includegraphics[width=5cm]{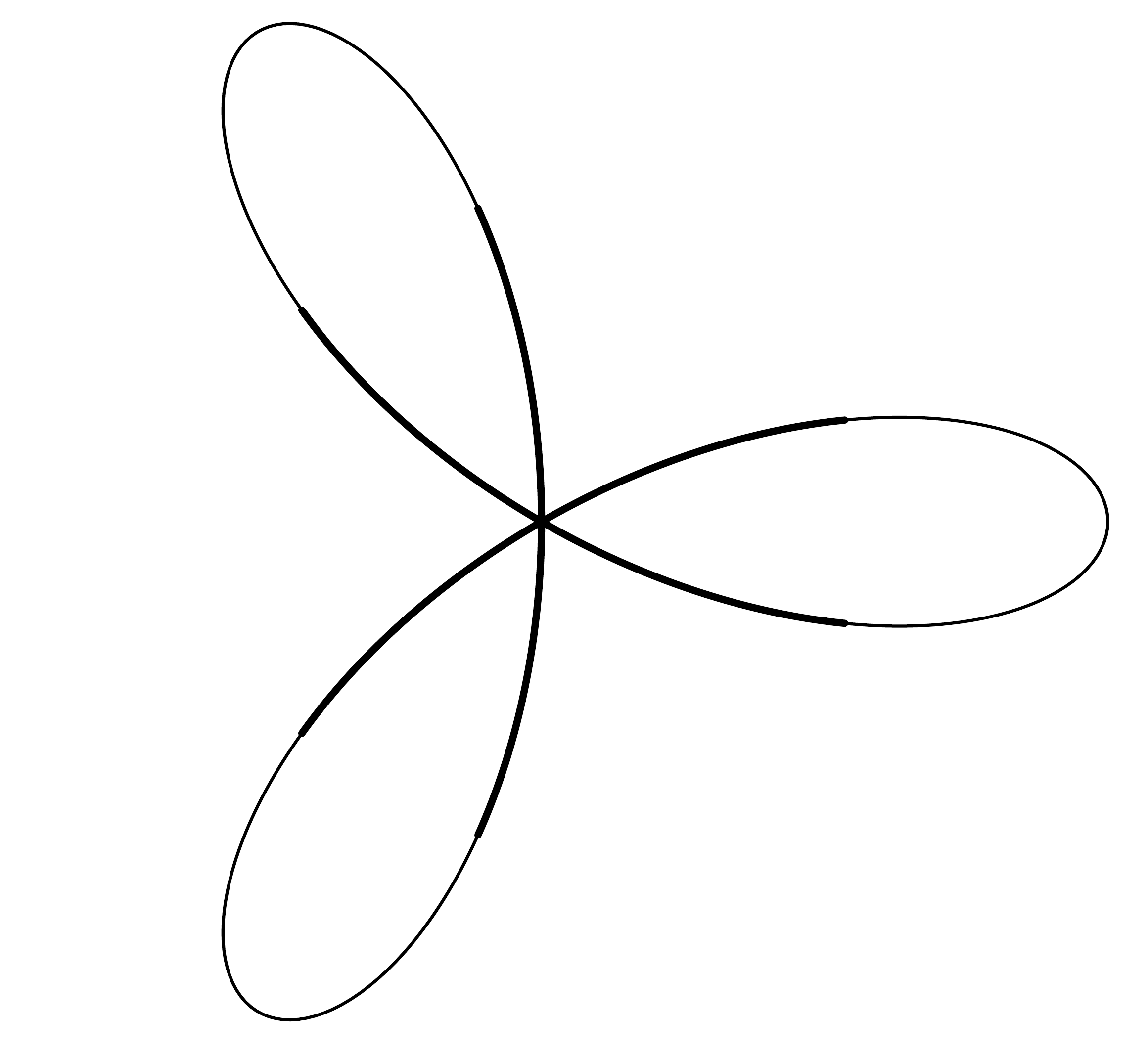}};
    \node at (-2.6, 0.1) {$E$};
    \node at (-1.5,-1.5) {$π^{-1}_* P$};
    \node at (-1.4,-0.5) {\tiny $π^{-1}_* P_1$};
    \node at (-0.2,-0.5) {\tiny $π^{-1}_* P_2$};
    \node at ( 1.9, 0.6) {\tiny $π^{-1}_* P_3$};

    \node at ( 8.5, 0.0) {$P$};
    \node at ( 7.0,-0.55) {\tiny $P_3$};
    \node at ( 7.0, 0.65) {\tiny $P_2$};
    \node at ( 5.9, 1.10) {\tiny $P_1$};
    \draw[->] (tb) -- (4.5,0) node[pos=.5,above] {$π$};
  \end{tikzpicture}

  \caption{Blow-up of the \emph{\foreignlanguage{french}{Paquerette de Mélibée}}}
  \label{fig:14-1}
\end{figure}
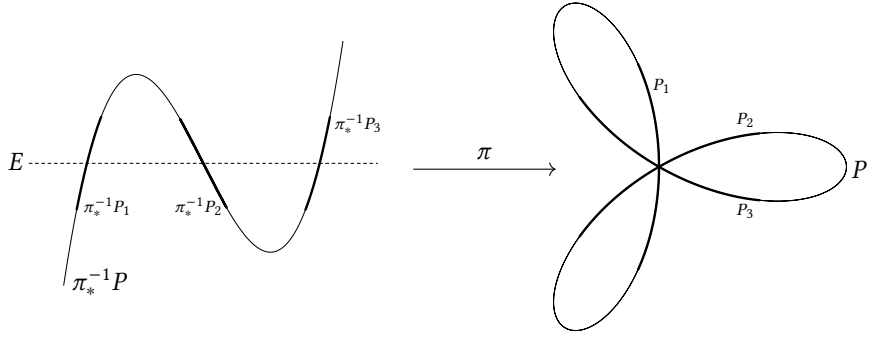

\begin{example}[Restriction of orbifold morphism to open set is no orbifold morphism]
  Let $π : X → Y$ be the blow-up of the affine plane $Y = ℂ²$ at the origin,
  with exceptional divisor $E ⊆ X$.  Consider the $\cC$-divisors
  \[
    D_Y := \frac{2}{3}·P ∈ \Div(Y)
    \quad\text{and}\quad
    D_X := \frac{2}{3}·π^{-1}_*P,
  \]
  where
  \[
    P := \bigl\{ (x,y) \::\: (x²+y²)² = x³-3·xy² \bigr\} ∈ \Div(Y)
  \]
  is the \emph{\foreignlanguage{french}{Paquerette de Mélibée}} shown in
  Figure~\ref{fig:14-1}.  An elementary computation shows that $π$ is an
  orbifold morphism.

  The situation changes if we consider a sufficiently small neighbourhood $U$ of
  the origin in $Y$.  There, $P$ decomposes into three components, $P = P_1 +
  P_2 + P_3$, and an elementary computation shows that \eqref{eq:14-1-1} is
  violated when we choose $Δ_X = E$ and $Δ_Y = P_1$.  It follows that the
  restricted morphism is not an orbifold morphism in the sense of
  Definition~\ref{def:14-1}.
\end{example}

\begin{example}[Orbifold morphisms cannot be glued]\label{ex:14-4}%
  Let $Y$ be a normal analytic variety that is locally $ℚ$-factorial but not
  $ℚ$-factorial.  Then, $\Id_Y : Y → Y$ is not an orbifold endomorphism of the
  $\cC$-pair $(Y,0)$.  Still, $Y$ can be covered by open sets $U ⊆ Y$ such that
  $\Id_Y|_U$ is an orbifold morphism.
\end{example}

\begin{rem}[Existence of divisors is not local]
  We fear that complications similar to those of Example~\ref{ex:14-4} arise in
  settings where Inequality~\eqref{eq:14-1-1} is void because there are no global
  divisors in $Y$ that can be used for $Δ_Y$.  This could happen if $Y$ is
  compact and of algebraic dimension zero.
\end{rem}

The authors of this paper feel that non-locality restricts the usefulness of
Definition~\ref{def:14-1} in practise.  The following alternative notion avoids
these problems.

\begin{defn}[Local orbifold morphism]\label{def:14-6}%
  Let $(X, D_X)$ and $(Y, D_Y)$ be two $\cC$-pairs, where $Y$ is locally
  $ℚ$-factorial.  A morphism $φ : X → Y$ is called \emph{local orbifold
  morphism}\index{local orbifold morphism!in the sense of Campana} if $\img φ ⊄
  \supp D_Y$ and if for every pair of sufficiently small open sets $Y⁺ ⊆ Y$ and
  $X⁺ ⊆ φ^{-1}(Y⁺)$, the restricted morphism $φ|_{X⁺} : X⁺ → Y⁺$ is an orbifold
  morphism in the sense of Definition~\ref{def:14-1}.
\end{defn}

Like the notion of a $\cC$-morphism, local orbifold morphisms are local in
nature and stable under removing small subsets from the source variety.  The
following analogue of Observation~\vref{obs:8-5} is not hard to
show\footnote{But watch out: Irreducibility is not a local property in the
analytic topology.}; its proof is left to the reader.

\begin{obs}[Local nature and removing small subsets, compare with Observation~\ref{obs:8-5}]\label{obs:14-7}%
  Given $\cC$-pairs $(X, D_X)$ and $(Y, D_Y)$ and a morphism $φ : X → Y$ such
  that $\img φ ⊄ \supp D_Y$, the following conditions are equivalent.
  \begin{itemize}
    \item The morphism $φ$ is a local orbifold morphism.

    \item There exist open coverings $(U_i)_{i ∈ I}$ and $(V_j)_{j ∈ J}$ of $X$
    and $Y$, respectively, such that every restricted morphism
    \[
      φ|_{U_i ∩ φ^{-1}(V_j)} : U_i ∩ φ^{-1}(V_j) → V_j
    \]
    is a local orbifold morphism between the restricted $\cC$-pairs
    \[
      \bigl(U_i ∩ φ^{-1}(V_j),\, D_X ∩ U_i ∩ φ^{-1}(V_j) \bigr)
      \quad\text{and}\quad
      (V_j, D_Y ∩ V_j).
    \]

    \item There exists a big open subset $U ⊂ X$ such that the restriction
    $φ|_U$ is a local orbifold morphism $(U, D_X ∩ U) → (Y, D_Y)$.  \qed
  \end{itemize}
\end{obs}

The notions ``$\cC$-morphism'' and ``local orbifold morphism'' are related but
not identical.  The following propositions compare the notions.  Together, they
imply that $\cC$-morphisms and local orbifold morphisms agree whenever the
target pair is nc.

\begin{prop}[$\cC$-morphism to locally $ℚ$-factorial target]\label{prop:14-8}%
  Let $(X, D_X)$ and $(Y, D_Y)$ be two $\cC$-pairs where $Y$ is locally
  $ℚ$-factorial.  Then, every $\cC$-morphism $φ : X → Y$ with $\img φ \not ⊂
  \supp D_Y$ is a local orbifold morphism.
\end{prop}

\begin{prop}[Local orbifold morphism to nc target]\label{prop:14-9}%
  Let $(X, D_X)$ and $(Y, D_Y)$ be two $\cC$-pairs.  If the pair $(Y, D_Y)$ is
  nc, then every local orbifold morphism $φ : X → Y$ is a $\cC$-morphism.
\end{prop}

Before proving Propositions~\ref{prop:14-8} and \ref{prop:14-9} below, we show
by way of an elementary example that local orbifold morphisms need not be
$\cC$-morphisms in general, even if the target is uniformizable.

\begin{example}[Local orbifold morphisms need not be $\cC$-morphisms]\label{ex:14-10}%
  Consider the space $Y = ℂ²$ and let $D_Y = \frac{2}{3}·D_1 + \frac{2}{3}·D_2 +
  \frac{1}{2}·D_3$ be the union of three lines passing through a common point $y
  ∈ Y$.  Recall from Example~\vref{ex:2-29} that $(Y,D_Y)$ is uniformizable.
  
  Let $φ : X → Y$ be the blow-up of $y ∈ Y$ and set $D_X := φ^{-1}_* D_Y +
  \frac{2}{3}·\operatorname{Exc}π$.  An elementary computation that we leave to
  the reader shows that $φ$ is an orbifold morphism for the pairs $(X,D_X)$ and
  $(Y,D_Y)$.  On the other hand, observe that
  \[
    K_X = φ^*K_Y+\operatorname{Exc}π \quad\text{and}\quad K_X+D_X = φ^*(K_Y+D_Y) - \frac{1}{6}·\operatorname{Exc}π.
  \]
  The shift in sign implies that the canonical pull-back map $(dφ)^{⊗6} : φ^*
  ω_Y^{⊗6} → ω_X^{⊗6}$ does \emph{not} extend to a pull-back map
  \[
    φ^* ω_Y^{⊗6}\bigl(6·D_Y\bigr) \xrightarrow{\text{pull-back}} ω_X^{⊗6}\bigl(6·D_X\bigr)
  \]
  Proposition~\ref{prop:13-2} therefore implies that $φ$ is \emph{not} a
  $\cC$-morphism.
\end{example}

\begin{proof}[Proof of Proposition~\ref*{prop:14-8}]
  \CounterStep{}We prove Proposition~\ref{prop:14-8} only under the simplifying
  assumption that $⌊D_X⌋ = 0$ and $⌊D_Y⌋ = 0$.  The proof of the general case is
  conceptually identical but requires additional case-by-case handling.

  \subsubsection*{Step 0: Setup and simplification}

  Let $φ : X → Y$ be a $\cC$-morphism between two $\cC$-pairs $(X, D_X)$ and
  $(Y, D_Y)$, where $Y$ is locally $ℚ$-factorial.  Given a pair of open sets $Y⁺
  ⊆ Y$ and $X⁺ ⊆ φ^{-1}(Y⁺)$, where $Y⁺$ is $ℚ$-factorial, and prime divisors
  $Δ_Y ⊂ Y⁺$ and $Δ_X ⊆ φ^{-1}(Δ_Y) ∩ X⁺$, we need to verify
  Inequality~\eqref{eq:14-1-1} from above,
  \begin{equation}\label{eq:14-11-1}
    \bigl( \mult_{Δ_X} φ^* Δ_Y\bigr)·\bigl(\mult_{\cC, Δ_X} D_X\bigr) ≥
    \mult_{\cC, Δ_Y} D_Y.
  \end{equation}
  To this end, choose a general point $x ∈ \supp Δ_X$ and set $y := φ(x)$.  For
  brevity of notation, set
  \[
    m_X := \mult_{\cC,Δ_X} D_X
    \quad\text{and}\quad
    m_Y := \mult_{\cC,Δ_Y} D_Y.
  \]
  Replacing $X$ and $Y$ by suitably small neighbourhoods of $x$ and $y$, we may
  assume without loss of generality that the following holds in addition.
  \begin{enumerate}
    \item\label{il:14-11-2} Using the assumption that $Y$ is locally
    $ℚ$-factorial, we assume that there exists a number $r ∈ ℕ⁺$ and a
    holomorphic function $f ∈ H⁰\bigl( Y⁺,\, 𝒪_Y \bigr)$ such that
    $\operatorname{div} f = r·Δ_Y$.

    \item\label{il:14-11-3} Using the choice that $x$ is general in $\supp Δ_X$,
    we assume that the divisor $φ^{-1}(Δ_Y)$ has only one component, so that
    $Δ_X = \supp φ^{-1}(Δ_Y)$.
  \end{enumerate}

  \subsubsection*{Step 1: Construct a cover of $Y$}

  To begin the proof in earnest, choose a component
  \[
    \what{Y} ⊆ \text{normalization of } \bigl\{ (y,z) ∈ Y⨯ℂ \::\: z^{r·m_Y} = f(y) \bigr\} \\
  \]
  and consider the associated cover $γ_Y : \what{Y} \twoheadrightarrow Y$.  The
  following properties hold by construction.
  \begin{enumerate}
    \item\label{il:14-11-4} The morphism $γ_Y$ is étale outside $Δ_Y$ and
    $\Branch (γ_Y) = Δ_Y$.

    \item All components of the divisor $γ^*_Y (Δ_Y)$ have multiplicity $m_Y$.

    \item\label{il:14-11-6} The ramification divisor of $γ_Y$ equals
    $\Ramification (γ_Y) = \frac{1}{m_Y}·γ^*_Y (Δ_Y)$.

    \item\label{il:14-11-7} There exists a function $\what{f} ∈ H⁰\bigl(
    \what{Y},\, 𝒪_{\what{Y}} \bigr)$ such that $\what{f}^{r·m_Y} = f◦γ_Y$.
  \end{enumerate}
  Using Items~\ref{il:14-11-6} and \ref{il:14-11-7}, recall from
  Example~\vref{ex:3-6} that the Kähler differential of $d\what{f} ∈ H⁰\bigl(
  \what{Y},\, Ω¹_{\what{Y}} \bigr)$ induces an adapted reflexive differential on
  $\what{Y}$, say
  \[
    σ_Y ∈ H⁰\Bigl( \what{Y},\, Ω^{[1]}_{(Y,D_Y,γ_Y)} \Bigr)
    ⊆ H⁰\Bigl( \what{Y},\, Ω^{[1]}_{\what{Y}} \Bigr).
  \]

  \subsubsection*{Step 2: Construct a cover of $X$}

  Secondly, choose a component $\what{X}$ in the normalization of $X ⨯_Y
  \what{Y}$.  We obtain a commutative diagram
  \begin{equation}\label{eq:14-11-8}
    \begin{tikzcd}[column sep=2cm, row sep=1cm]
      \what{X} \ar[r, "\what{φ}"] \ar[d, two heads, "γ_X"'] & \what{Y} \ar[d, two heads, "γ_Y"] \\
      X \ar[r, "φ"'] & Y.
    \end{tikzcd}
  \end{equation}
  Item~\ref{il:14-11-4} implies that the morphism $γ_X$ is étale outside
  $φ^{-1}(Δ_Y) \overset{\text{\ref{il:14-11-3}}}{=} Δ_X$.  Given that $Δ_Y$ is
  the zero-set of the function $f$, this can be formulated as follows.
  \begin{enumerate}
    \item\label{il:14-11-9} The ramification divisor of $γ_X$ is contained in
    the zero-locus of the function $\what{f}◦\what{φ}$.
  \end{enumerate}

  \subsubsection*{Step 3: An adapted differential on $\what{X}$}

  Since $\what{φ}$ admits pull-back of adapted reflexive differentials, recall
  from Observation~\ref{obs:7-10} (``Compatibility with pull-back of Kähler
  differentials'') that the Kähler differential $d\bigl(\what{f}◦\what{φ}\bigr)
  ∈ H⁰\bigl( \what{X},\, Ω¹_{\what{X}} \bigr)$ induces an adapted reflexive
  differential on $\what{X}$, say
  \[
    σ_X ∈ H⁰\Bigl( \what{X},\, Ω^{[1]}_{(X,D_X,γ_X)} \Bigr)
    ⊆ H⁰\Bigl( \what{X},\, Ω^{[1]}_{\what{X}} \Bigr).
  \]
  This poses conditions on the vanishing orders of the function
  $\what{f}◦\what{φ}$ along prime divisors $Δ_{\what{X}} ⊆ γ^*_X Δ_X$.  To make
  this statement precise, choose one $Δ_{\what{X}}$, and recall that
  Item~\ref{il:14-11-9} together with Example~\vref{ex:3-6} implies that
  \begin{equation}\label{eq:14-11-10}
    \mult_{Δ_{\what{X}}} \operatorname{div} \bigl(\what{f} ◦ \what{φ} \bigr)
    ≥ \frac{\mult_{Δ_{\what{X}}} γ^*_X Δ_X}{\mult_{\cC, Δ_X} D_X}.
  \end{equation}
  But the left side of \eqref{eq:14-11-10} can be computed
  \begin{align*}
    \mult_{Δ_{\what{X}}} \operatorname{div} \bigl(\what{f} ◦ \what{φ} \bigr) & = \frac{\mult_{Δ_{\what{X}}} \operatorname{div} \bigl(f ◦ γ_Y ◦ \what{φ} \bigr)}{m_X·r} && \text{by \ref{il:14-11-7}} \\
    & = \frac{\mult_{Δ_{\what{X}}} \operatorname{div} \bigl(f ◦ φ ◦ γ_X \bigr)}{m_Y·r} && \text{commutativity of \eqref{eq:14-11-8}} \\
    & = \frac{\mult_{Δ_{\what{X}}} γ^*_X φ^* Δ_Y}{m_Y} && \text{by \ref{il:14-11-2}} \\
    & = \frac{\bigl(\mult_{Δ_{\what{X}}} γ^*_X Δ_X\bigr)·\bigl(\mult_{Δ_X} φ^* Δ_Y\bigr)}{m_Y} && \text{functoriality}
  \end{align*}
  Inserting this into \eqref{eq:14-11-10}, we obtain Inequality~\eqref{eq:14-11-1}, as
  required to end the proof of Proposition~\ref{prop:14-8}.
\end{proof}

\begin{proof}[Proof of Proposition~\ref*{prop:14-9}]
  \CounterStep{}Again we prove Proposition~\ref{prop:14-9} under the simplifying
  assumption that $⌊D_X⌋ = 0$ and $⌊D_Y⌋ = 0$.  The proof of the general case is
  conceptually identical but notationally more involved.
  
  Using Observation~\ref{obs:14-7} to remove a suitable small subset from $X$
  and consider a suitable open cover of $X$ and $Y$, it suffices to prove
  Proposition~\ref{prop:14-9} under the simplifying assumption that there exist
  local coordinates $x_• ∈ 𝒪_X(X)$ and $y_• ∈ 𝒪_Y(Y)$ such that the following
  holds.
  \begin{enumerate}
    \item\label{il:14-12-1} The varieties $X$ and $Y$ are simply connected
      subsets of $ℂ^•$.

    \item\label{il:14-12-2} There exist numbers $n_i ∈ ℕ$ such that $D_Y = \sum
      \textstyle{\frac{n_i-1}{n_i}}·\{y_i = 0\}$.

    \item\label{il:14-12-3} There exists a number $n ∈ ℕ$ such that $D_X =
      \textstyle{\frac{n-1}{n}}·\{x_1 = 0\}$.

    \item\label{il:14-12-4} We have $\supp φ^*D_Y ⊂ \{x_1 = 0\}$.
  \end{enumerate}
  Item~\ref{il:14-12-4} allows writing the morphism $φ$ in coordinates as
  \[
    φ: (x_1,\: x_2,\: x_3,\: …) ↦ \bigl(…,\: \underbrace{x_1^{a_i}·f_i(x_1,\: x_2,\: …)}_{i\text{.th position}},\: … \bigr), \quad\text{where all } f_{•} ∈ 𝒪^*_X(X).
  \]
  Assumption~\ref{il:14-12-1} allows choosing roots $g_• := \sqrt[n_•]{f_•} ∈
  𝒪^*_X(X)$.  Setting $N := n·\prod n_•$, we can then find smooth varieties
  $\what{X}$, $\what{Y}$ with coordinates $\what{x}_• ∈ 𝒪_{\what{X}}(\what{X})$
  and $\what{y}_• ∈ 𝒪_{\what{Y}}(\what{Y})$ and a commutative diagram,
  \[
    \begin{tikzcd}[column sep=3.5cm, row sep=large]
      \what{X} \ar[r, "\what{φ}"] \ar[d, two heads, "a\text{, adapted cover}"'] & \what{Y} \ar[d, two heads, "b\text{, strongly adapted cover}"] \\
      X \ar[r, "φ"'] & Y,
    \end{tikzcd}
  \]
  where
  \begin{align*}
    a : \bigl(\what{x}_1,\: \what{x}_2,\: \what{x}_3,\: …\bigr) & ↦ \bigl(\what{x}_1^{\,N}\,\: \what{x}_2,\: \what{x}_3,\: …\bigr) \\
    b : \bigl(\what{y}_1,\: \what{y}_2,\: \what{y}_3,\: …\bigr) & ↦ \bigl(\what{y}_1^{\,n_1},\: \what{y}_2^{\,n_2},\: \what{y}_3^{\,n_3},\: …\bigr) \\
    \what{φ} : (\what{x}_1,\: \what{x}_2,\: \what{x}_3,\: …) & ↦ \Bigl(…,\: \underbrace{\what{x}_i^{\frac{N·a_i}{n_i}}·g_i\bigl(\what{x}_1^{\,N},\: \what{x}_2,\: \what{x}_3,\: …\bigr)}_{i\text{.th position}},\: …\Bigr),
  \end{align*}
  so that
  \[
    Ω^{[1]}_{(X,D_X,a)} = \bigl\langle \what{x}_1^{\frac{N}{n}-1}·d\what{x}_1,\: d\what{x}_2,\: …\bigr\rangle ⊆ Ω¹_{\what{X}}
    \quad\text{and}\quad
    Ω^{[1]}_{(Y,D_Y,b)} = Ω¹_{\what{Y}}.
  \]
  With this description of the orbifold cotangent bundles, the following
  statements are clearly equivalent.
  \begin{equation}\label{eq:14-12-5}
    \begin{aligned}
      & \hphantom{⇔ \forall} \text{$\what{φ}$ admits pull-back of adapted reflexive 1-differentials} \\
      & ⇔ \forall i: d \Bigl(\what{x}_i^{\frac{N·a_i}{n_i}}·g_i\bigl(\what{x}_1^{\,N},\: \what{x}_2,\: \what{x}_3,\: …\bigr)\Bigr) ∈ Ω^{[1]}_{(X,D_X,a)}\bigl(\what{X}\bigr) \\
      & ⇔ \forall i: \frac{N·a_i}{n_i}-1 ≥ \frac{N}{n}-1 \\
      & ⇔ \forall i: a_i·n ≥ n_i \\
      & ⇔ \forall i: \bigl(\mult_{\{x_1=0\}} φ^* \{y_i=0\}\bigr)·\bigl(\mult_{\cC,\{x_1=0\}} D_X\bigr) ≥ \mult_{\cC,\{y_i=0\}} D_Y \\
      & ⇔ \text{$φ$ is an orbifold morphism}
    \end{aligned}
  \end{equation}
  To finish the proof of Proposition~\ref{prop:14-9}, recall that $Ω^{[1]}_{(Y,
  D_Y, b)} = Ω¹_{\what{Y}}$ is locally free.  Proposition~\ref{prop:9-3}
  therefore applies: to show that $φ$ is a $\cC$-morphism, it suffices to show
  $\what{φ}$ admits pull-back of adapted reflexive 1-differentials.  That,
  however, follows from the equivalences \eqref{eq:14-12-5}.
\end{proof}

\subsection{Orbifold morphisms in the sense of Lu}
\approvals{Erwan & yes \\ Stefan & yes}

Morphisms of smooth $\cC$-pairs have been considered by Lu in the influential
preprint \cite{Lu02}.  At first glance, Lu's definition \cite[Def.~3.7]{Lu02}
differs conceptually from Definition~\ref{def:8-1}: Given snc $\cC$-pairs
$(X,D_X)$, $(Y,D_Y)$ and a morphism $φ : X → Y$, Lu's definition of ``orbifold
morphism''\index{orbifold morphism!in the sense of Lu} considers open sets $U ⊆
Y$, commutative diagrams of the form
\[
  \begin{tikzcd}[column sep=2cm]
    φ^{-1}(U) \ar[r, "φ|_{φ^{-1}(U)}"] \ar[rd, bend right=5, "a\text{, holomorphic}"'] & U \ar[d, dashed, "b\text{, meromorphic}"] \\
    & Z
  \end{tikzcd}
\]
and compares the saturation of $a^* ω_Z$ with the $φ$-pull back of the
saturation of $b^* ω_Z$.  However, later in his paper Lu characterizes
``orbifold morphisms'' in terms that are close to the ideas pursued here.  We
reformulate his characterization in the language of the present paper.

\begin{prop}[\protect{Characterization of orbifold morphisms in the sense of Lu, \cite[Prop.~4.4]{Lu02}}]\label{prop:14-13}%
  Let $(X, D_X)$ and $(Y, D_Y)$ be snc $\cC$-pairs where $X$ and $Y$ are
  projective.  If $φ: X → Y$ is any morphism with $\supp φ^{-1} D_Y ⊆ \supp
  D_X$, then there exists a commutative diagram of the following form,
  \begin{equation*}
    \begin{tikzcd}[column sep=large, row sep=1cm]
      \what{X} \ar[r, "\what{φ}"] \ar[d, two heads, "a\text{, adapted cover}"'] & \what{Y} \ar[d, two heads, "b\text{, adapted cover}"] \\
      X \ar[r, "φ"'] & Y,
    \end{tikzcd}
  \end{equation*}
  where $\what{X}$ and $\what{Y}$ are smooth.  Then, $φ$ is an orbifold morphism
  if and only if the pull-back morphism
  \[
    dφ: \what{φ}^* \Bigl( Ω¹_{(Y, D_Y, b)} \Bigr) → Ω¹_X (\log a^{-1} ⌊D_X⌋)
  \]
  factors through the inclusion $Ω¹_{(X, D_X, a)} ↪ Ω¹_X (\log a^{-1} ⌊D_X⌋)$.
  \qed
\end{prop}

Using that all spaces and pairs in Lu's setting are smooth, the criterion for
$\cC$-morphisms spelled out in Proposition~\vref{prop:9-3} shows that orbifold
morphisms in the sense of Lu are morphisms of $\cC$-pairs indeed.

\subsection{$B$-birational morphisms in the sense of Fujino}
\approvals{Erwan & yes \\ Stefan & yes}

For his proof of the abundance theorem for semilog canonical threefolds, Fujino
introduced ``$B$-birational morphisms'' of pairs.  Given the potential for
confusion between the various notions of ``morphisms of pairs'', we briefly
recall the definition even though Fujino's notion is essentially unrelated to
the $\cC$-morphisms discussed here.

\begin{defn}[\protect{$B$-birational morphisms in the sense of Fujino, \cite[Def.~1.5]{MR1756108}}]\label{def:14-14}%
  Let $(X, D_X)$ and $(Y, D_Y)$ be two $ℚ$-Gorenstein pairs where $X$ and $Y$
  are projective.  A birational map $φ: X \dasharrow Y$ is called
  ``$B$-birational''\index{B-birational map@$B$-birational map} if there exists
  a common resolution of singularities
  \[
    \begin{tikzcd}
      Z \ar[d, "α\text{, resolution}"'] \ar[r, equal] & Z \ar[d, "β\text{, resolution}"] \\
      X \arrow[r, dashed, "f"'] & Y
    \end{tikzcd}
  \]
  such that the following equality of $ℚ$-divisors on $Z$,
  \[
    α^*\bigl(α_*K_Z+ D_X\bigr) = β^*\bigl(β_*K_Z+D_Y\bigr),
  \]
  holds for one (equivalently: every) canonical divisor $K_Z ∈ \Div(Z)$.
\end{defn}

It is easily seen that $B$-birational maps are hardly ever morphisms of
$\cC$-pairs.

\begin{example}[$B$-birational morphism, not a $\cC$-morphism]\label{ex:14-15}%
  Consider the space $Y = ℙ²$ and let $D_Y$ be a line passing through a point $y
  ∈ Y$.  Let $φ : X → Y$ be the blow-up of $y ∈ Y$ and set $D_X := φ^{-1}_*
  D_Y$.  Choose $Z := X$, $α := \Id_X$, and $β := φ$.  Fixing one canonical
  divisor $K_Y ∈ \Div(Y)$, an elementary computation shows that
  \[
    α^*(α_* K_Z + D_X) = K_X + D_X = φ^* (K_Y + D_Y) = β^*\bigl(β_* K_Z + D_Y\bigr).
  \]
  It follows that $φ$ is a $B$-birational morphism.  However, $φ$ is not a
  morphism of $\cC$-pairs.
\end{example}

% !TEX root = orbiAlb1
%
% Do not edit the following line.  The text is automatically updated by
% subversion.
%
\svnid{$Id: 15-openQuestions.tex 849 2024-07-15 07:39:06Z kebekus $}
\selectlanguage{british}

\section{Problems and open questions}
\subversionInfo
\approvals{Erwan & yes \\ Stefan & yes}
\label{sec:15}

As pointed out in the introduction, this paper is the first in a series
\cite{orbialb2, orbialb3} that develops the theory of $\cC$-pairs with a view
towards hyperbolicity questions, entire curves and rational points.  Still, we
see many other interesting directions of research and feel that $\cC$-pairs and
their adapted reflexive tensors are far from understood.  We close with a few of
questions and problems that we cannot answer at present.

\subsection{Adapted reflexive tensors}
\approvals{Erwan & yes \\ Stefan & yes}
\label{sec:15-1}

We feel that the pull-back results presented in Section~\ref{sec:5} might not be
optimal.  A ``bigger picture'' is still missing.

\begin{question}[Pull-back for $\cC$-pairs with mild singularities]\label{q:15-1}%
  Consider a setup analogous to Setting~\ref{set:5-2}: Let $(X, D_X)$ be a
  $\cC$-pair, let $(Y, D_Y)$ be a nc log pair, and consider a sequence of
  morphisms
  \[
    \begin{tikzcd}[column sep=2.8cm]
      Y \arrow[r, two heads, "φ\text{, resolution of sings.}"] & \what{X} \arrow[r, "γ\text{, $q$-morphism}"] & X,
    \end{tikzcd}
  \]
  where $\supp φ^* γ^* ⌊D_X⌋ ⊆ \supp D_Y$.  We ask for conditions to guarantee
  that a natural pull-back morphisms for adapted reflexive differentials,
  \[
    d_{\cC} φ : φ^* Ω^{[p]}_{(X,D_X,γ)} → Ω^p_Y(\log D_Y)
  \]
  exists for some or all values of $p$, with universal properties similar to
  those discussed in Section~\ref{sec:5-5}?
  \begin{enumerate}
    \item\label{il:15-1-1} Do pull-back maps exist if $(X,D_X)$ is klt?

    \item\label{il:15-1-2} What can we say if $(X,D_X)$ is log canonical?
  
    \item In analogy to the results obtained in \cite{KS18}, is there a natural
    class of pairs (``pairs with $\cC$-rational singularities'') that behave
    optimally with respect to pull-back?
  \end{enumerate}
\end{question}

\begin{rem}[Partial results]
  In case where $\what{X} = X$ and $γ = \Id_X$, the papers \cite{GKKP11,
  MR3084424, KS18} answer Questions~\ref{il:15-1-1} and \ref{il:15-1-2} in the
  positive.  These results are however insufficient to establish a meaningful
  theory for $\cC$-pairs.
\end{rem}

\begin{question}[Pull-back for forms of small degree]
  Maintaining the setup of Question~\ref{q:15-1}, we expect that adapted
  reflexive $p$-forms become easier to pull-back, the smaller the value of $p$.
  \begin{enumerate}
    \item Are there results for particularly small values of $p$ that can be
    seen as $\cC$-analogues of the earlier results \cite{SS85, Flenner88}?

    \item Are there results of the form ``the pull-back behaviour of adapted
    reflexive $p$-forms follows the extension behaviour $(p+1)$-forms'' that
    could be seen as analogues of \cite[Thm.~1.4]{KS18}?
  \end{enumerate}
\end{question}

\begin{rem}[Partial results]
  In the special case that $p =1$, Pedro Núñez has shown in his Ph.D.~thesis
  \cite{Nunez23} that a natural morphism as in \ref{il:15-1-1} exists.  It is
  however unclear at present, if it satisfies enough universal properties to be
  useful in real-world applications.
\end{rem}

\subsection{Invariants of $\cC$-pairs}
\approvals{Erwan & yes \\ Stefan & yes}
\label{sec:15-2}

As pointed out in Section~\ref{sec:6-1}, the irregularity is of fundamental
importance when we discuss $\cC$-analogues of the Albanese in the follow-up
paper \cite{orbialb2}.  Still, there are aspects that we do not fully
understand.

\begin{question}[Irregularities]\label{q:15-5}
  Let $(X,D)$ be a $\cC$-pair where $X$ is compact Kähler.  How the
  irregularities $q(X,D,γ)$ depend on the choice of the cover $γ: \what{X}
  \twoheadrightarrow X$?  Can we say anything about the relation between
  $q(X,D,•)$ and the local geometry of the covering spaces?  If $X$ has rational
  singularities, is it possible that $q(X,D,γ) = 0$ for all covers $γ : \what{X}
  \twoheadrightarrow X$ where $\what{X}$ has rational singularities, and becomes
  large only for covers that are more singular?
\end{question}

Concerning Question~\ref{q:15-5}, recall from \cite[Prop.~5.13]{KM98} that
rational singularities cannot cover non-rational ones!

\begin{question}[Bogomolov-Sommese vanishing on covers of $X$]\label{q:15-6}
  We do not expect that Proposition~\ref{prop:6-15} is optimal.  Are there
  better results for pairs with mild (but worse than uniformizable)
  singularities?
\end{question}

Note that any answer to Question~\ref{q:15-1} will also answer \ref{q:15-6}.

\begin{problem}[Special pairs, topology and arithemtics]
  Establish arithmetic, topological and geometric properties of mildly singular
  $\cC$-pairs that are special.
\end{problem}

\subsection{Morphisms of $\cC$-pairs}
\approvals{Erwan & yes \\ Stefan & yes}
\label{sec:15-3}

We have seen in Section~\ref{sec:13-2} that $\cC$-resolutions of singularities
will typically not exist.  Still, we wonder if they do exist for mildly singular
pairs.

\begin{question}[$\cC$-resolution of singularities, existence]
  Is there a $\cC$-resolution for pairs with mild singularities better than
  Proposition~\ref{prop:10-10}?  Is there a converse to
  Corollary~\ref{cor:13-5}?
\end{question}

\begin{question}[$\cC$-resolution of singularities, properties]\label{q:15-9}%
  If $(X,D)$ is a $\cC$ and if a $\cC$-resolution $(\wtilde{X}, \wtilde{D}) →
  (X,D)$ exists, then what is the precise relation between the coefficients of
  $\wtilde{D}$ and the discrepancies of $(X,D)$?  Can the coefficients be
  expressed in terms of the local fundamental group?
\end{question}

\subsection{Birational geometry of $\cC$-pairs}
\approvals{Erwan & yes \\ Stefan & yes}
\label{sec:15-4}

Campana has studied meromorphic maps, meromorphic fibrations and bimeromorphic
maps of nc $\cC$-pairs extensively.  We feel that this part of the theory should
be extended to singular pairs, in order to tie it up with minimal model theory.

\begin{problem}[$\cC$-bimeromorphic maps]\label{p:15-10}
  Develop a theory of $\cC$-meromorphic and $\cC$-bimeromorphic maps.  For
  mildly singular pairs, prove that relevant invariants are bimeromorphically
  invariant.  Following Campana, characterize special pairs in terms of
  bimeromorphic fibrations to $\cC$-pairs of general type.
\end{problem}

\begin{problem}[Special pairs]
  In line with Problem~\ref{p:15-10}, follow Campana and characterize mildly
  singular special $\cC$-pairs in terms of $\cC$-bimeromorphic fibrations to
  $\cC$-pairs of general type.
\end{problem}

\begin{problem}[Core map]
  In line with Problem~\ref{p:15-10}, follow Campana and establish a core map for
  mildly singular special $\cC$-pairs.  Study its properties.
\end{problem}

\subsection{$\cC$-pairs in other settings}
\approvals{Erwan & yes \\ Stefan & yes}

$\cC$-pairs have been applied successfully in algebraic and arithmetic settings.
Still, we feel that a systematic treatment of $\cC$-morphisms is lacking in many
contexts.

\begin{problem}[$\cC$-pairs in the algebraic setting]
  Develop a viable theory of $\cC$-pairs for algebraic varieties over
  algebraically closed (or perhaps: perfect) fields of arbitrary characteristic.
\end{problem}

\begin{problem}[$\cC$-pairs in the arithmetic setting]
  Develop a viable theory of $\cC$-pairs for algebraic varieties over global
  fields, perhaps following ideas of \cite{KPS22}.
\end{problem}

% !TEX root = orbiAlb1

\appendix
\addcontentsline{toc}{part}{Appendix}
\printindex
\clearpage

\listoffigures
\listoftables

\end{document}